\documentclass[11pt,letterpaper]{amsart}

\usepackage[OT2,T1]{fontenc}
\DeclareSymbolFont{cyrletters}{OT2}{wncyr}{m}{n}
\DeclareMathSymbol{\Sha}{\mathalpha}{cyrletters}{"58}

\usepackage{indentfirst}
\usepackage{amstext}
\usepackage{amsopn}
\usepackage{amsfonts}
\usepackage{latexsym}
\usepackage{amscd}
\usepackage{amssymb}
\usepackage{amsmath}
\usepackage[all,cmtip]{xy}
\usepackage{leftidx}
\usepackage{graphicx}
\usepackage{tikz}
\usepackage{ulem}

\usepackage{hyperref}
\let\goth\mathfrak
\def\cA{\mathcal A}
\def\cB{\mathcal B}

\def\cO{\mathcal O}

\def\cL{\mathcal L}

\def\cE{\mathcal E}
\def\cF{\mathcal F}

\def\cP{\mathcal P}

\def\gm{\goth m}

\def\GG{\mathbb{G}}

\def\WW{\mathbf{W}}
\def\VV{\mathbf{V}}
\def\bV{\mathbb{V}}

\def\gG{\goth G}
\def\gP{\goth P}
\def\gH{\goth H}

\def\gN{\goth N}
\def\gS{\goth S}
\def\gT{\goth T}
\def\gV{\goth V}
\def\gY{\goth Y}
\def\gX{\goth X}

\def\gZ{\goth Z}

\def\1{\mbox{\bf 1}}

\def\rad{\mathrm{rad}}

\def\r0{{A^1}}

 \DeclareMathOperator{\Hom}{Hom}
\DeclareMathOperator{\Aut}{Aut}

\DeclareMathOperator{\Out}{Out}

\DeclareMathOperator{\Spin}{\rm Spin}

\DeclareMathOperator{\PGL}{\rm PGL}
\DeclareMathOperator{\GL}{\rm GL}
\DeclareMathOperator{\SL}{\rm SL}
\DeclareMathOperator{\SK}{\rm SK}

\newcommand{\incl}[1][r]
{\ar@<-0.2pc>@{^(-}[#1] \ar@<+0.2pc>@{-}[#1]}

\newtheorem{stheorem}{Theorem}[section]

\newtheorem{scorollary}[stheorem]{Corollary}
\newtheorem{slemma}[stheorem]{Lemma}
\newtheorem{sproposition}[stheorem]{Proposition}
\newtheorem{sremark}[stheorem]{Remark}
\newtheorem{sremarks}[stheorem]{Remarks}
\newtheorem{sexample}[stheorem]{Example}
\newtheorem{sexamples}[stheorem]{Examples}
\newtheorem{sdefinition}[stheorem]{Definition}
\newtheorem{squestion}[stheorem]{Question}

\theoremstyle{definition}

\numberwithin{equation}{section}

\newcounter{nc}
\renewcommand{\thenc}{{\rm(\roman{nc})}}
\newenvironment{romlist}%
{\begin{list}{\thenc}{
\usecounter{nc}
\parsep=0pt
\setlength  \labelwidth{\leftmargin}
\addtolength\labelwidth{-\labelsep}
}
}{\end{list}}
\newcounter{nnc}
\renewcommand{\thennc}{{\rm(\alph{nnc})}}
{\begin{list}{\thennc}{
\usecounter{nnc}
\parsep=0pt
\setlength  \labelwidth{\leftmargin}
\addtolength\labelwidth{-\labelsep}
}
}{\end{list}}
\newcommand{\pauseromlist}%
{\global\edef\savecount{\arabic{nc}}\end{romlist}}
\newcommand{\finpauseromlist}%
{\begin{romlist}\setcounter{nc}{\savecount}}

\newcounter{ctnum}
\renewcommand{\thectnum}{\textup{(\arabic{ctnum})}}
{\begin{list}{\thectnum}{
\usecounter{ctnum}
\parsep=0pt
\leftmargin=0pt%
\setlength{\itemindent}{\labelwidth}%
\addtolength{\itemindent}{\labelsep}%
}
}{\end{list}}
\def\ZZ{\mathbb{Z}}

\def\gE{\mathfrak{E}}

\def\gG{\mathfrak{G}}

\def\gP{\mathfrak{P}}
\def\gQ{\mathfrak{Q}}
\def\gL{\mathfrak{L}}
\def\gU{\mathfrak{U}}

\def\QQ{\mathbb{Q}}

\def\A{\mathbb A}

\def\bG{\text{\rm \bf G}}

\def\bW{\text{\rm \bf W}}

\def\cL{\mathcal{L}}
\def\cO{\mathcal{O}}

\def\ol{\overline}

\def\et{\text{\rm \'et}}

\def\Lie{\mathop{\rm Lie}\nolimits}

\def\2int{\mathop{2\int}\nolimits}

\def\dim{\mathop{\rm dim}\nolimits}

\def\Spec{\mathop{\rm Spec}\nolimits}
\def\Specmax{\mathop{\rm max}\nolimits}

\def\Lie{\mathop{\rm Lie}\nolimits}

\def\Hom{\mathop{\rm Hom}\nolimits}

\def\Pic{\mathop{\rm Pic}\nolimits}

\def\Aut{\text{\rm{Aut}}}
\def\Out{\text{\rm{Out}}}

\def\resp.{\mathop{\rm resp.}\nolimits}
\def\limproj{\mathop{\oalign{lim\cr
\hidewidth$\longleftarrow$\hidewidth\cr}}}
\def\limind{\mathop{\oalign{lim\cr
\hidewidth$\longrightarrow$\hidewidth\cr}}}
\def\Ker{\mathop{\rm Ker}\nolimits}

\def\lgr{\longrightarrow}
\def\la{\longleftarrow}

\font\math=cmmi10
\def\varpi{\hbox{\math\char'44}}

\def\simlgr{\buildrel\sim\over\lgr}
\def\simla{\buildrel\sim\over\la}

\def\pa{\S\kern.15em }

\def\un{\uppercase\expandafter{\romannumeral 1}}
\def\deux{\uppercase\expandafter{\romannumeral 2}}
\def\trois{\uppercase\expandafter{\romannumeral 3}}
\def\quatre{\uppercase\expandafter{\romannumeral 4}}
\def\cinq{\uppercase\expandafter{\romannumeral 5}}
\def\six{\uppercase\expandafter{\romannumeral 6}}

\def\et{\acute et}

\def\hfl#1#2#3{\smash{\mathop{\hbox to#3{\rightarrowfill}}\limits
^{\scriptstyle#1}_{\scriptstyle#2}}}
\def\gfl#1#2#3{\smash{\mathop{\hbox to#3{\leftarrowfill}}\limits
^{\scriptstyle#1}_{\scriptstyle#2}}}

\title[$R$-equivalence on group schemes]{
$R$-equivalence on reductive group schemes}

\author{Philippe Gille}\address{UMR 5208
Institut Camille Jordan - Universit\'e Claude Bernard Lyon 1
43 boulevard du 11 novembre 1918
69622 Villeurbanne cedex - France
}

\email{gille@math.univ-lyon1.fr}

\author{Anastasia Stavrova}\address{Chebyshev Laboratory, Department of Mathematics and Computer Science, St.
Petersburg State University, 14th Line V.O. 29B, 199178, Saint Petersburg, Russia, and
St. Petersburg Department of Steklov Mathematical Institute, nab. r. Fontanki 27, 191023, Saint Petersburg, Russia
}

\thanks{The second author is supported
by the Russian Science Foundation grant 19-71-30002.}

\email{anastasia.stavrova@gmail.com}

\date{\today}

\begin{document}

 \begin{abstract}
Let $A$ be an equicharacteristic henselian regular local ring. Let $k$ and $K$
be the residue field and the fraction field of $A$. We show that
for any reductive group scheme $\gG$ over $A$
there is a canonical isomorphism of Manin's $R$-equivalence class groups $\gG_K(K)/R\cong\gG_k(k)/R$.
Our proof is based on extending the notion of $R$-equivalence from algebraic varieties over fields to schemes
over commutative rings, and showing that the two canonical homomorphisms $\gG(A)/R\to \gG_k(k)/R$ and $\gG(A)/R\to \gG_K(K)/R$
are isomorphisms. If $\gG$ is a torus or an isotropic simply connected semisimple group, the first isomorphism
in fact holds without the assumption that $A$ is regular, and the second one without the assumption that $A$ is
henselian. As a consequence, if $X$ is a connected smooth scheme over a field $k$, and $\gG$ is a
reductive $X$-group scheme belonging to one
of the two classes mentioned above, then $\gG$ being retract rational at the generic point of $X$ implies
that all fibers $\gG_x$, $x\in X$, are retract rational.

\smallskip

\noindent {\em Keywords:} reductive group scheme, algebraic torus, $R$-equivalence, $A^1$-equivalence,
Whitehead group, non-stable $K_1$-functor.\\

\noindent {\em MSC 2020: 20G15, 20G35, 19B99}
\end{abstract}

\maketitle

{\small \tableofcontents }

\bigskip


\section{Introduction} \label{sect_intro}


Yu. Manin~\cite[\S 14]{Manin} introduced the notion of $R$-equivalence for points of algebraic varieties over a field.
This notion has been used extensively in the study of reductive algebraic groups, e.g.~\cite{CTS1,CTS2,Gi2,ACP}.
In the present paper, we propose
a generalized definition of $R$-equivalence that is applicable to arbitrary schemes over an affine base
and allows to extend several of the
above-mentioned results to reductive group schemes in the sense of~\cite{SGA3}.


Among reductive groups, two classes  play a fundamental
 role, the tori and the semisimple simply connected isotropic groups. In these two cases the $R$-equivalence
class group $G(k)/R$ of a reductive group $G$ over a field $k$ is already known to coincide
with  the value of a certain
functor defined on the category of all commutative $k$-algebras, and even on all commutative rings $B$
such that $G$ is defined over $B$.

Namely, if $G=T$ is a $k$-torus and
\begin{equation}
1\to F\to P\to T\to 1
\end{equation}
is a flasque resolution of $T$, then $T(k)/R$ coincides with the first Galois (or \'etale)
cohomology group $H^1_{\et}(k,F)$~\cite{CTS1}, and $H^1_{\et}(-,F)$ is the functor of the above kind.

If $G$ is a simply connected absolutely
almost simple $k$-group having a proper parabolic $k$-subgroup, then $G(k)/R$ coincides with the Whitehead group of $G$, which is the
subject of the Kneser--Tits problem, and with the group of $\mathbf{A}^1$-equivalence classes of $k$-points. Recall that the Whitehead group of $G$ over $k$ is defined
as the quotient of $G(k)$ by the subgroup
generated by the $k$-points of the unipotent radicals of all proper parabolic $k$-subgroups of $G$.
In the setting of reductive groups over rings, the Whitehead group is also called a non-stable $K_1$-functor,
which is defined as follows.

Let $B$ be a ring. If $\gG$ is a reductive $B$--group scheme equipped with a parabolic
$B$--subgroup $\gP$ of unipotent radical $\rad_u(\gP)$,  we define the elementary subgroup
$E_\gP(B)$ of $\gG(B)$ to be the subgroup generated by $\rad_u(\gP)$ and
$\rad_u(\gP^{-})$ where $\gP^{-}$ is an opposite $B$--parabolic to $\gP$ (it is independent of that choice,
see \cite[\S 1]{PS}).
We define the non stable $K_1$-functor
$K_1^{\gG, \gP}(B)= \gG(B)/E_\gP(B)$
called also the Whitehead coset.
We say that $\gG$ has $B$-rank $\ge n$,
if every normal semisimple $B$-subgroup of $\gG$ contains $(\GG_{m,B})^n$. If $B$ is semilocal and $\gP$ is minimal,
or if the $B$-rank of $\gG$ is $\ge 2$ and $\gP$ is strictly proper
(i.e. $\gP$ intersects properly every semisimple normal subgroup of $G$), then $E_\gP(B)$ is a normal subgroup independent of the specific choice of
$\gP$~\cite[Exp. XXVI]{SGA3},~\cite{PS}
and the group $K_1^{\gG, \gP}(B)$
is  denoted often by $K_1^{\gG}(B)$ in that case.

A related, more universal construction is the 1st Karoubi-Villamayor
$K$-functor, or the group of $\mathbf{A}^1$-equivalence classes,  denoted here by $\gG(B)/\r0$
where  $\r0\gG(B)$ consists in the (normal) subgroup of $\gG(B)$
generated by the elements $g(0) g^{-1}(1)$ for $g$ running over $\gG(B[t])$.

Coming back to the field case, if $G$ is a semisimple simply connected over a field $k$ and equipped with a strictly proper parabolic
$k$--subgroup $P$, the preceding paragraph
defines the groups  $K_1^{G, P}(k)$ and $G(k)/\r0$ and
we know that the  natural maps
$$
K_1^{G, P}(k) \to G(k)/\r0 \to G(k)/R
$$
are bijective~\cite{Gi2}.
In the present paper, we investigate to which extent
such a result holds over the ring $B$, especially in the semilocal case and in the regular case.

Our first task is the extension  of  the notion of $R$-equivalence for rational points
of algebraic varieties to integral points of a $B$-scheme in such a way that it is functorial with respect
to ring homomorphisms.
 This is the matter of section \ref{sec_R_eq};  an advantage of $R$-equivalence
 is  the nice functoriality with respect to fibrations.

In the subsequent sections we study the properties of $R$-equivalence on reductive group schemes.
 For tori over regular rings, the  Colliot-Th\'el\`ene and Sansuc computation of
 $R$--equivalence extend verbatim to the ring setting, see \S \ref{subsec_tori}.

For non-toral reductive groups we obtain several results under the assumption that $B$
is an equicharacteristic semilocal regular domain.
Namely, we show in Theorem~\ref{thm_main} that
for a   semisimple simply connected $B$--group $\gG$ of $B$-rank  $\geq 2$,
the maps
 $$
K_1^{\gG, \gP}(B) \to \gG(B)/\r0 \to \gG(B)/R
$$ are isomorphisms;
if the $B$-rank of $\gG$ is only $\ge 1$, then the second map is an isomorphism (Theorem~\ref{thm_KV_R}).
 In particular, this provides several new cases where
 $E_\gP(B)=\gG(B)$ holds (Cor.\ \ref{cor_main}).

Let $K$ be the fraction field of $B$. Another  main result is the surjectivity of the map
$\gG(B)/R \to \gG(K)/R$, assuming either that $\gG$ is a reductive group of $B$-rank $\ge 1$,
or that $\gG$ has no parabolic subgroups over the residue fields of $B$ (Theorem \ref{thm:surj}).
If $\gG$ is simply connected semisimple of $B$-rank $\ge 1$, then this map is an isomorphism
(Theorem~\ref{thm_KV_R}).
This statement was previously known
(for $\gG(B)/\r0$ instead of $\gG(B)/R$) in the case where $\gG$ is defined over an infinite perfect subfield of
$B$ and is of classical type, see~\cite[Corollary 4.3.6]{AHW} and~\cite[Example 2.3]{Mo-book}.

As a corollary,  we conclude that if $\gG$ is a $B$-torus (necessarily isotrivial \cite[X.5.16]{SGA3}) or a simply connected semisimple $B$-group of $B$-rank $\ge 2$, then $\gG(C)/R=1$ for each semilocal $B$--ring $C$
 if and only if $\gG_K$ is retract rational over $K$
(Proposition~\ref{prop_retract_torus} and Theorem~\ref{thm_vanish}).
In particular, if $X$ is a connected smooth scheme over a field $k$, and $\gG$ is a
reductive $X$-group scheme belonging to one
of the two classes mentioned above, then $\gG$ being retract rational at the generic point of $X$ implies
that all fibers $\gG_x$, $x\in X$, are retract rational. This is reminiscent of the recent results on the rationality
of fibers of smooth proper schemes over smooth curves~\cite{KoTsch,NiSh}.

The assumption that $B$ is equicharacteristic arises from the fact that we
use a geometric construction developped by I. Panin for the proof of the Serre--Grothendieck
conjecture for equicharacteristic semilocal regular rings~\cite[Theorem 2.5]{Pa} (see also~\cite{PaStV,FP}).
Recently, K. \v{C}esnavi\v{c}ius partially generalized this construction
to semilocal regular rings which are essentially smooth
over a discrete valuation ring and proved the Serre--Grothendieck conjecture for quasi-split
reductive groups over such rings in the unramified case~\cite{Ces-GS}.
We expect that in the future this approach will yield a similar extension of our results.


Another motivation for the present work was to address the specialization problem for
$R$-equivalence~\cite[6.1]{CT},~\cite{Ko},~\cite{Gi1}.
Let $A$ be a  henselian  local domain  of residue field $k$
and fraction field $K$.
Let $\gG$ be a reductive $A$--group scheme and denote by $G= G \times_A k$
its closed fiber. In this setting,
the specialization problem asks whether there exists a natural  specialization
homomorphism $\gG(K) /R \to G(k)/R$ and a
lifting map $G(k)/R \to \gG(K) /R$.
It makes sense to approach these questions using the generalized $R$-equivalence for $\gG$, since
we may investigate whether the maps in the diagram
$$
\xymatrix{
G(k)/R &\ar[l]  \gG(A)/R  \ar[r]  & \gG(K)/R
}
$$
are injective/surjective/bijective. In general, the only apriori evidence
is the surjectivity of $\gG(A)/R \to G(k)/R$ which follows from the surjectivity
of $\gG(A) \to G(k)$ (Hensel's lemma). We prove that if $\gG$ is a torus or a simply connected semisimple group
scheme equipped with a strictly proper parabolic $A$-subgroup, then the map $\gG(A)/R \to G(k)/R$
is an isomorphism (Proposition~\ref{prop_torus2} and Theorem~\ref{thm:hens-isotr}).
In the case where $A$ is a  henselian  regular local ring containing a prime field $k_0$,
we prove  that for any reductive $\gG$ there are two isomorphisms
$$
    G(k)/R  \simla \gG(A)/R \simlgr G(K)/R
$$
and in particular there is a well-defined specialization (resp.\ lifting) homomorphism (Theorem~\ref{thm:hens-inj}).
Note that the recent results  on the local-global principles
over semi-global fields~\cite{CTHHKPS} crucially use the existence of an (independently constructed) specialization map
for two-dimensional rings.

\smallskip


\bigskip

\noindent{\bf Acknowledgments}. We thank K.~\v{C}esnavi\v{c}ius for
valuable comments in particular about fields of representatives
for henselian rings. We thank D.~Izquierdo for useful conversations. Finally we  are indebted to the referee for his constructive review.

\bigskip

\noindent{\bf Notations and conventions.}

We use mainly  the terminology and notation of Grothendieck-Dieudonn\'e
\cite[\S 9.4  and 9.6]{EGA1},  which agrees with that of Demazure-Grothendieck used in \cite[Exp. I.4]{SGA3}

Let $S$ be a scheme and let $\cE$ be a quasi-coherent sheaf over $S$.
 For each morphism  $f:T \to S$,
we denote by $\cE_{T}=f^*(\cE)$ the inverse image of $\cE$ by the morphism $f$.
 Recall that the $S$--scheme
$\VV(\cE)=\Spec\bigl( \mathrm{Sym}^\bullet(\cE)\bigr)$ is affine over $S$ and
represents the $S$--functor $T \mapsto \Hom_{\cO_T}(\cE_{T}, \cO_T)$
\cite[9.4.9]{EGA1}.

We assume now that $\cE$ is locally free of finite rank
and denote by $\cE^\vee$ its dual.
In this case the affine $S$--scheme $\VV(\cE)$ is  of finite presentation
(ibid, 9.4.11); also
the $S$--functor $T \mapsto H^0(T, \cE_{(T)})= \Hom_{\cO_T}(\cO_T, \cE_{T} )$
is representable by the  affine $S$--scheme $\VV(\cE^\vee)$
which is also denoted by  $\WW(\cE)$  \cite[I.4.6]{SGA3}.

For scheme morphisms $Y \to X \to S$, we denote by $\prod\limits_{X/S}(Y/X)$
the $S$--functor defined by $$
\Bigl( \prod\limits_{X/S}(Y/X) \Bigr)(T)= Y(X \times_S T)
$$
for each $S$--scheme $T$. Recall that if $\prod\limits_{X/S}(Y/X)$ is representable by an $S$-scheme, this scheme is called the Weil restriction of $Y$ to $S$.

If $\gG$ is a $S$--group scheme locally of finite presentation, we denote by $H^1(S, \gG)$
the set of isomorphism classes of sheaf $\gG$--torsors
 for the fppf topology.


\section{$R$-equivalence for schemes}\label{sec_R_eq}



\subsection{Definition} Let $B$ be a ring (unital, commutative).
We denote by $\Sigma$ the multiplicative subset
of polynomials $P \in B[T]$ satisfying $P(0), P(1) \in B^\times$.
Note that evaluation at $0$ (and $1$) extend from $B[t]$
to the localization $B[t]_\Sigma$.

Let $\cF$ be a $B$-functor in sets.
We say that two points $x_0,x_1 \in \cF(B)$ are \emph{directly $R$--equivalent} if
there exists
$x \in \cF\bigl( B[t]_\Sigma \bigr)$
such that  $x_0=x(0)$ and $x_1=x(1)$.
The $R$-equivalence on $\cF(B)$ is the equivalence relation
generated by this elementary relation.

\smallskip

\begin{sremarks}\label{rem_def}{\rm

(a) If $B$ is a field, then $B[t]_\Sigma$ is the semilocalization
of $B[t]$ at $0$ and $1$ so that the definition
agrees with the classical definition.

\smallskip

\noindent (b) If $B$ is a semilocal ring with maximal ideals
$\gm_1,\dots,  \gm_r$, then $B[t]_\Sigma$ is the semilocalization of
$B[t]$ at the maximal ideals $\gm_1 B[t]+ t B[t]$, $\gm_1 B[t]+ (t-1) B[t]$,
\dots, $\gm_r B[t]+ t B[t]$, $\gm_r B[t]+ (t-1) B[t]$. In particular
$B[t]_\Sigma$ is a semilocal ring.

\smallskip

\noindent (c) The most important case is for the $B$--functor of points $h_\gX$ of
 a $B$--scheme $\gX$. In this case we write $\gX(B) /R$ for $h_\gX(B)/R$.

 \smallskip

\noindent  (d) If the $B$--functor $\cF$ is locally of finite presentation
 (that is commutes with filtered direct limits), then two
points $x_0,x_1 \in \cF(B)$ are directly $R$--equivalent if
there exists a polynomial $P \in B[t]$ and
$x \in \cF\bigl( B[t, \frac{1}{P}] \bigr)$ such that $P(0)$,
$P(1) \in B^\times $ and
$x_0=x(0)$ and $x_1=x(1)$.
This applies in particular to the case of $h_\gX$ for a $B$--scheme $\gX$
locally of finite presentation.
}
\end{sremarks}

The important thing is the functoriality. If $ B \to C$ is a morphism of
rings, then the map $\cF(B) \to \cF(C)$ induces a map
$\cF(B)/R \to \cF(C)/R$.
We have also a product compatibility $(\gX \times_B \gY)(B)/R
\simlgr \gX(B)/R \times \gY(B)/R$ for $B$-schemes $\gX, \gY$.

If $\gG$ is a $B$--group scheme (and more generally
a $B$-functor in groups), then the $R$--equivalence is compatible with
left/right translations by $\gG(B)$, also the subset $R\gG(B)$ of elements of $\gG(B)$
which are $R$-equivalent to $1$ is a normal subgroup. It follows that the set
$\gG(B)/R \cong \gG(B)/R\gG(B)$ is equipped with a natural group structure.

\subsection{Elementary properties}

We start with the homotopy property.

\begin{slemma} \label{lem_homotopy} Let $\cF$ be a $B$--functor.

\smallskip

\noindent (1) The map $\cF(B)/R \to \cF(B[u])/R$ is bijective.

\smallskip

\noindent (2) Assume that $\cF$ is a $B$--functor in groups. Then two points of $\cF(B)$ which are
$R$--equivalent are directly $R$--equivalent.

\end{slemma}

\begin{proof} (1) The specialization at $0$ provides a splitting
of $B \to B[u]$, so that the map $\cF(B)/R \to \cF(B[u])/R$ is split injective.
It is then enough to establish the surjectivity.
Let $f \in \cF(B[u])$. We put $x(u,t)=f(ut) \in  \cF(B[u,t])$
so that $x(u,0) = f(0)_{B[u]}$ and $x(u,1) =f$.
In other words, $f$ is directly $R$-equivalent to $f(0)_{B[u]}$
and we conclude that the map is surjective.

\smallskip

\noindent (2) We put $\cB=B[t]$ and are given two elements $f, f' \in \cF(B)$ which are
$R$-equivalent. By induction on the length of the chain connecting $f$ and $f'$,
we can assume that there exists $f_1 \in  \cF(B)$ which is directly $R$-equivalent
to $f$ and $f'=f_2$. Also by translation we can assume that $f=1$.
There exists $g(t), h(t) \in \cF(\cB)$ such that $g(0)=1$,
$g(1)=f_1^{1}=h(0)$ and $h(1)= f_2$.
We put $f(t)= g(1-t)^{-1} \, h(t) \in \cF(\cB)$. Then
$f(0)=1$ and $f(1)= f_2$ as desired.
\end{proof}

\begin{slemma} \label{lem_limit} Let $\cF$ be a $B$--functor
locally of finite presentation and consider a direct limit
$B_ \infty= \limind_{\lambda \in \Lambda} B_\lambda$ of $B$--rings.
Then the map $\limind_{\lambda \in \Lambda}  \cF(B_\lambda)/R \to \cF(B_\infty)/R$ is bijective.
\end{slemma}

\begin{slemma} \label{lem_weil} Let $C$ be a locally free
$B$--algebra of degree $d$. Let  $\cE$ be a $C$--functor and consider the $B$--functor $\cF= \prod_{C/B}\cE$
defined by $\cF(B')= {\cE(C \otimes_B B')}$ for each $B$--algebra $B'$.
Then the morphism $\cF(B)/ R \to \cE(C)/R$
is an isomorphism.
\end{slemma}

\begin{proof}
We distinguish the multiplicative subsets $\Sigma_B$ and $\Sigma_C$. The map \break 
${B[t]_{\Sigma_B} \otimes_B C} \to C[t]_{\Sigma_C}$  induces  a map
$\cF(B[t]_\Sigma)= \cE(B[t]_{\Sigma_B} \otimes_B C) \to \cE(C[t]_{\Sigma_C})$.
We get then  a morphism $\cF(B)/ R \to \cE(C)/R$.
We claim that the map $B[t]_{\Sigma_B} \otimes_B C
 \to C[t]_{\Sigma_C}$ is an isomorphism
 which rephrases to prove that
the map $C[t]_{\Sigma_B} \to C[t]_{\Sigma_C}$
is an isomorphism in view of the isomorphism
 $B[t]_{\Sigma_B} \otimes_{B[t]} {C[t]} \simlgr C[t]_{\Sigma_B}$ \cite[Tag 00DK, 9.11.15]{St}.
 For establishing this fact,
 it is then enough to show that any element
 of $\Sigma_C$ divides an element of $\Sigma_B$
 in $C[t]$ \cite[II, \S 2.3, Prop.\ 8]{BAC}.

Since $C$ is locally free over $B$ of degree $d$, we
can consider the norm map $N: C \to B$ as defined in \cite[Tag 0BD2, 31.17.6]{St}; by definition
it is multiplicative and applies then units on units.
It is well-known that there exists a polynomial map $N': C \to B$ such that $N(c)= c \, N'(c)$ for each $c \in C$
(this follows from the Hamilton-Cayley's theorem).
Given $Q(t) \in C[t]$ such that $Q(1), Q(0) \in C^\times$,
we consider  $P(t)=N_{C/B}(Q(t)) \in B[t]$. We have $P(0)= N_{C/B}(Q(0)) \in B^\times$ and
similarly $P(1) \in B^\times$ so that $P(t)$
belongs to $\Sigma_B$. Since
 $Q(t)$ divides $P(t)$,
$Q(t)$ divides an element of $\Sigma_B$.
 Since $B[t]_\Sigma \otimes_{B[t]} {C[t]} \simlgr C[t]_{\Sigma_B}$ \cite[Tag 00DK, 9.11.15]{St},
 we conclude that the map $B[t]_{\Sigma_B} \otimes_B C \to C[t]_{\Sigma_C}$ is an isomorphism.
 As counterpart we get that the map $\cF(B)/ R \to \cE(C)/R$
is an isomorphism.
\end{proof}

\begin{slemma} \label{lem_sorite1} Let $\gX$ be a $B$-scheme.

\noindent (1)  Assume that $\gX=\Spec(B[\gX])$ is affine and let
$\gU=\gX_f$ be a principal open subset of $\gX$ where
$f \in B[\gX]$. If two points $x_0, x_1 \in \gU(B)$ are
directly $R$-equivalent in $\gX(B)$, then they are
directly $R$-equivalent in $\gU(B)$.

\smallskip

\noindent (2) Assume that $B$ is semilocal. Let $\gU$ be an open $B$--subscheme of
$\gX$. If two points $x_0, x_1 \in \gU(B)$ are
directly $R$-equivalent in $\gX(B)$, then they are
directly $R$-equivalent in $\gU(B)$.

\smallskip

\noindent (3)  Let $\gG$ be a $B$--group scheme and let $\gU$ be an open $B$--subscheme of
$\gG$. If $\gU$ is a principal  open subset or if $B$ is semilocal, then the map $\gU(B)/R \to \gG(B)/R$ is injective.
\end{slemma}

Note that (3) was known in the field case under an assumption
of unirationality \cite[Prop.\ 11]{CTS1}.

\begin{proof}
(1) Let $x_1,x_2 \in \gU(B)$ and let $x(t) \in \gX(B[t]_\Sigma)$ such that $x(0)=x_0$ and $x(1)=x_1$.
We consider the polynomial $P(t)=f(x(t)) \in B[t]_\Sigma$.
Since $P(0)= f(x(0)) \in B^\times$ and $P(1)= f(x(1)) \in B^\times$, it follows that
$P \in \Sigma$ hence $x(t) \in \gU(B[t]_\Sigma)$. Thus  $x_0$ and $x_1$ are directly $R$--equivalent in $\gU(B)$.

\smallskip

\noindent (2) Let $x(t) \in \gX(B[t]_\Sigma)$ such that $x(0)=x_1$ and $x(1)=x_1$.
Since $B[t]_\Sigma$ is a semilocal ring and the closed points of
$\Spec(B[t]_\Sigma)$ map to points of $\gU$, it follows that $x(t) \in \gU(B[t]_\Sigma)$.

\smallskip

\noindent (3)  This follows from the fact that two points of $\gG(B)$
are $R$-equivalent if and only they are directly $R$-equivalent according
to Lemma \ref{lem_homotopy}.(2).
 \end{proof}

\begin{slemma} \label{lem_sorite2}

\noindent (1) Let $\cL$ be a finitely generated locally free $B$--module and
consider the associated vector group scheme $\bW(\cL)$.
Let  $\gU \subset \bW(\cL)$
be an open  subset of the affine space $\bW(\cL)$.
We assume that  $\gU$ is a principal open subset or
 that  $B$ is semilocal.
Then any two points of $\gU(B)$ are directly $R$-equivalent.
In particular if $\gU(B) \not = \emptyset$, we have $\gU(B)/R= \bullet$.

\smallskip

\noindent (2) Let $\gG$ be an affine $B$--group scheme of finite presentation
such that $H^1(B,\gG)=1$, $H^1(B[t]_\Sigma,\gG)=1$ and $\gG(B)/R=1$.
Let $f: \gY \to \gX$ be a morphism of $B$--schemes which is a $\gG$--torsor. Then
the map $\gY(B)/R \to \gX(B)/R$ is bijective.
\end{slemma}

\begin{proof}
 (1) According to Lemma \ref{lem_sorite1}.(3), it is enough to show that
 two points of $\bW(\cL)(B)=\cL$ are $R$--equivalent.
 Let $x_0,x_1 \in \cL$ and consider
 $x(t)= (1-t)x_0 + t x_1 \in \cL \otimes_B B[t] \subset \cL \otimes_B B[t]_\Sigma
 = \bW(\cL)( B[t]_\Sigma ) $.
 Since  $x(0)=x_0$ and $x(1)=x_1$, we conclude that $x_0$ and $x_1$
 are directly $R$--equivalent.

 \smallskip

 \noindent (2) Since $H^1( B, \gG)=1$, it follows that the map $\gY(B)\to \gX(B)$ is surjective
 in view of \cite[prop.\ III.3.14]{Gd}; a fortiori the map $\gY(B)/R \to \gX(B)/R$ is onto.
 For the injectivity, it is enough to prove that two points
 $y_0, y_1 \in \gY(B)$ such that their images $x_0, x_1 \in \gX(B)$
 are directly $R$--equivalent are $R$-equivalent.
Our assumption is that there exists  $x(t) \in \gX(B[t]_\Sigma)$
 such that $x(0)=x_0$ and $x(1)=x_1$.
Since $H^1( B[t]_\Sigma, \gG)=1$ by assumption,  we can  lift $x(t)$ to some element $y(t) \in \gY(B[t]_\Sigma)$.
 Then $y_0=y(0). g_0$ and $y_1=y(1). g_1$ for (unique) elements $g_0, g_1$
 of $\gG(B)$. By (1), $g_0$ and $g_1$ are $R$--equivalent to $1$
 which enables us to conclude that $y_0$ and $y_1$ are $R$--equivalent.
 \end{proof}

\begin{sexamples}\label{ex_R_triv} {\rm
(1) The $B$-scheme $\GG_{m,B}$ is a principal
open subscheme of the vector group scheme $\GG_{a,B}$.
Lemma \ref{lem_sorite2}.(1) shows that $\GG_m(B)/R=1$.
\smallskip \newline
\noindent (2) More generally, let $C$ be a $B$--algebra
which is  finite locally free and consider the
Weil restriction $\gG=R_{C/B}(\GG_m)$.
It is the principal open subscheme of
the vector group scheme $R_{C/B}(\GG_{a,C})= \bV(C)$
defined by the norm map $N: R_{C/B}(\GG_{a,C}) \to \GG_{a,B}$.
Lemma \ref{lem_sorite2}.(1) shows that $\gG(B)/R=1$.
}
\end{sexamples}

\begin{slemma} \label{lem_sorite3}
 Let $\gG$ be a flat affine $B$--group scheme of finite presentation.
Let $f: \gY \to \gX$ be a morphism of $B$--schemes which is a $\gG_X$--torsor.  We assume either

\smallskip

(i) $\gG$
arises by successive extensions of vector group schemes
(with respect to finite locally free modules);

\smallskip

(ii) $\gG$ is a  split $B$--torus and
$\Pic(B)= \Pic(B[t]_\Sigma)=0$;

\smallskip

(iii) $B$ is regular semilocal  and
 $\gG$ is a quasitrivial $B$--torus;

 \smallskip

\noindent  Then the map $\gY(B)/R \to \gX(B)/R$ is bijective.

\end{slemma}

\begin{proof} According to Lemma \ref{lem_sorite2}.(2),
we need to show in each case that
$H^1(B,\gG)=1$, $H^1(B[t]_\Sigma,\gG)=1$ and $\gG(B)/R=1$.

\smallskip

\noindent{\it Case (i):}
 By induction we can assume that $\gG$ is a vector group scheme associated to
a finite locally free $R$--module $\cL$.
 In this case $H^1(C, \gG)=1$ for each $B$--ring $C$
 and $\gG(C)/R=1$ according to Lemma \ref{lem_sorite2}.(1).

\smallskip

\noindent{\it Case (ii):}  It remains  to show that $\GG_m(B)/R=1$ which follows of Lemma \ref{lem_sorite2}.(1).

\smallskip

\noindent{\it Case (iii):} We assume that $B$ is semilocal
and that $\gG=R_{B'/B}(\GG_m)$ for a finite \'etale $B$--algebra $B'$. The ring $B'$ is semilocal and regular
in view of \cite[Tag 07NF]{St} and so is $B'[t]$.
According to \cite[XXIV.8.4]{SGA3}, the maps
$H^1(B,\gG) \to H^1(B',\GG_m)$
and $H^1(B[t]_{\Sigma_B},\gG) \to H^1(B'[t]_{\Sigma_B},\GG_m)$ are isomorphisms. Since $B'$ is semilocal, we
have $H^1(B',\GG_m)=0$ so that $H^1(B,\gG)=0$.
According to \cite[thm.\ 2.2.(i)]{CTS2}, the map $\Pic(B'[t]) \to \Pic(B'[t]_{\Sigma_B})$ is onto.
In view of \cite[lem.\ 2.4]{CTS2}, the map
$\Pic(B') \to \Pic(B'[t])$ is an isomorphism
so that the map $\Pic(B') \to \Pic(B'[t]_{\Sigma_B})$ is onto. Since $B'$ is semilocal, we conclude that
$\Pic(B'[t]_{\Sigma_B})=0$.
We have established that
$H^1(B,\gG)=1$, $H^1(B[t]_\Sigma,\gG)=1$ and $\gG(B)/R=1$.
Finally we have $\gG(B)/R=1$ in view of Example \ref{ex_R_triv}.(2).
\end{proof}


\subsection{Retract rationality and $R$-equivalence}


It is well-known that over a field, there is a close relation between retract rationality of
algebraic varieties and the triviality of their $R$-equivalence class groups; see e.g.
the survey~\cite{CT}. We extend Saltman's definition \cite{Sa} of retract rationality over fields to the setting of pointed $B$--schemes.

A {\it pointed $B$-scheme} is  a pair $(\gX,x)$ consisting of a $B$--scheme $\gX$ and a point $x \in \gX(B)$.
A morphism $(\gX,x) \to (\gY,y)$ of pointed
$B$-schemes is a morphism  $f: \gX \to \gY$ of $B$-schemes
such that the following diagram commutes
\[\xymatrix@1{
\gX  \ar[r]^f & \gY \\
\Spec(B) \ar[u]^x \ar[ru]_y.
}\]
For short we say sometimes a {\it pointed morphism}.
By an open $B$--subset of  $(\gX,x)$ we mean
an open $B$--subset $\gU$ of $\gX$ such that there is a
(unique) factorization
\[\xymatrix@1{
\gU  \quad \ar@{^{(}->}[r] & \gX \\
& \Spec(B) \ar[u]^x \ar@{.>}[lu]^u;
}\]
in this case $(\gU,u) \to (\gX,x)$ is a pointed morphism.
\begin{sdefinition} \label{def_retract}
Let $(\gX,x)$ and $(\gY,y)$ be a pointed $B$--schemes
We say that $(\gX,x)$ is a $B$-retract of $(\gY,y)$
is there exist morphisms $i: (\gX,x) \to (\gY , y)$
and $p:(\gY,y) \to (\gX,x)$ of pointed $B$--schemes
such that $p \circ i= id_{(\gX,x)}$.
\end{sdefinition}
Note that $i$ is an immersion \cite[Tag 01KT]{St}
which is closed if $\gY$ is separated.
We remind the reader that an open $B$-subscheme
$\gU$ of a $B$-scheme $\gX$ is {\it $B$-dense} if it is dense in each fibre of $\gX$ over $\Spec(B)$. For example
$B$-density holds in the case when $\gU$ is an
open $B$-subscheme of a pointed scheme $(\gX,x)$
such that the fibers of $\gX \to \Spec(B)$ are irreducible.
Given  an immersion $i: X \to Y$ of schemes, 
we recall that $X$ is retrocompact in $Y$ if
$i$ is quasi-compact \cite[Tag 005A]{St}.

\begin{sdefinition} \label{def_rat} Let $(\gX,x)$ be a pointed $B$--scheme such that $\gX$ is finitely presented over $B$. We say that $(\gX,x)$ is

\begin{enumerate}
 \item  \emph{$B$--rational} if $(\gX,x)$ admits an
 open retrocompact $B$-subscheme $(\gU,x)$ which is $B$-dense such that $(\gU,x)$ is $B$--isomorphic to an  open retrocompact subscheme  of $(\mathbf{A}^N_B,0)$;

 \item \emph{stably $B$--rational} if $(\gX,x)$ admits an
 open retrocompact $B$-subscheme $(\gU,x)$ which is $B$-dense such that  $(\gU \times_B \mathbf{A}^d_B, (x,0))$ is $B$--rational for some $d \geq 0$;

 \item \emph{retract $B$--rational} if $(\gX,x)$ admits an open retrocompact $B$-subscheme $(\gU,x)$ which is $B$-dense and  which is a $B$--retract of an
 open retrocompact $B$--subset of some $(\mathbf{A}^N_B,0)$.
\end{enumerate}
\end{sdefinition}

\begin{sremarks} \label{rem_retro}
{\rm
(a) In (1), (2) and (3), the definition implies that $\gU$ is finitely presented  over $B$.
Of course, if $B$ is noetherian,  we can omit everywhere retrocompacity assumptions (see \cite[Tag 01OX]{St}).
\smallskip \newline
\noindent (b)  Let $\gV$ be a retrocompact open $B$-subset
of $(\gX,x)$ and $B$-dense. If $(\gX,x)$ is $B$-rational (resp.\
stably $B$-rational, resp.\ retract $B$-rational), so
is $(\gV,x)$.
\smallskip \newline
\noindent (c) If $\gX \to \Spec(B)$ has geometrically
irreducible fibers, $B$-density implies  then universal
$B$-density; in this case the three definitions
are stable after an arbitrary base change of the base ring.
}
\end{sremarks}

For later use, we record the following nice behaviour
under limits.

\begin{slemma} \label{lem_retro_limit} Let $B_0$ be a ring and assume that
$B=\limind B_i$ where the $B_i$'s are $B_0$-rings.
Let $(\gX_0,x_0)$ be a $B_0$--scheme of finite presentation
and having geometrically integral fibers.
Put $(\gX,x)=(\gX_0, x_0) \times_{B_0} B$. Then the following are equivalent:

\smallskip

(i) $(\gX, x)$ is $B$-rational (resp.\ $B$--stably
$B$-rational, resp.\ $B$--retract rational);

 \smallskip

(ii) There exists an index $i$ such that
$(\gX_0, x_0) \times_{B_0} B_j$ is $B_i$-rational (resp.\ $B_j$--stably  $B_j$-rational, resp.\ $B_j$--retract rational)
 for each $j \geq i$.
\end{slemma}

\begin{proof} The implication $(ii) \Longrightarrow (i)$
is the functoriality pointed out in Remark \ref{rem_retro}.(c). We prove the implication $(i) \Longrightarrow (ii)$
in the first case, the two others being similar.
Our assumption is that $(\gX,x)$ admits an
 open retrocompact $B$-subscheme $(\gU,x)$ which is $B$-dense such that $(\gU,x)$ is $B$--isomorphic to a  subscheme  of $(\mathbf{A}^N_B,0)$. Since $\gU$ is of finite presentation, over $B$, there exists an indice $i_1$ and
 a  finitely presented $B_{i_1}$-scheme $\gU_{i_1}$ such that
 $\gU \simlgr \gU_{i_1} \times_{B_{i_1}} B$ in view
 of  \cite[8.8.2.(2)]{EGA4}.
We put $\gX_i= \gX_0 \times_B B_i$ and
 $\gU_i= \gU_0 \times_{B_{i_1}} B_i$ for all $i \geq i_1$.
 The first item of the above reference
 provides an index $i_2 \geq i_1$ such that
 $x$  descends to point $x_{i_2} \in \gU_{i_1}(B_{i_2})$
 and such that   the map $\gU \to \gX$ (resp.\
 $\gU \to \mathbf{A}^N_B$)  descend to a map
 $f_{i_2}:  (\gU_{i_2}, x_{i_2}) \to (\gX_{i_2}, x_{0,i_2})$ (resp.\ $g_{i_2} : (\gU_{i_2}, x_{i_2}) \to
 (\mathbf{A}^N_{B_{i_2}},0)$).
 Applying \cite[8.8.5.(iii) and (iv)]{EGA4}
 provides an index $i_3$ such that
 $f_{i_2} \times_{B_{i_2}} B_{i_3}$ and
  $g_{i_2} \times_{B_{i_2}} B_{i_3}$  are open immersions.
 Both open immersions are of finite presentation  in view
 of \cite[Tag 02FV]{St} so are retrocompact.
 Finally the $B_{i_3}$-density of $\gU_{i_3}$
 in $\gX_{i_3}$ (resp.\ $\mathbf{A}^N_{B_{i_3}}$)
 follows of the comment after Definition \ref{def_retract}.
 Thus $(\gX_{i_3}, x_{0, i_3})$ is $B_{i_3}$--rational.
\end{proof}


\begin{slemma}  \label{lem_retract}
Assume that $B$ is semilocal. Let  $(\gU,x)$ be a pointed $B$-scheme which is a retract of an open subset $(\gV,0)$ of some $(\mathbf{A}^N_B,0)$. Then we have $\gU(B)/R=1$.
\end{slemma}
\begin{proof} In view of Lemma \ref{lem_sorite2}.(1), we have
$\gV(B)/R=1$. Since the map $\gU(B)/R \to \gV(R)$
admits a retraction, we conclude that $\gU(B)/R=1$.
\end{proof}

\begin{sdefinition}\label{def:lift}(1) We say that a $B$--scheme $\gX$ satisfies the
{lifting property} if for each semilocal $B$--ring $C$,
the map
\begin{equation}\label{eq_lift}
\gX(C) \to \prod\limits_{\gm \in \Specmax(C)} \gX(C/\gm)
\end{equation}
is onto, where $\Specmax(C)$ denotes the maximal spectrum of $C$.
\smallskip

\noindent (2) We say that a pointed $B$--scheme $(\gX,x)$
satisfies the lifting property if $\gX$ satisfies the lifting property.
\end{sdefinition}

\begin{slemma}\label{lem_lift}
(1) Let $n \geq 1$ be an integer. Then
$\mathbb{A}^n_B$ satisfies the lifting property.
\smallskip \newline
\noindent (2) Let $(\gX,x)$ be a $B$--scheme
satisfying the lifting property and
let $\gU$ be an open $B$-subscheme  of
$(\gX,x)$. Then $(\gU,x)$ satisfies the
lifting property.
\end{slemma}
\begin{proof}
(1) We can assume that $n=1$.
Let $C$ be a semilocal ring.
The map \eqref{eq_lift} for $\mathbb{A}^1_B$
 reads as $C \to \prod\limits_{\gm \in \Specmax(C)}
 C/\gm$. This map is surjective \cite[\S II.3.5]{BAC} so that
 the lifting property holds.
\smallskip
\noindent (2) Let $C$ be a semilocal $B$--ring.
We claim that the following commutative diagram
\[\xymatrix@1{
\gU(C) \ar@{^{(}->}[d] \ar[r]& \prod\limits_{\gm \in \Specmax(C)} \gU(C/\gm) \ar@{^{(}->}[d] \\
\gX(C) \ar[r]& \prod\limits_{\gm \in \Specmax(C)} \gX(C/\gm)}\]
is cartesian. Let $x \in \gX(C)$ mapping
to $\prod\limits_{\gm \in \Specmax(C)} \gU(C/\gm)$.
We put $\gV= \gU \times_{\gX} \Spec(C)$, this is an open
subscheme of $\Spec(C)$ which contains $\Specmax(C)$
so which $\Spec(C)$. It follows that $x \in \gX(C)$.
If the bottom map is onto, it follows
that the bottom map is onto. In other words,
if $(\gX,x)$ satisfies the lifting property, so does $\gU$.
\end{proof}

We extend Saltman's criterion of retract rationality \cite[th.\ 3.9]{Sa}.

\begin{sproposition} \label{prop_retract}
We assume that $B$ is semilocal with residue fields $\kappa_1$, \dots, $\kappa_c$.
Let $(\gX,x)$ be a pointed affine  finitely presented
 integral $B$-scheme with irreducible fibers. Then the following  assertions are equivalent:

\smallskip

$(i)$ $(\gX,x)$  is retract  $B$-rational;

 \smallskip

 $(ii)$  $(\gX,x)$   admits an
 open retrocompact affine
$B$-subscheme $(\gV,x)$,
$B$-dense and which satisfies the lifting property.
\end{sproposition}

\begin{sremarks}\label{rem_integral}{\rm

(a)  Note that, since $\gX$ is integral, the assumption $\gX(B) \not = \emptyset$ implies that $B$ is an integral ring.

\smallskip

(b) Assume that $B$ is an integral ring of field of fractions $K$.
Let $\gY$ be a flat  affine $B$--scheme such that $\gY_K$ is integral.
Then $B[\gY]$ injects in $K[\gY]$ so that $\gY$ is integral.
In particular, if $\gG$ is a smooth affine $B$-group scheme
such that $\gG_K$ is connected, then $\gG$ is integral.
}
\end{sremarks}

\begin{proof}[Proof of Proposition~\ref{prop_retract}]
Let $\gm_1, \dots, \gm_c$ be the maximal ideals
of $B$ and put $\kappa_i=B/\gm_i$ for $i=1,\dots,c$.

\smallskip

\noindent $(i) \Longrightarrow (ii)$. By definition $(\gX,x)$ admits an open retrocompact
$B$-subscheme $(\gU,x)$
which is $B$-dense and
 such that $(\gU,x)$ is a $B$--retract
 open retrocompact subscheme
 of some $(\mathbf{A}^N_B,0)$.
 The scheme $\gU$ is quasi-affine.
 According to \cite[Tag 0F20]{St},
$\gU$ admits  a principal  open subscheme $\gV$
which is affine and which contains the specializations $x_i$'s of $x$ with respect to the maximal ideals
$\gm_i$'s of $B$. Then $(\gV,x)$ is a $B$--retract of an open
 retrocompact subscheme of  $(\mathbf{A}^N_B,0)$.
 Since $\gV(B)$ is not empty, $\gV$ is $B$-dense
 in $\gU$ and a fortiori in $\gX$.
 The $B$-scheme $\gV$
 satisfies the lifting property by
 combining Lemma \ref{lem_lift}.(1) and (2).
%
 \smallskip \newline
 \noindent $(ii) \Longrightarrow (i)$.
Let $(\gV,x)$ be an open retrocompact subscheme of $(\gX,x)$, $B$-dense and
 satisfying the lifting property.
The same argument as for the direct application
provides an affine principal open $B$--subscheme
$\gV'$ of $\gV$ such that  $x \in \gV'(B)$
and such that $\gV'$ is $B$-dense in $\gX$.
Since $\gV$ is a retrocompact open subset of
$\gX$, so is $\gV'$; in particular $\gV'$ is of
finite presentation over $B$ (in view of \cite[Tag 01TU]{St}.
 Lemma \ref{lem_lift}.(2) shows that $(\gV',x)$ satisfies
 the lifting criterion.
 We denote by $x_i \in \gV'(\kappa_i)$ the image of
 $x$. Let $j: \gV' \to\mathbb{A}^N$ be a closed immersion and
  write $B[\gV']= B[t_1,\dots, t_N]/ \cP$ for a prime  ideal $\cP$ of $B[t_1,\dots, t_N]$ which is finitely generated.
  We denote by $\eta: \Spec(\kappa(\gX)) \to
 \mathbf{A}^N_B$ the generic point of $\gX$.
 We consider the semilocalization $C$ of  $B[t_1,\dots, t_N]$
 at the points $\eta , x_1, \dots, x_c$ of
 $\mathbf{A}^N_B$. Our assumption implies that the
 map
 $$
 \gV'(C)  \to \gV'( \kappa(\gX)) \times \gV'(\kappa_1) \times  \dots \times \gV'(\kappa_c)
 $$
 is onto. Let $y  \in \gV'(C)$ be a lifting of $(\eta, x_1,\dots, x_c)$. Since $\gV'$ is of finite presentation over $B$,
 the point $y$ extends to a principal  neighborhood
  $B[t_1,\dots, t_N]_f$ of $(\eta, x_1,\dots, x_c)$, i.e. there is a
  $B$-map $\phi: (\mathbf{A}^N_B)_{f} \to \gV'$ which satisfies
  $\phi(\eta)=\eta$ and $\phi(x_i)=x_i$ for $i=1,...,c$.
  The composite $\gV_{j \circ f} \to  (\mathbf{A}^N_B)_{f} \xrightarrow{\phi}  \gV$ fixes the point  $\eta$.
 Since $\eta$ is the generic point of $\gV_{j \circ f}$,
 this composite is the immersion map.
  Thus the open subset $(\gV'_{j \circ f},x)$ of $(\gV',x)$ is a $B$-retract of
  an open retrocompact
  subscheme  $(\mathbf{A}^N_B)_{f}$ of $\mathbf{A}^N_B$ and is $B$-dense.
\end{proof}

\begin{sremarks}{\rm (a) Under the assumptions of the proposition,
it follows that the retract rationality property is of
birational nature (with respect to our base point).
Furthermore inspection of the proof shows that if
$(\gX,x)$ is $B$--retract rational, we can take $\gV$ to be a principal open subset of $\gX$ and it is a $B$--retract of  a principal open subset of  $\mathbf{A}^N_B$.

\smallskip

\noindent (b) The direct implication $(i) \Longrightarrow (ii)$ does not require
$\gX$ to be integral.
}
\end{sremarks}

\begin{sexample}\label{ex_scheme_tori}{\rm  Let $B$ be a semilocal ring such that
its residue fields $\kappa_1, \dots, \kappa_c$ are infinite.
Let $\gG$ be a reductive
$B$--group scheme and let $\gT$ be a maximal $B$-torus of $\gG$
(such a torus exists according to Grothendieck's theorem \cite[XIV.3.20 and footnote]{SGA3}).
Let $\gX=\gG/\gN_\gG(\gT)$ be its $B$--scheme of maximal tori.
We claim that $\gX$ satisfies the lifting property so that, if $B$ is integral, $(\gX, \bullet)$ is
retract rational over $B$ according to Proposition~\ref{prop_retract}.
 It is enough to show that the map $\gX(B) \to \prod_{i=1,\dots, c} \gX(\kappa_i)$ is onto.
Let $T_i$ be a maximal $\kappa_i$-torus of $\gG_{\kappa_i}$ for $i=1,\dots, c$.
Since $\kappa_i$ is infinite, there exists $X_i \in \Lie(T_i)(\kappa_i) \subset \Lie(\gG)(\kappa_i)$
such that $T_i= C_{\gG_{\kappa_i}}(X_i)$ \cite[XIV.5.1]{SGA3}. We pick a lift $X \in  \Lie(\gG)(R)$
of the $X_i's$. Then $\gT= C_\gG(X)$ is a maximal $B$--torus of $\gG$
which lift the $T_i$'s. By inspection of the argument we can actually assume only that
$\sharp \kappa_i \geq \dim_{\kappa_i}( \gG_{\kappa_i})$ by using~\cite[Thm.\ 1]{Barnes}.

}
\end{sexample}

\begin{sproposition}\label{prop_vanish} We assume that $B$ is semilocal.
Let $\gG$ be a  $B$--group scheme
of finite presentation
 with connected geometric fibers such that

\smallskip

$(i)$ $(\gG,1)$ is retract $B$--rational;

\smallskip

$(ii)$ $\gG(\kappa)$ is dense in $\gG_\kappa$ for each residue field $\kappa$ of a maximal ideal of $B$.

\smallskip

\noindent Then  $\gG(B)/R=1$.
\end{sproposition}

Note that $(ii)$ is satisfied if $\gG$ is reductive and if $B$ has infinite residue fields.

\begin{proof} Let $\gm_1, \dots, \gm_c$ be
the maximal ideals of $B$ and put $\kappa_i=B/\gm_i$ for
$i=1,\dots, c$. The algebraic groups $\gG_{\kappa_i}$
are then retract rational and $\gG(\kappa_i)$ is dense
in $\gG_{\kappa_i}$ for $i=1,\dots, c$.

Let $\gU$ be an open retrocompact
subscheme of $(\gG,1)$
which is a $B$--retract of some open
retrocompact  of $\mathbf{A}^N_B$.
We consider the open subset
\begin{equation}
\gH= \bigcup_{u \in \gU(B)} u \, \gU \ \subset\ \gG
\end{equation}
Since $\gU(B)$ maps onto $\gU(\kappa_1) \times \dots \times
\gU(\kappa_c)$ (Lemma \ref{lem_lift}.(1) and (2) combined) and since  $\gU(\kappa_i)$ is dense in
$\gG(\kappa_i)$ by assumption,
it follows that $\gH_{\kappa_i}=\gG_{\kappa_i}$
for $i=1,...,c$.  It follows that $\gH=\gG$.
Lemma~\ref{lem_retract} shows that  $\gU(B)/R=1$.
We conclude that $\gG(B)/R=1$.
\end{proof}

\begin{sremark}\label{rem_quasitrivial} {\rm Proposition \ref{prop_vanish}  applies to quasitrivial tori
so that it is coherent with Lemma  \ref{lem_weil}.
}
\end{sremark}


\section{$R$-equivalence for reductive groups}



\subsection{$R$-equivalence as a birational invariant}


 The following statement generalizes \cite[Prop.\ 11]{CTS1}.

 \begin{sproposition} \label{prop_dominant}
  Assume that $B$ is semilocal with infinite residue fields. Let $\gG$ be a $B$--group scheme of finite presentation
  and let $f: (\gV, v_0) \to (\gG,1)$ be a $B$--morphism of pointed $B$--schemes such that
  $(\gV,v)$ is an open subset of some $(\mathbf{A}^n_B, 0)$ and such that
  $f_{B/\gm}$ is dominant for each maximal ideal $\gm$ of $B$.
  Let $(\gU, 1)$ be an open neighborhood of $(\gG,1)$.

  \smallskip

  \noindent (1) We have $f( \gV(B)) \, . \, \gU(B) = \gG(B)$.

  \smallskip

  \noindent (2) The map $\gU(B)/R \to \gG(B)/R$ is bijective.
 \end{sproposition}

\begin{proof} Let $\gm_1, \dots, \gm_c$
 be the maximal ideals of $B$.

 \smallskip

 \noindent (1) From the proof of \cite[Prop.\ 11]{CTS1}, we have
 $f( \gV(B/\gm_i)) \, . \, \gU(B/\gm_i) = \gG(B/\gm_i)$ for $i=1,\dots,c$.
 We are given $g \in \gG(B)$ and denote by $g_i$ its reduction
 to $\gG(B/\gm_i)$ for $i=1,\dots,c$, then
 $g_i= f(v_i) \, u_i$ for some $v_i \in \gV(B/\gm_i)$ and $u_i \in \gU(B/\gm_i)$.
Let $v \in \gV(B)$ be a common lift of the elements $v_i$, then
 $f(v)^{-1}g \in \gG(B)$ belongs to $\gU(B)$.

 \smallskip

 \noindent (2)  The surjectivity follows from (1) and the fact
  $\gV(B)/R=1$ established in Lemma \ref{lem_sorite2}.(1). On the other hand, the  injectivity has been proven in Lemma \ref{lem_sorite1}.(3).
\end{proof}

We say that a $B$-group scheme $G$ is $B$-linear, if for some $N\ge 1$
there is a closed embedding of $B$-group schemes $G\to\GL_{N,B}$.

\begin{slemma} \label{lem_unirational} Assume that $B$ is semilocal with infinite residue fields
$\kappa_1, \dots , \kappa_c$.
Let $\gG$ be a reductive $B$--group scheme.

\smallskip

(1)  There exist maximal $B$--tori $\gT_1, \dots, \gT_n$ of $\gG$ such that the
product map \break $\psi: \gT_1 \times \dots  \times \gT_n \to \gG$ satisfies the following property:
$\psi_{\kappa_j}$ is smooth at the origin  for each $j=1,\dots,c$.
Furthermore, the submodules $\Lie(\gT_i)(B)$ together generate $\Lie(\gG)(B)$ as a $B$--module.
\smallskip

(2) Assume furthermore that $\gG$ is $B$-linear. Then there exists a quasi--trivial
$B$--torus $\gQ$ and
a  $B$--morphism of pointed $B$--schemes $f: (\gQ, 1) \to (\gG,1)$  such that
  $f_{\kappa_j}$ is smooth at the origin for each $j=1,\dots,c$.
  \end{slemma}

\begin{proof}
(1) We start with the case of an infinite field $k$ and of a reductive $k$--group $G$.
We know that  $G(k)$ is Zariski  dense in $G$. Let $T$ be a maximal $k$--torus of $G$
and let $1=g_1,g_2 \dots, g_n$ be elements of $G(k)$ such that
$\Lie(G)$ is generated by the $^{g_i}\!\Lie(T)(k)$'s.
We consider the map of $B$--schemes

\[\xymatrix@1{
\gamma: \gT^n & \to & \gG \\
(t_1,\dots, t_n) & \mapsto &   {^{g_1}\!t_1} \dots {^{g_n}\!t_n} .
}\]
Its differential at $1$ is

\[\xymatrix@1{
d\gamma_{1_k}: \Lie(T)(k)^n & \to & \Lie(G)(k) \\
(X_1,\dots, X_n) & \mapsto &   {^{g_1}\!X_1} + \dots + {^{g_n}\!X_n}
}\]
which is onto by construction. We put $T_i= {^{g_i}\!T}$ for $i=1,...,n$
and observe that the product map $\psi: T_1 \times_k \dots \times_k T_n \to G$
is smooth at $1$. In this construction we are free to add more factors.

In the general case, we fix $n$ large enough and maximal $\kappa_j$--tori
 $T_{1,j}$,  \dots , $T_{n,j}$ such that the product map
 $\psi_i: T_{1,j} \times_{\kappa_j} \dots  \times_{\kappa_j} T_{n,j} \to \gG_{\kappa_j}$
 is smooth at $1$ for $j=1, \dots, c$. Example \ref{ex_scheme_tori}
 shows that there exists a maximal $B$--torus $\gT_{i}$
 which lifts the $T_{i,j}$'s for $i=1,...,n$. Then the product map
  $\psi: \gT_1 \times_B \dots \times_k \gT_n \to \gG$ satisfies the
  desired requirements.
Nakayama's lemma implies that
the $\Lie(\gT_i)(B)$'s generate $\Lie(\gG)(B)$ as a $B$--module.

 \smallskip

 \noindent (2) We assume that $\gG$ is linear so that the $\gT_i$'s are
 isotrivial according to \cite[Cor. 5.1]{Gi21}. Then by~\cite[Prop.\ 1.3]{CTS2} there exist flasque
resolutions $1 \to \gS_i \to \gQ_i \xrightarrow{q_i} \gT_i \to 1$
 of $\gT_i$ where $\gQ_i$ is a quasi-trivial
 $B$--torus and $\gS_i$ is a flasque $B$-torus for $i=1,...,n$.
 We consider the map

\[\xymatrix@1{
f: \gQ_1  \times_B \dots \times_B \gQ_n & \to & \gG \\
(v_1,\dots, v_n) & \mapsto &   q_1(v_1) \dots q_n(v_n) .
}\]
Since the $q_i$'s are smooth, $f=\psi \circ (q_1,\dots, q_n)$ satisfies the desired requirements.
\end{proof}


\begin{scorollary} \label{cor_retract_red}
  Assume that $B$ is semilocal with infinite residue fields. Let $\gG$ be a $B$-linear
  reductive $B$--group scheme
  and let  $\gU$ be an  open $B$--subscheme of $(\gG,1)$.

  \smallskip

  \noindent (1) Let $f: (\gQ, 1) \to (\gG,1)$
  be the  morphism constructed in Lemma \ref{lem_unirational}.(2) (where $\gQ$ is a quasi-trivial $B$-torus). Then $f( \gQ(B)) \, . \, \gU(B) = \gG(B)$.

  \smallskip

  \noindent (2)
The map $\gU(B)/R \to \gG(B)/R$ is bijective.
\end{scorollary}

\begin{proof}
(1)  This follows of  Proposition  \ref{prop_dominant}.(1).
\smallskip
\noindent (2) This follows of  Proposition  \ref{prop_dominant}.(2).
\end{proof}


\subsection{The case of tori}\label{subsec_tori}


Let $B$ be a commutative ring such that the connected components of $\Spec(B)$ are open (e.g. $B$ is Noetherian
or semilocal). Let $\gT$ be an isotrivial $B$--torus. According to
\cite[Prop.\ 1.3]{CTS2}, there exists a flasque resolution
$$
1 \to \gS \to \gQ \xrightarrow{\pi} \gT \to 1,
$$
that is an exact sequence of $B$--tori where $\gQ$ is
a quasitrivial $B$--torus and $\gS$ is a flasque $B$--torus.

\begin{sproposition} \label{prop_torus1}
Assume additionnally that $B$ is a regular integral domain.
 We have $\pi(\gQ(B)) = R \gT(B)$ and the characteristic
map $\gT(B) \to H^1(B,\gS)$ induces an isomorphism
$$
 \gT(B)/R \simlgr \ker\bigl( H^1(B,\gS) \to H^1(B,\gQ) \bigr).
$$
In particular, if $B$ is a regular semilocal domain, we have an isomorphism $\gT(B)/R \simlgr H^1(B,\gS)$.
\end{sproposition}

\begin{sremark}{\rm
This extends the corresponding result over fields  due to Colliot-Th\'el\`ene and Sansuc \cite[Thm. 3.1]{CTS1}.
Take the case $B=k$ a base field and a
$k$-torus $T$. A flasque resolution $1 \to S \to E \to T \to 1$ induces an isomorphism $T(k)/R \simlgr H^1(k,S)$.
}
\end{sremark}
We proceed now to the proof of Proposition \ref{prop_torus1}.

\begin{proof}
According to Example \ref{ex_R_triv}.(2) we have $\gQ(B)/R=1$,
hence the inclusion $\pi(\gQ(B)) \subseteq R \gT(B)$.
For the converse, it is enough to show that a point $x \in \gT(B)$
which is directly $R$--equivalent to $1$ belongs to
$\pi\bigl( \gQ(B) \bigr)$.
By definition, there exists a polynomial $P\in B[t]$ such that
$P(0), P(1) \in B^\times$ and $x(t) \in \gT\bigl( B[t,1/P] \bigr)$
satisfying ${x(0)=1}$ and $x(1)=x$. We consider the obstruction $\delta(x(t)) \in H^1(B[t,1/P], \gS)$.
Since $\gS$ is flasque and $B$ is a regular domain, the map
$$
H^1(B, \gS) \to  H^1(B[t,1/P], \gS)
$$
is onto by~\cite[Cor.\ 2.6]{CTS2}.
It follows that $\delta(x(t))= \delta(x)(0)=   \delta(x(0))=1$ so that
$x(t)$ belongs to the image of $\pi: \gQ\bigl( B[t,1/P] \bigr) \to \gT\bigl( B[t,1/P] \bigr)$.
Thus $x=x(1)$ belongs to $\pi\bigl( \gQ(B) \bigr)$.
 We have established that $\pi(\gQ(B)) \subseteq R \gT(B)$.
Combined with the long exact sequence of \'etale cohomology
$$
\gQ(B) \to \gT(B)  \to H^1(B,\gS) \to H^1(B, \gQ),
$$
we obtain the isomorphism  $\gT(B)/R \simlgr \ker\bigl( H^1(B,\gS) \to H^1(B,\gQ) \bigr)$.
\end{proof}

\smallskip

Using the above result, we extend Colliot-Th\'el\`ene and Sansuc's  criterion of retract rationality, see~\cite[Prop.\ 7.4]{CTS2} and~\cite[Prop.\ 3.3]{M96}.

\smallskip

\begin{sproposition}\label{prop_retract_torus} Let $B$ be a semilocal ring
and let $\gT$ be an isotrivial
$B$--torus.  Let  $1 \to \gS \to \gQ \xrightarrow{\pi} \gT \to 1$ be a flasque resolution.

\smallskip

\noindent (1) We consider the following assertions:

\smallskip

$(i)$  $\gS$ is an invertible $B$-torus (i.e. a direct summand of a quasitrivial $B$--torus);

\smallskip

$(ii)$ there exists an open retrocompact
subscheme $\gU$ of $(\gT,1)$
such that  $\pi^{-1}(\gU) \cong \gS \times_B \gU$;

\smallskip

$(ii')$ there exists a principal open  $\gU$ of $(\gT,1)$
such that  $\pi^{-1}(\gU) \cong \gS \times_B \gU$;

\smallskip

$(iii)$  the pointed $B$-scheme $(\gT, 1)$ is retract rational;

\smallskip

$(iv)$  $\gT$ is $R$--trivial on semilocal rings, that is $\gT(C)/R=1$ for each semilocal $B$-ring $C$;

\smallskip

$(iv')$  $\gT$ is $R$--trivial on fields, that is $\gT(F)/R=1$ for each $B$-field $F$.

\smallskip

$(v)$ $\gT$ satisfies the lifting property.

\smallskip

\noindent Then we have the implications
$(i) \Longrightarrow (ii)  \Longrightarrow  (iii)  \Longrightarrow  (iv)
\Longrightarrow  (iv') \Longrightarrow  (v)$.
Furthermore if $B$ is integral, we have the equivalences
$(iii) \Longleftrightarrow  (iv)  \Longleftrightarrow  (iv')
 \Longleftrightarrow (v)$.

\smallskip

\noindent (2) We assume furthermore that $B$ is a normal domain of fraction field $K$.
We consider the following assertions:

\smallskip

$(vi)$ $\gS_K$ is an invertible $K$-torus;

\smallskip

$(vii)$ $\gT_K$ is $R$--trivial on semilocal rings, that is, $\gT(A)/R=1$ for each semilocal $K$-ring $A$;

\smallskip

$(vii')$ $\gT_K$ is $R$--trivial on fields, that is, $\gT(F)/R=1$ for each $K$--field $F$;

\smallskip

$(viii)$  the pointed $K$-scheme $(\gT_K, 1)$ is retract rational.

\smallskip

\noindent Then the assertions $(i)$, $(ii)$, $(iii)$, $(iv)$, $(iv')$, $(v)$, $(vi)$,
$(vii)$, $(vii')$ and $(viii)$ are equivalent.
\end{sproposition}

Since the preceding statement is rather long,
we extract the following.
\begin{scorollary} \label{cor_retract_torus} Assume that $B$ is semilocal
normal domain of fraction field $K$. Let $\gT$ be a
$B$--torus. Then the following are equivalent:
\smallskip \newline
\noindent $(iii)$  the pointed $B$-scheme $(\gT, 1)$ is retract rational;
\smallskip \newline
\noindent $(viii)$  the pointed $K$-scheme $(\gT_K, 1)$ is retract rational.
\end{scorollary}
We proceed now to the proof of Proposition \ref{prop_retract_torus}.

\begin{proof}  Let $\gm_1, \dots, \gm_c$ be the maximal ideals of $B$
and put $\kappa_i=B/\gm_i$ for $i=1, \dots, c$.

\smallskip
\noindent (1)
 $(i) \Longrightarrow (ii)$. Let $C$ be the semilocal ring of $\gT$ at the points $1_{\kappa_1}, \dots,  1_{\kappa_c}$ of $\gT$.
Since $\gS$ is invertible, we have $H^1(C, \gS)=1$.
In particular the $\gS$--torsor $\pi: \gQ \to \gT$ admits a splitting $s: \Spec(C) \to \gQ$.
By definition, $C$ is the inductive limit of
the $B[\gT]_f$ for $f$ running to the elements
such that $f(1) \not \in \gm_i$ for
$i=1,\dots,c$.
 It follows that there exists a principal open  neighborhood $\gU$ of $(\gT,1)$ such that the $\gS$--torsor $\pi: \gQ \to \gT$ admits a splitting $s: \gU \to \gQ$.
Clearly $\gU$ is a retrocompact open subscheme of  $\gT$.

\smallskip

\noindent 
$(ii) \Longrightarrow (ii')$.
We are given  an open retrocompact neighborhood $\gU$ of $(\gT,1)$
such that  $\pi^{-1}(\gU) \cong \gS \times_B \gU$.
As in the proof of \ref{prop_retract}, we can find a
principal open subset $\gU'$ of $(\gT,1)$
such that $\gU' \subset \gU$.

\smallskip

\noindent 
$(ii') \Longrightarrow (iii)$.
We are given  a principal  open  neighborhood $\gU
=\gT_f$ of $(\gT,1)$
such that  $\pi^{-1}(\gU) \cong \gS \times_B \gU$.
Since $\gT$ has integral fibers over $B$,
$\gU$ is $B$-dense in $\gT$.
Thus $(\gU,1)$ is a $B$-retract of $(\pi^{-1}(\gU),1)$ which is open retrocompact in some affine $B$--space, so that $(\gT,1)$ is retract rational.

\smallskip

\noindent $(iii) \Longrightarrow (iv)$. Let $\kappa_1, \dots, \kappa_c$
the residue fields of the maximal ideals of $B$.
If all $\kappa_i$'s are infinite, Proposition \ref{prop_vanish} shows that $\gT$ is $R$--trivial.
It is enough to show that $\gT(B)/R=1$.
The general case requires more work.
\smallskip  \newline
\noindent{\it Case $B$ is the semilocalization of a finitely generated $\ZZ$-algebra.}
In particular $B$ is noetherian so that  we can assume
that $B$ is connected without loss of generality.
Let $B'$ be a finite \'etale extension of $B$
which splits $\gT$ and let $l$ be a prime number
which is coprime to the degree $d$ of $B'/B$.
By restriction-corestriction we have $d \, \gT(B)/R=1$.
According to \cite[prop.\ 2.10.(2)]{GN},
there exists an inductive limit $B_\infty= \limind_n B_n$
such that $B_0=B$, $B_n$ is semilocal and finite \'etale
of rank $l^n$ over $B$ and $B_\infty$ is semilocal.
Since $\gT \times_B B_\infty$ is $B_\infty$-retract rational
(Remark \ref{rem_def}.(3)), we have $\gT(B_\infty)/R=1$
by the first case.
By restriction-corestriction,
the maps $\gT(B)/R \to \gT(B_n)/R$ are injective.
On the other hand  we have
 $\limind \gT(B_n)/R=1 \simlgr \gT(B_\infty)/R$
 so that $\gT(B)/R=1$.
\noindent{\it General case.}
It goes by noetherian induction.
First we can write $B$ as a direct limit
$\limind_{i \in I} B_i$ where the $B_i$'s
are semilocalizations of  finitely generated $\ZZ$-algebras.
Without loss of generality we can assume that
$\gT$ has constant rank $r$.
Since $\gT$ is isotrivial,
it admits a closed immersion
$\rho: \gT \hookrightarrow \GL_{n,B}$ (which is a homomorphism), see \cite[Th.\ 3.3]{Gi21}.
According to the proof of \cite[Prop.\ 2.1.2.]{Co}
there exists an index $i_1$ and a closed immersion
$\rho_{i_1}: \gT_{i_1} \hookrightarrow \GL_{n,B_{i_1}}$
such that $\rho=\rho_{i_1} \times_{B_{i_1}} B$
and $\gT_{i_1}$ is a $B_{i_1}$--torus of rank $r$
which is isotrivial in view of \cite[Th.\ 3.3]{Gi21}.
Since $(\gT,1)$ is $B$--retract rational,
Lemma \ref{lem_retro_limit} provides an index $i_2 \geq i_1$
such that $(\gT_{i_1} \times_{B_{i_1}} B_j ,1)$
is $B_j$--retract rational for all $j \geq i_2$.
In view of the preceding case we have
$\gT_{i_1}(B_j)/R=1$ for $j \geq i_1$.
Since $\gT(B)/R= \limind \gT_{i_1}(B_j)/R$,
we conclude that $\gT(B)/R=1$.

\smallskip

\noindent $(iv) \Longrightarrow (iv')$.  Obvious.

\smallskip

\noindent $(iv') \Longrightarrow (v)$.  We assume that $\gT$ is $R$--trivial on fields.
It enough to show that $\gT(B)$ maps onto $\gT(\kappa_1) \times \dots \times \gT(\kappa_c)$.
We consider the commutative diagram
$$
\xymatrix@C=20pt{
 \gQ(B) \ar[r] \ar[d] &  \gT(B) \ar[d] \\
 \prod_i \gQ(\kappa_i) \ar[r]  &   \prod_i \gT(\kappa_i) .
}
$$
The left vertical map is onto since $\gQ$ satisfies the lifting property
and the bottom horizontal map is onto by Proposition~\ref{prop_torus1}, since $\gT(\kappa_i)/R=1$.
Thus the right vertical map is onto.

 \smallskip

 Finally, if $B$ is integral, then $\gT$ is an integral scheme according to
Remark \ref{rem_integral}.(b). Proposition \ref{prop_retract}
 shows that $(v)$ implies $(iii)$.

 \smallskip

\noindent (2) Since $B$ is normal,
$\gT$ and $\gS$ are isotrivial, that is, split by a finite
\'etale cover $B'$ of $B$ \cite[X.5.16]{SGA3}; we can assume that $B'$ is connected and Galois of group $\Gamma$.

Then $B'$ is a normal ring and its fraction field $K'$ is a Galois extension of $K$ of group
$\Gamma$.
According to Lemma \cite[Lemme 2, (vi)]{CTS1}, $(i)$
(resp.\ $(vi)$) is equivalent to
saying that the $\Gamma$--module $\widehat \gS(B')$ (resp.\, $\widehat \gS(K')$)  is invertible.
Since  $\widehat \gS(B')=\widehat \gS(K')$ we get the equivalence $(i) \Longleftrightarrow (vi)$.
The statement over fields~\cite[Prop.\ 7.4]{CTS2} provides the equivalences
  $(vi)   \Longleftrightarrow   (vii') \Longleftrightarrow  (viii)$.
Taking into account the first part of the Proposition  and the obvious
implications, we have  the following picture
{\small
$$
\xymatrix@C=10pt{
(i) \ar@2{<->}[d] & \Longrightarrow & (ii)  & \Longrightarrow & (iii) &
\Longleftrightarrow & (iv) \ar@2{->}[dllll] & \Longleftrightarrow & (iv') & \Longleftrightarrow & (v)  \\
  (vi) & \Longrightarrow &  (vii)  & \Longrightarrow & (vii') &  \Longleftrightarrow & (viii) &
\Longleftrightarrow & (vi)}
$$
Seeing twice the assertion (vi), we conclude that
the assertions $(i)$, $(ii)$, $(iii)$, $(iv)$, $(iv')$,
$(v)$, $(vi)$, $(vii)$, $(vii')$ and $(viii)$ are equivalent.}

\end{proof}


\subsection{Parabolic reduction}


Let $B$ be a ring and let $\gG$ be a reductive $B$--group scheme.
Let $\gP$ be a parabolic $B$--subgroup of $\gG$ together with an opposite
parabolic $B$--subgroup $\gP^{-}$.
We know that  $\gL= \gP \times_\gG \gP^{-}$ is a Levi subgroup of $\gP$.
We consider the big Bruhat cell
$$
\Omega:=\rad_u(\gP^{-}) \times_B \gL \times \rad_u(\gP^{-}) \subseteq \gG
$$

\begin{slemma} \label{lem_parabolic0} We have
$$
\gL(B)/R \simlgr \gP(B)/R \hookrightarrow \gG(B)/R.
$$
\end{slemma}

\begin{proof}
Let $\gU$ be the unipotent radical of $\gP$,
this is a successive extension of vector group schemes
and we have a Levi decomposition $\gP= \gU \rtimes \gL$
\cite[XVI.2.1 and 2.3]{SGA3}. In particular the projection
$\gP \to \gL$ is a $\gU$--torsor.
 Lemma  
\ref{lem_sorite3}, case (i), yields the  isomorphism
$\gP(B)/R \simlgr \gL(B)/R$ whence the isomorphism
$\gL(B)/R \simlgr \gP(B)/R$.
According to Lemma \ref{lem_cell}, the big cell $\Omega$ is a principal
open subset of $\gG$ so $ \Omega(B)/R$ injects in  $\gG(B)/R$
according to Lemma \ref{lem_sorite1}.(1).
Since $\rad_u(\gP)$ and $\rad_u(\gP^{-})$ are
successive extensions of vector group schemes,
the map $\Omega(B)/R \to \gP(B)/R$ is bijective by Lemma
\ref{lem_sorite3}, case (i),
hence $\gP(B)/R$ injects into $\gG(B)/R$.
\end{proof}

Given a homomorphism
$\lambda: \GG_{m,B} \to \gG$, its centralizer
$\gZ_\gG(\lambda)$ is representable
by a smooth closed $B$-group scheme of
$G$ \cite[Th.\ 4.1.7]{Co}. We can also consider the attractor functor
$\gP_\gG(\lambda)$ defined as
$$
\gP_\gG(\lambda)(C)= \Bigl\{g \in \gG(C) \,\mid \,
\lambda(t).g \in \gG(C[t]) \subset \gG(C[t, t^{-1}]) \Bigr\}
$$
for each $B$--algebra $C$. According to {\it loc.\ cit.\ \!\!},
this functor is representable by a closed smooth
$B$-subgroup scheme. Furthemore $\gP_\gG(\lambda)$
is a $B$--parabolic subgroup scheme of $\gG$ and
$\gZ_\gG(\lambda)$ is a Levi subgroup scheme of $\gP_\gG(\lambda)$ \cite[ex.\ 5.2.2]{Co}.
According to \cite[Thm.\ 7.3.1]{Gi4}, there exists a homomorphism $\lambda:
\GG_{m,B} \to \gG$ such that $\gP=\gP_\gG(\lambda)$ and
$\gL=\gZ_\gG(\lambda)$.
This reference requires $B$ to be connected
but the usual Noetherian reduction trick provides
the general case as in \cite[Cor.\ 7.3.2.(1)]{Gi4}

\begin{slemma}\label{lem_parabolic}
(1) Assume  that the group
$\gG(B)$ is generated by $\gP(B)$ and $\gP^{-}(B)$
(where $\gP^{-}$ is a $B$--parabolic subgroup
scheme of $\gG$ opposite to $\gP$).
Then the map $\gL(B)/R \to \gG(B)/R$ is an isomorphism.
\smallskip \newline
\noindent (2) Let $\gS$  be a  central split $B$-subtorus of
 $\gL$ such that  there is a factorization $\lambda: \GG_{m,B} \to \gS \hookrightarrow \gG$ and assume
that $\Pic(B)= \Pic(B[t]_\Sigma)=0$. Then we have an isomorphism
 $$
\xymatrix@C=30pt{
 \gL(B)/R \ar[r]^{\sim\quad} &   \bigl( \gL/\gS\bigr)(B)/R .
}
$$
\end{slemma}

\begin{proof}
(1) The map
$\gL(B)/R \to \gG(B)/R $  is injective
by Lemma \ref{lem_parabolic0}.
 The assumption implies that $\gL(B)$ generates $\gG(B)/R$
so that the injective map $\gL(B)/R  \to \gG(B)/R$ is onto hence an isomorphism.

\smallskip

\noindent (2) This follows of Lemma
\ref{lem_sorite3}, case (ii).
\end{proof}

\begin{scorollary}\label{cor_parabolic}
 We assume that $B$ is semilocal connected.
Let $\gS$  be a  central split $B$-subtorus of
 $\gL$ such that  there is a factorization $\lambda: \GG_{m,B} \to \gS \hookrightarrow \gG$.
Then we have isomorphisms
 $$
\xymatrix@C=30pt{
 \gG(B)/R  & \ar[l]_\sim  \gL(B)/R \ar[r]^{\sim\quad} &   \bigl( \gL/\gS\bigr)(B)/R .
}
$$
In particular, this holds if
$\gS$ is the maximal central split $B$--subtorus
of $\gL$ (as defined in \cite[XXVI.6.8]{SGA3}).
\end{scorollary}

\begin{proof}
As a consequence of Demazure's theorem
\cite[XXVI.5.1]{SGA3} (see \cite[th.\ 3.1.(c)]{GN}), we have
that $\gG(B)= \rad^u(\gP)(B) \, \rad^u(\gP^{-})(B) \, \gP(B)$
so that the group $\gG(B)$ is generated by
$\gP(B)$ and $\gP^{-}(B)$.  Lemma \ref{lem_parabolic}.(1)
applies and shows that the map $\gL(B)/R \to \gG(B)/R$ is
an isomorphism.
Also  we have $\Pic(B)= \Pic(B[t]_\Sigma)=0$ since $B$ and $B[t]_\Sigma$ are
semilocal rings. Thus Lemma \ref{lem_parabolic}.(2) applies
and shows that the map $\gL(B)/R \to    \bigl( \gL/\gS\bigr)(B)/R$ is an isomorphism.

Assume that $\gS$ is the  maximal central split $B$--subtorus of $\gL$. We need to show that
there is a factorization $\lambda: \GG_{m,B} \to \gS \hookrightarrow \gG$.
Put $K=\ker(\lambda)$. According
to \cite[IX]{SGA3}, $K$ is a $B$--subgroup of multiplicative type of $\GG_{m,B}$ and $\lambda$ factorizes uniquely
as follows
$$
\GG_{m,B} \to \GG_{m,B}/K \xrightarrow{\ol{\lambda}}
\gG
$$
   where $\GG_{m,B}/K$ is a $B$--group
of multiplicative type  (of finite type)
and  $\ol{\lambda}$ is  a closed immersion.
Since $B$ is connected, $K$ and
$\GG_{m,B}/K$ are diagonalizable in view of
\cite[IX.2.7.(1)]{SGA3} so that
$\GG_{m,B}/K$ is a split $B$-torus. Since, moreover, $\gL=\gZ_\gG(\lambda)$, it follows that $\GG_{m,B}/K \subseteq \gS$
so that $\lambda$ factorizes through $\gS$.
\end{proof}


\section{$\mathbf{A}^1$-equivalence and non-stable $K_1$-functors}



\subsection{$\mathbf{A}^1$-equivalence}


Let $B$ be an arbitrary (unital, commutative) ring.
Let $\cF$ be a $B$-functor in sets.
We say that two points $x_0,x_1 \in \cF(B)$ are {\it directly $\mathbf{A}^1$--equivalent} if
there exists
$x \in \cF\bigl( B[t]\bigr)$
such that  $x_0=x(0)$ and $x_1=x(1)$.
The (naive) $\mathbf{A}^1$-equivalence on $\cF(B)$ is the equivalence relation
generated by this relation.

Let $G$ be a $B$--group scheme.
We denote the equivalence class of $1\in G(B)$ by $\r0G(B)$ and the group of
$\mathbf{A}^1$-equivalence classes by
$$
G(B)/\r0=G(B)/\r0G(B).
$$
This group is functorial in $B$, and the functor $G(-)/\r0$ on the category of $B$-schemes is sometimes called the 1st
Karoubi-Villamayor $K$-theory functor corresponding to $G$, and denoted by
$KV_1^G(B)$~\cite{J,AHW}.

Clearly, for any ring $B$ we have a canonical surjection
$$
G(B)/\r0 \, \to \hskip-6mm \lgr \, G(B)/R.
$$
The analog of Lemma~\ref{lem_homotopy} is true for $\mathbf{A}^1$-equivalence.
In particular, two points $g_0,g_1 \in G(B)$ are $\mathbf{A}^1$-equivalent
if and only if they are directly $\mathbf{A}^1$--equivalent.


\subsection{Patching pairs and $\mathbf{A}^1$-equivalence}


Let $R \to R'$ be a morphism of rings and let $f \in R$.
We say that that $(R \to R',f)$ is a
{\it patching pair} if
$R'$ is flat  over $R$ and $R/fR \simlgr R'/fR$.
The other equivalent terminology is to say that
\begin{equation}\label{eq:patching}
\xymatrix{
R \ar[d]  \ar[r]  & R_f  \ar[d]   \\
R'   \ar[r]  & R'_f
}
\end{equation}
is a  {\it patching diagram}.
In this case, there is an  equivalence of categories between the category
of  $R$-modules and the category  of glueing data $(M',M_1,\alpha_1)$
where $M'$ is an $R'$--module, $M_1$ an $R_f$--module
and $\alpha_1: M' \otimes_{R'} R'_f \simlgr M_1 \otimes_{R_f} R'_f$  \cite[Tag 05ES]{St}.
Note that this notion of a patching diagram is less restrictive than the one used
by Colliot-Th\'el\`ene and Ojanguren in~\cite[\S 1]{CTO}.

\begin{sexamples}\label{ex_patch}{\rm
 (a) (Zariski patching) Let $g \in R$ such $R=fR+gR$.
Then \break $(R \to R_g,f)$ is a patching pair.

\smallskip

\noindent (b) Assume that $R$ is noetherian.
If  $\widehat R=\limproj R/f^nR$,
then $(R \to \widehat R, f)$ is a patching pair according to \cite[Tags 00MB, 05GG]{St}.

\smallskip

\noindent (c)  Assume that $R=k[[x_1,\ldots,x_n]]$ is a ring of formal power series
over a field and let $h$ be a monic Weierstrass polynomial of $R[x]$ of degree $\geq 1$.
Then $(R[x]\to R[[x]],h)$ is a patching  pair, see \cite[page 803]{BR}.
}
\end{sexamples}

We recall that $(R \to R',f)$ is a {\it glueing pair} if
$R/f^n R \simlgr R'/ f^n R'$ for each $n \geq 1$ and if the sequence
\begin{equation}\label{complex}
0 \to R \to R_f \oplus R' \xrightarrow{\gamma}  R'_f \to 0
\end{equation}
is exact where $\gamma(x,y)=x -y$ \cite[Tag 0BNI]{St}.

\smallskip

\begin{sexamples}\label{ex:glue}
{\rm
\noindent (a) A patching pair  is a glueing pair. Indeed, we have $R/f^n R \simlgr R'/f^n R'$
for all $n \geq 1$ by~\cite[Tags 05E7 and 05E9]{St},
and the complex \eqref{complex} is exact at $R$ and $R_f \oplus R'$ by~\cite[Tag 05EK]{St},
and at $R'_f$ by~\cite[Tag 0BNR]{St}.

\smallskip

\noindent (b) If $f$ is a non zero divisor in  $R$ and $\widehat R=\limproj R/f^nR$,
then $(R \to \widehat R, f)$ is a glueing pair \cite[Tag 0BNS]{St}, even if $R\to \hat R$ is not flat.
}
\end{sexamples}

If $(R \to R',f)$ is a glueing pair, the Beauville-Laszlo theorem
provides an equivalence of categories between the category
of flat  $R$-modules and the category  of glueing data $(M',M_1,\alpha_1)$
where $M'$ is a flat $R'$--module, $M_1$ a flat $R_f$--module
and $\alpha_1: M' \otimes_{R'} R'_f \simlgr M_1 \otimes_{R_f} R'_f$~\cite[Tags 0BP2, 0BP7 and 0BNX]{St}.
In particular we can patch torsors under an affine flat $R$--group scheme $G$ in this setting, this means
that the base change induces  an equivalence from the
category of $G$-torsors  to that of triples $(T,T', \iota)$
where $T$ is a $G$-torsor over $\Spec(R_f)$, $T'$ a $G$--torsor over $\Spec(R')$
and $\iota: T \times_{R_f} R'_f \simlgr T' \times_{R'} R'_f$ an isomorphism
of $G$--torsors over $\Spec(R'_f)$, see~\cite[lemma 2.2.10]{BC}. This is a generalization of~\cite[proposition 2.6]{CTO}.
More specifically, there is an exact sequence of pointed sets
\begin{equation}\label{eq:patch-torsors}
1 \to G(R') \backslash G(R'_f) / G(R_f)
\to H^1(R,G) \to H^1(R' ,G) \times H^1(R_f,G).
\end{equation}
This sequence can be used to relate the $\mathbf{A}^1$-equivalence on $G$ with local triviality
of $G$-torsors.

\begin{slemma}\label{lem_KV}
Let  $G$ be a flat $B$-linear $B$--group scheme. Let $h \in B$.

 \medskip

\noindent (1)
Let $(B \to A,h)$ be a glueing pair  and   assume  that
\begin{equation}\label{eq:torsor-ker-cond}
\ker\bigl( H^1(B[x], G) \to H^1(B_h[x], G) \bigr)=1.
\end{equation}
 Then  we have $\r0G( A_h) =  \r0G(A) \, \r0G(B_h)$ and the map
\begin{equation}\label{eq:BB-square23}
\ker\bigl( G(B)/\r0 \to G(B_h)/\r0  \bigr)
\, \to \,
\ker\bigl( G(A)/\r0 \to G(A_h)/\r0  \bigr)
\end{equation}
is surjective.

 \smallskip

\noindent  (2)
   Assume that $h$ is a non zero divisor in $B$ and that~\eqref{eq:torsor-ker-cond} is satisfied.
Let $\widehat B= \limproj_{n \geq 0} \, B/h^{n+1}B$ be the completion.
 Then we have the equality
 $\r0G(\widehat B_h) = \r0G(\widehat B) \, \r0(B_h)$
 and the map

\begin{equation}\label{eq:BB-square25}
\ker\bigl( G(B)/\r0 \to  G(B_h)/\r0   \bigr)   \, \to \,
\ker\bigl( G(\widehat B)/\r0 \to  G\bigl( \widehat B_h \bigr)/\r0 \bigr)
\end{equation}
is surjective.
Assuming  furthermore that $G( \widehat B_h) = G(\widehat B) \, \r0G(\widehat B_h)$,
 we have $G(B_h) = G( B) \, \r0(B_h)$.

\end{slemma}

\begin{proof}(1)  Since $(B[t] \to A[t],h)$ is a glueing pair, we have an exact sequence of pointed sets
$$
1 \to G(B_h[x]) \backslash G(A_h[x]) / G(A[x])
\to H^1(B[x] ,G) \to H^1(B_h[x] ,G) \times H^1(A[x],G).
$$
Then our assumption provides a decomposition $G(A_h[x]) = G(A[x]) \, G(B_h[x])$,
and a fortiori a decomposition  $G(A_h) = G(A) \, G(B_h)$.
Let  $x \in \r0G( A_h)$. Then there exists $g \in G(A[x]_h)$
such that $g(0)=1$ and $g(1)=x$. We can decompose then
$g= g_1 g_2$ with $g_1 \in G(A[x])$, $g_2 \in  \, G(B[x]_h)$.
Since  $1= g_1(0) g_2(0)$ we can assume that
 $g_1(0)=1$ and $g_2(0)=1$.
It follows that $x \in \r0G(A) \, \r0G(B_h)$.
This  establishes the equality
 $\r0G(A_h) = \r0G(A) \, \r0G(B_h)$.

For showing the surjectivity   of the map  \eqref{eq:BB-square23},
let $x\in G(A)$ be such that its image in $G(A_h)$ belongs to $\r0G(A_h)$. Then there are $y\in\r0G(A)$, $z\in \r0G(B_h)$
such that $x=yz$. Since $xy^{-1}\in G(A)$ and $z\in G(B_h)$ have the same image in $G(A_h)$,
and since $(B,A,h)$ is a glueing pair, there is an element
$\tilde x\in G(B)$ such that the image of $\tilde x$ in $G(A)$ is $xy^{-1}$ and the image of $\tilde x$ in $G(B_h)$
is $z$. It follows that $[\tilde x] \in G(B)/\r0$ is mapped to $[x]\in G(B)/\r0$ and to
$[1]\in G(B_h)/\r0$, as required.

\smallskip

\noindent (2) By Example~\ref{ex:glue} (b) this is the special case $A=\widehat B$ of (1).
The last fact is a straightforward consequence.
Indeed, we have
$$
G(\hat B_h)=G(\hat B)\cdot\r0G(\hat B_h)=G(\hat B)\cdot\r0G(\hat B)\cdot\r0G(B_h)=G(\hat B)\cdot\r0G(B_h).
$$
Since the sequence~\eqref{complex} for the pair $(B\to \hat B,\, h)$ is exact,
an element of $G(\hat B)$ that belongs to the image  of $G(B_h)$ in $G(\hat B_h)$ lifts to $G(B)$.
It follows that $G(B_h)=G(B)\cdot\r0G(B_h)$.
\end{proof}




The condition~\eqref{eq:torsor-ker-cond} in Lemma~\ref{lem_KV} is not easy to check in general.
 Later on we will discuss a case where it is known to hold as a corollary of the work of Panin on the
Serre--Grothendieck conjecture~\cite{Pa,Pa20}. However, Moser obtained the following unconditional result
in the special case of Example~\ref{ex_patch} (a).

\begin{slemma}\label{lem_moser} (Moser, \cite[lemma 3.5.5]{Moser}, see also \cite[lemma 3.2.2]{AHW})
 Let $G$ be a finitely presented $B$-group scheme which is $B$-linear.

 \smallskip

 \noindent (1) Let $f_0,f_1 \in B$ such that $Bf_0+Bf_1=B$.
Let $g \in  G(B_{f_0 f_1}[T ])$ be an element such that  $g(0) = 1$. Then there exists a decomposition
$g = h_0^{-1} \, h_1$ with $h_i \in G(B_{f_i}[T ])$ and $h_i(0) = 1$
for $i=0,1$.

 \smallskip

 \noindent (2) The sequence of pointed sets
 $$
\xymatrix@C=20pt{
 G(B)/\r0 \ar[r]  & G(B_{f_0})/\r0 \times G(B_{f_1})/\r0 \ar@<2pt>[r]\ar@<-2pt>[r] & G(B_{f_0f_1})/\r0
}
$$
 is exact at the middle term.
 \end{slemma}
\begin{proof}
 (1) The original reference does the case $B$ noetherian and the general case holds
 by the usual noetherian approximation trick.

 \smallskip

 \noindent (2) Let $[g_0] \in G(B_{f_0})/\r0$ and let $[g_1] \in G(B_{f_1})/\r0$
 such that $[g_0]=[g_1] \in G(B_{f_0f_1})/\r0$. Then there exists $g \in G(B_{f_0f_1}[T])$
 such that $g_0 \, g_1^{-1}= g(1) \in G(B_{f_0f_1}[T])$ and $g(0)=1$.
 By (1) we write $g = h_0^{-1} \, h_1$ with $h_i \in G(B_{f_i}[T ])$ and $h_i(0) = 1$
for $i=0,1$ so that $g_0 \, g_1^{-1} = h_0^{-1}(1) \, h_1(1)$.
Since $[h_i(1) g_i]= [g_i] \in G(B_{f_i})/\r0$, we can replace $g_i$ by $h_i(1) g_i$
and deal then with the case  $g_0 = g_1 \in  G(B_{f_0f_1})$.
This defines an unique element $m \in G(B)$ such that $[m]= [g_i] \in G(B_{f_i})/\r0$.
\end{proof}

 \begin{sremark}{\rm By induction we get the following generalization.
 Let $f_1, \dots, f_c \in B$ such that $Bf_1+ \dots + Bf_c=B$ and put
 $f=f_1\dots f_c$.
Let $g \in  G(B_{f}[T ])$ be an element such that  $g(0) = 1$. Then there exists a decomposition
$g = h_1 \dots h_c$ with $h_i \in G(B_{f_i}[T ])$ and $h_i(0) = 1$
for $i=1, \dots, c$. It follows that the image of
$G(B)/\r0$ in $\prod_{i=1,..,c} G(B_{f_i})/\r0$
consists of elements having same image in $G(B_f)/\r0$.
 }
 \end{sremark}

Since Lemma~\ref{lem_moser} does not presuppose any results about $G$-torsors,
Moser was able to use it to establish a local-global principle for torsors~\cite[3.5.1]{Moser}
generalizing Quillen's local-global principle for finitely presented
modules~\cite[Theorem 1]{Q}.  In our context, we combine Lemma~\ref{lem_moser} with a
theorem of Colliot-Th\'el\`ene and Ojanguren to obtain the following result.

\begin{sproposition} \label{prop_moser}
Let $k$ be an infinite field and let $G$ be an affine $k$--algebraic group. Let $A$
be the local ring at a prime ideal of a polynomial algebra $k[t_1, \dots, t_d]$. Then the homomorphism
$$
G( A)/\r0 \to G\bigl( k(t_1, \dots, t_d) \bigr)/\r0
$$
is injective.
\end{sproposition}
\begin{proof}
Our plan is to use Colliot-Th\'el\`ene and Ojanguren method \cite[\S 1]{CTO} as abstracted in
the appendix \ref{appendix_cto}.
We  consider the  $k$--functor in groups $B \mapsto F(B)=G( B)/\r0$.
The claim follows from Proposition~\ref{prop_cto}
once properties $\bf P_1$, $\bf P_2$ and $\bf P'_3$ are checked for the $k$--functor
$F$. The property $\bf P_1$ is clear, since $G$ is finitely presented over $k$.

Let $L$ be a $k$--field and let $d \geq 0$ be an integer. We
have $F(L)= F\bigl( L[t_1,\dots, t_d])$, and $F(L)$ injects in $ F\bigl( L(t_1,\dots, t_d) \bigr)$,
since every polynomial over $L$ has an invertible value. Property $\bf P_2$  is established. On the other hand Lemma \ref{lem_moser}.(2) establishes
the surjectivity of the  map
$$
 \ker\bigl( G(B)/\r0 \to G(B_{f_0})/\r0 \bigr) \to
 \ker\bigl( G(B_{f_1})/\r0 \to G(B_{f_0f_1})/\r0 \bigr)
$$
for $B=Bf_0+Bf_1$ so that Zariski patching property $\bf P'_3$ holds
for the functor $F$.
\end{proof}

\begin{sremark}{\rm The extension to the finite field case is established in
Corollary \ref{cor:polynomial_local}.
 }
\end{sremark}





\subsection{Non stable $K_1$-functor} \label{subsec_non_stable}


 Let $\gG$ be a reductive  group scheme over our base ring $B$.
We say that a parabolic $B$-subgroup $\gP$ of $\gG$ is {\it strictly proper,
if $\gP$ itersects properly every semisimple normal $B$-subgroup of $\gG$,
or, equivalently, if
for any ring homomorphism $B\to\bar k$ from $B$ to an algebraically closed field $\bar k$,
the type (in the sense of~\cite[Exp. XXVI, \S 3.2]{SGA3}) of the parabolic subgroup $\gP_{\bar k}$
does not contain any connected component of the Dynkin diagram of $\bG_{\bar k}$.}

Let $\gP$ be a strictly proper parabolic subgroup of $\gG$.
Let $\gP^{-}$ be an opposite $B$--parabolic subgroup scheme of $\gG$, and denote by
$E_\gP(B)$ the subgroup of $\gG(B)$ generated by $\rad_u(\gP)(B)$ and  $\rad_u(\gP^{-})(B)$
(it does not depend on the choice of $\gP^-$ by~\cite[XXVI.1.8]{SGA3}).
We consider the Whitehead coset
$$
K_1^{\gG,\gP}(B)=\gG(B)/ E_\gP(B).
$$
As a functor on the category of commutative $B$-algebras, $K_1^{\gG,\gP}(-)$ is also called the non-stable (or unstable)
$K_1$-functor associated to $\gG$ and $\gP$.

Recall that if $B$ is semilocal, then the functor $C\mapsto E_\gP(C)$ on the category of commutative $B$-algebras
$C$ does not depend on the choice of a strictly proper parabolic $B$-subgroup $\gP$, see~\cite[XXVI.5]{SGA3}
and~\cite[Th.\ 2.1.(1)]{S1}. In particular, in this case $E_\gP(B)$ is a normal subgroup of $\gG(B)$.
For an arbitrary ring $B$, the same holds if $\gG$ satisfies the condition (E) below, see~\cite{PS}.
In these two cases we will occasionally write $K_1^\gG(C)$ instead of $K_1^{\gG,\gP}(C)$, omitting the specific
strictly proper parabolic $B$-subgroup.

\medskip

\noindent{\it Condition} (E). For any maximal ideal $\gm$ of $B$,
all irreducible components of the relative root system
of $\gG_{B_\gm}$ in the sense of~\cite[XXVI.7]{SGA3}  are of rank at least 2.

\medskip

Note that the condition (E) is satisfied if $\gG$ has $B$-rank $\ge 2$, since in this case all $\gG_{B_\gm}$
also have $B_\gm$-rank $\ge 2$.

Since the radicals $\rad_u(\gP)$ and $\rad_u(\gP^{-})$ are
successive extensions of vector group schemes~\cite[XXVI.2.1]{SGA3},
Lemma \ref{lem_sorite2}.(1)
implies that $E_\gP(B) \subseteq \r0 \gG(B) \subseteq \gG(B)$.
We get then  surjective maps
\[
K_1^{\gG,\gP}(B) \, \to \hskip-6mm \lgr \, \gG(B)/\r0 \, \to \hskip-6mm \lgr \, \gG(B)/R.
\]


\subsection{Comparison of $K_1^{\gG}$, $\mathbf{A}^1$-equivalence and $R$-equivalence}


\begin{slemma}\label{lem_W_trivial2}
We consider the following assertions:

\smallskip

$(i)$ The map $K_1^{\gG,\gP}(B) \to K_1^{\gG,\gP}(B[u])$ is bijective;

\smallskip

$(ii)$ $\gG(B[u])= \gG(B) \, E_\gP(B[u])$;

\smallskip

$(iii)$ The map $K_1^{\gG,\gP}(B) \to G(B)/\r0$ is
 bijective.

 \smallskip

\noindent Then we have the implications $(i) \Longleftrightarrow (ii)
 \Longrightarrow (iii)$.
 Furthermore if $(iii)$  holds,
 we have that $E_\gP(B)=\r0\gG(B)$; in particular
 $E_\gP(B)$ is a normal subgroup of $\gG(B)$ which does
 not depend of $\gP$.
 \end{slemma}

\begin{proof} $(i) \Longleftrightarrow (ii)$. The map $K_1^{\gG,\gP}(B) \to K_1^{\gG,\gP}(B[u])$ is always
injective, since it has a left inverse induced by $u\mapsto 0$. Clearly, this map is surjective,
if and only if we have the decomposition $\gG(B[u])= \gG(B) \, E_\gP(B[u])$.

\smallskip

\noindent $(ii) \Longrightarrow (iii)$.
The map $K_1^{\gG,\gP}(B) \to G(B)/\r0$ is
 surjective. Let $g_0, g_1 \in \gG(B)$ mapping to the same element
 of $G(B)/\r0$. There exists $g(t) \in \gG(B[t])$
 such that $g(0)=g_0$ and $g(1)=g_1$.
 Our assumption implies that $g(t)= g \, h(t)$ with $g \in \gG(B)$
 and $h(t) \in E_\gP(B[u])$. It follows that $g_i= g \, h(i)$ for
 $i=0,1$ with $h(i) \in E_\gP(B)$.
 We get that $g_0= g \, h(0)= (g \, h(1)) \, (h(1)^{-1} \, h(0)) \in g_1 \, E_\gP(B)$.
 Thus $g_0, g_1$ have same image in $K_1^{\gG,\gP}(B)$.

\smallskip

The last assertion of the lemma is immediate.
\end{proof}



\begin{sremarks}\label{rem_invariance}{\rm (a) Assume that $\gG$ satisfies condition $(E)$.
In this case, homotopy invariance reduces to the case of
the ring $B_\gm$ for each maximal ideal $\gm$ of
$B$ according to a generalization of the Suslin local-global principle \cite[lemma 17]{PS}.

\smallskip

\noindent (b) If $B$ is a regular ring containing a field $k$, and $\gG$ satisfies $(E)$,
then we know that $K_1^{\gG}(B)  \simlgr K_1^{\gG}(B[u])$ by~\cite[Th.\ 1.1]{St20}.

\smallskip

\noindent (c) Let us provide a counterexample to $K_1^{\gG}(B)  \simlgr K_1^{\gG}(B[u])$ in the  non-regular case.
Given a field  $k$ (of characteristic zero), we consider the domain $B=k[x^2, x^3] \subset k[x]$.
We claim that $K_1^{\SL_n}(B) \subsetneq K_1^{\SL_n}(B[u])$
for $n>>0$ so that $1=K_1^{\SL_n}(B_\gm) \subsetneq K_1^{\SL_n}(B_\gm[u])$
for some maximal ideal of $B$.
For  $n >>0$, we have $K_1^{\SL_n}(B)= SK_1(B)$ and
$K_1^{\SL_n}(B[u])= \SK_1(B[u])$.
Inspection of the proof of Krusemeyer's computation
of $\SK_1(B)$ \cite[Prop.\ 12.1]{Kr} provides
functorial maps  $\Omega^1_A \to \SK_1(A \otimes_k B)$ for
a $k$--algebra $A$.
We get then  commutative diagram of maps
$$
\xymatrix@C=30pt{
\Omega^1_k   \ar[d] \ar[r]^\sim  &  \SK_1(B) \ar[d] \\
\Omega^1_{k[u]}   \ar[d] \ar[r] &  \SK_1(B[u]) \ar[d] \\
\Omega^1_{k(u)}    \ar[r]^\sim &  \SK_1(B_{k(u)})
}
$$
where the top and the bottom horizontal maps are isomorphisms \cite[Prop.\ 12.1]{Kr}.
Since $\Omega^1_k \subsetneq \Omega^1_{k[u]}$,
a diagram chase yields that $\SK_1(B) \subsetneq \SK_1(B[u])$.
Since $K_1^{\SL_n}(B_m)=1$, this example also shows that
the condition $(iii)$ of Lemma~\ref{lem_W_trivial2} does not imply $(i)$.

\noindent (d) In case of regular rings, the condition $(iii)$ of Lemma~\ref{lem_W_trivial2} may hold while $(i)$ does not,
if $\gG$ does not satisfy (E). Let $k$ be a field.
Let $\gP$ be the standard parabolic subgroup of $\SL_2$ consisting of upper triangular matrices.
Then one has $\SL_2(k[x])=E_\gP(k[x])$.
Consequently, $K_1^{\SL_2,\gP}(k[x])=1$, and hence $\SL_2(k[x])/\r0=1$, so $(iii)$ holds.
On the other hand, $K_1^{\SL_2,\gP}(k[x,u])\neq 1$~\cite{C}, so $(i)$ does not hold.

}
\end{sremarks}

\begin{slemma}\label{lem_W_trivial3} We consider the  following assertions:

\smallskip

$(i)$ The map $K_1^{\gG,\gP}(B) \to K_1^{\gG,\gP}(B[u]_\Sigma)$ is bijective;

\smallskip

$(ii)$ $\gG(B[u]_\Sigma)= \gG(B) \, E_\gP(B[u]_\Sigma)$;

\smallskip

$(iii)$ The map $K_1^{\gG,\gP}(B) \to G(B)/R$ is
 bijective.

 \smallskip

 \noindent Then we have the implications $(i) \Longleftrightarrow (ii)
 \Longrightarrow (iii)$.  Furthermore if $(iii)$ holds,
 we have that $E_\gP(B)=R\gG(B)$; in particular
 $E_\gP(B)$ is a normal subgroup of $\gG(B)$ which does
 not depend of $\gP$.

\end{slemma}

\begin{proof} This is similar to that of Lemma
\ref{lem_W_trivial2}
\end{proof}


\section{Passage to the field of fractions}





\begin{slemma}\label{lem_monic}
Let $B$ be a regular ring containing a field, and let $G$ be a reductive group over $B$ having a strictly proper parabolic $B$-subgroup.
Let $f\in B[x]$ be a monic polynomial.
Then the natural map of \'etale
cohomology sets $H^1_{\et}(B[x],G)\to H^1_{\et}(B[x]_f,G)$
has trivial kernel.
\end{slemma}

\begin{proof}
Clearly, we can assume that $B$ is a domain.
Let $K$ be the field of fractions of $B$.
By~\cite[Lemma 5.4]{St20} for any maximal ideal $m$ of $B$
the map $H^1_{\et}(B_m[x],G)\to H^1_{\et}(K[x],G)$ has trivial kernel.
Furthermore, the map $H^1_{\et}(K[x],G)\to H^1_{\et}(K(x),G)$ has trivial kernel by~\cite[Proposition 2.2]{CTO}.
Then for any monic polynomial $f$ the map $H^1_{\et}(B_m[x],G)\to H^1_{\et}(B_m[x]_{f},G)$
has trivial kernel.
Since $B$ is regular, by~\cite[Corollary 3.2]{Thomason} $G$ is $B$-linear. Since $G$ is reductive,
it is also $B$-flat. Then the claim holds for $G$ by~\cite[Lemma 4.2]{St20}.
\end{proof}

In the extreme opposite case we have the following fact.

\begin{slemma} \label{lem_dec_aniso}
Let $B$ be a Noetherian commutative ring, and
let $G$ be a $B$-linear reductive $B$-group. We assume
that $G_{B/m}$ is anisotropic for each maximal ideal $m$ of $B$.
Let $f\in B[x]$ be a monic polynomial.
Then the natural map of \'etale
cohomology sets
$H^1_{\et}(B[x],G)\to H^1_{\et}(B[x]_f,G)$
has trivial kernel.
\end{slemma}
\begin{proof}
Assume first that $B$ is semilocal.
Let $\xi=[E]\in H^1_{\et}(B[x],G)$ be an element of  the kernel.
We extend $E$ to a $G$-bundle $\hat E$
 on $\mathbb{P}^1_B$
by patching it to the trivial $G$-bundle over  $\mathbb{P}^1_B \setminus\{f=0\}$.
We denote by $\hat\xi$ its class;
since $f$ is monic, we have $\hat\xi|_\infty=*$.

Let $m_1,\dots, m_c$ be the maximal ideals of $B$ and put  $k_i=B/m_i$.
Since $G_{k_i}$ is anisotropic, then by~\cite[Th.\ 3.8 (b)]{Gi02} $\hat\xi_{k_i}$ is trivial.
Next we apply \cite[Lemma 8.3]{Ces-GS} and get that $\hat\xi$
belongs to the image of $H^1_{\et}(B,G) \to H^1_{\et}(\mathbb{P}^1_B,G)$.
Since $\hat\xi|_\infty=*$, we conclude that $\hat\xi=*$.
Thus $E$ is a trivial $G$--torsor over $B[x]$.

If $B$ is not necessarily semilocal, the claim reduces to the maximal localizations of $B$
by applying the local-global principle~\cite[Lemma 4.2]{St20}.
\end{proof}

\begin{sremarks}{\rm
\noindent (a) The statement that a principal $G$-bundle on $\mathbb{P}^1_B$
is in the image of $H^1_{\et}(B,G) \to H^1_{\et}(\mathbb{P}^1_B,G)$, once it has
trivial restrictions to $\mathbb{P}^1_{B/m_i}$ for all $i$~\cite[Lemma 8.3]{Ces-GS}, is sometimes called the rigidity property of $G$-bundles.
It was proved in~\cite[Th.\ 1]{R78} and~\cite[Prop.\ 9.6]{PaStV} under the assumption that $B$ is semilocal and contains a field
(i.e. is equicharacteristic).
Tsybyshev~\cite[Theorem 1]{Tsy} was able to prove it assuming only that $B$ is reduced
and $\Pic(B)=0$.
\v{C}esnavi\v{c}ius~\cite{Ces-GS} observed that one can remove the condition that $B$ contains a field by using
Alper's theorem stating that $\GL_N/G$ is affine for any $B$~\cite[Cor.\ 9.7.7]{A}.
The idea to use~\cite[Th.\ 3.8 (b)]{Gi02} for anisotropic groups appeared in~\cite[p. 178]{FP}
and in~\cite[Th.\ 1 and remark 2.1.(iii) on the anisotropic case]{F}.
Fedorov also introduced the use of affine Grassmannians to treat the case of not necessarily semilocal $B$ and anisotropic $G$
~\cite[Theorem 5]{F}.

\smallskip

\noindent (b) Let $G_0$ the underlying Chevalley $B$--group scheme
of $G$. The condition of linearity on $G$ is satisfied if
the  $\Out(G_0)_S$--torsor $\mathrm{Isomext}(G_{0}, G)$ is isotrivial,
see \cite[Prop.\ 3.2]{M2}; this reference provides then a representation
such that $\GL_n/G$ is affine, so there is no need to appeal
to Alper's result in this case.
This includes the semisimple case and
the case when $B$ is a normal ring due to Thomason \cite[Corollary 3.2]{Thomason}.

\smallskip

\noindent (c) The claim of Lemma~\ref{lem_dec_aniso} does not hold if
$G$ is anisotropic over $B$ and isotropic
over $B/m$, even if $B$ is regular local  and $G$ is simply connected~\cite[Corollary 2.3]{F}.

 }
\end{sremarks}

\begin{stheorem}\label{thm:surj}
Let $B$ be a regular semilocal domain that contains a field $k$, and let $K$ be the
fraction field of $B$.
Let $\gG$ be a reductive $B$-group scheme.

(1) Assume that either $\gG$ contains a strictly proper parabolic $B$-subgroup, or
$\gG$ is anisotropic over $B/m$ for all maximal ideals $m$ of $B$.
Then the map
$$
\gG(B)/R\to \gG(K)/R
$$
is surjective.

(2) Assume that $\gG$ contains a strictly proper parabolic $B$-subgroup.
Then  the map
$$\gG(B)/\r0 \to \gG(K)/\r0
$$ is injective.
\end{stheorem}
\begin{proof}
Clearly, we can assume that $k$ is a finite field or $\mathbb{Q}$
without loss of generality. Then the embedding $k\to B$ is geometrically regular,
since $k$ is perfect~\cite[(28.M), (28.N)]{Mats}.
Then by Popescu's theorem~\cite{Po90,Swan} $B$ is a filtered direct limit of smooth
$k$-algebras.
Since the group scheme $\gG$ is finitely
presented over $B$, and the functors $\gG(-)/R$ and $\gG(-)/\r0$
commute with filtered direct limits, we can assume that $\gG$ is defined over
a smooth $k$-domain $C$, and
$B=C_S$ is a localization of $C$ at a set $S$ that is the complement of a union of a finite set of
 prime ideals $p_i$ of $C$. Moreover, since parabolic subgroups of $\gG$ are also finitely presented,
depending on the assumption on $\gG$ we can secure that
$\gG$ contains a strictly proper parabolic subgroup over $C$, or $\gG$ is anisotropic over
 $C_{p_i}/p_iC_{p_i}$ for all $p_i$'s.

(1) We need to show that $\gG(B)/R\to \gG(K)/R$ is surjective, where $K$
is the fraction field of $B$ and $C$. Clearly, it is enough to show the same
 for the localization of $C$
at the complement of the union of maximal ideals $m_i\supseteq p_i$ (note that if $\gG$ is anisotropic over
$C_{p_i}/p_iC_{p_i}$, then it is automatically anisotropic over $C/m_i$). Hence we can
assume that $B$ is
a localization of $C$ at the complement of a union of a finite set of maximal ideals. On top
of that, in order to show that
$\gG(B)/R\to \gG(K)/R$ is surjective, it is enough to show that for any
 $f\in\bigcap_i m_i$ and any $g\in \gG(C_f)$
the image of $g$ in $\gG(K)$ belongs to $\gG(B)\cdot R\gG(K)$.

We apply Panin's theorem \cite[Th.\ 2.5]{Pa}.
This provides  a monic polynomial $h \in B[t]$,
an inclusion of rings $B[t] \subset A$, a homomorphism $\phi: A \to B$  and a
commutative diagram
\begin{equation}\label{eq:panin-diag}
\xymatrix@C=30pt{
 B[t] \ar[d] \ar[r]   & A \ar[d]    &   \ar[l]_{u} C  \ar[d]\\
 B[t]_h \ar[r] & A_h     & C_f \ar[l]_{v}.
}
\end{equation}
such that

\smallskip

(i) the left hand square is a elementary distinguished Nisnevich square in the category of smooth $B$-schemes in the sense of \cite[3.1.3]{MV};

\smallskip

(ii) the composite $C  \xrightarrow{u} A \xrightarrow{\phi} B$ is
the canonical localization homomorphism;

\smallskip

(iii) the map $B[t] \to A \xrightarrow{\phi}  B$ is the evaluation at $0$;

\smallskip

(iv) $h(1) \in B^\times$;

\smallskip

(v) there is an $A$--group scheme isomorphism $\Phi: \gG_B \times_B A \simlgr \gG \times^u_C A$.

\medskip

By inspection of the construction  $A$ is finite \'etale over $B[t]$ and
$h(t) = N_{A/B[t]}(u(f))= u(f) a$ with $a \in  A$.
Property (4) of  \cite[theorem 3.4]{PaStV} states that the map $\phi: A \to  B$
extends to a map $A_a \to  B$, so that
$\phi(a) \in  B^\times$.
We compute
\begin{eqnarray} \nonumber
h(0) &=& \phi(h)  \qquad [\mbox{property \enskip} (iii)]\\ \nonumber
&=& \phi(u(f)) \, \phi(a)  \\ \nonumber
&=& f \, \phi(a) \qquad [\mbox{property \enskip} (ii)];
\end{eqnarray}
it  follows that $h(0)$ is a non-zero element of $B$.
In particular $\phi$ extends to a map $\phi_h: A_h \to B_{h(0)}$.

\noindent
Since $(B[t]\to A,h)$ is a glueing pair, we have an exact
sequence of pointed sets
$$
1 \to \gG(B[t]_h) \backslash \gG(A_h) / \gG(A)
\to H^1(B[t] ,\gG) \to H^1(B[t]_h ,\gG) \times H^1(A,\gG).
$$
Our assumptions on $\gG$ imply that the map $H^1(B[t],\gG) \to H^1(B[t]_h,\gG)$
has trivial kernel.
Indeed, if $\gG$ contains a strictly proper parabolic subgroup over $B$,
this follows from Lemma~\ref{lem_monic}.
If $\gG$ is anisotropic modulo every maximal ideal of $B$, then the same
follows from Lemma~\ref{lem_dec_aniso}, taking into account that $B$ is regular and hence
by~\cite[Corollary 3.2]{Thomason} $G$ is $B$-linear. Therefore we have $\gG(A_h)=\gG(B[t]_h)\gG(A)$.

Set $\widetilde g=\Phi^{-1}(v_*(g)) \in \gG(A_h)$.
Then $\widetilde g=b\cdot a$, where $b\in \gG(B[t]_h)$ and $a\in\gG(A)$. Note that by (iii) we have
$\phi(h)=h(0)$.
We have $\phi_h(\widetilde g)=\phi_h(v(g))=g\in \gG(B_{h(0)})$ by (ii). It follows that $g=\phi_h(b)\cdot\phi_h(a)$.
Clearly we have $\phi_h(a)\in\gG(B)\subseteq \gG(B_{h(0)})$.
We claim that $\phi_h(b)\in\gG(B)\cdot R\gG(B_{h(0)})$. Indeed,
we have $\phi_h(b)=b|_{t=0}$ by (iii), and since $h(1)\in B^\times$, we have $b|_{t=1}\in\gG(B)$.
Then the image of $b$ in $\gG\bigl(B_{h(0)}[t]_h \bigr)$ provides an $R$-equivalence between $\phi_h(b)$ and an element of $\gG(B)$.
Summing up, the image of any $g\in\gG(C_f)$ under the composition
  $\gG(C_f)\xrightarrow{v}\gG(A_h)\xrightarrow{\phi_h}\gG(B_{h(0)})$ belongs to $\gG(B)\cdot R\gG(B_{h(0)})$. It follows that
the image of $g$ in $\gG(K)$ belongs to $\gG(B)\cdot R\gG(K)$.

(2) Let $[g] \in \ker\bigl( \gG(B)/\r0 \to \gG(K)/\r0 \bigr)$.
Up to shrinking of $X=\Spec(C)$, we can assume that  $g \in \gG(C)$.
Then there exists then $f \in C$ such that
$[g] \in \ker\bigl( \gG(C)/\r0 \to \gG(C_f)/\r0 \bigr)$.
As in (1), we apply Panin's theorem \cite[Th.\ 2.5]{Pa} and obtain a diagram~\eqref{eq:panin-diag} satisfying the properties
(i)--(v).
But this time we set $\widetilde g=\Phi^{-1}(u_*(g)) \in \gG(A)$ and we have
$[\widetilde g] \in  \ker\bigl( \gG(A)/\r0 \to \gG(A_h)/\r0 \bigr)$.
According to Lemma \ref{lem_monic}, the map $H^1(B[t][x],\gG) \to H^1(B[t]_h[x],\gG)$
has trivial kernel so that  Lemma \ref{lem_KV}.(1) shows that the map

\begin{equation}\label{eq:BB-square27}
\ker\bigl( \gG(B[t])/\r0 \to  \gG(B[t]_h)/\r0   \bigr)  \, \to \,
\ker\bigl( \gG(A)/\r0   \to  \gG(A_h)/\r0 \bigr)
\end{equation}
is surjective.
Since $\gG(B)/\r0= \gG(B[t])/\r0$ and $h(1)\in B^\times$,
we deduce  that
$$
\ker\bigl( \gG(A)/\r0 \to \gG(A_h)/\r0 \bigr)=1.
$$
We have $[\widetilde g] =1 \in \gG(A)/\r0$
and get $[u_*(g)]=1 \in  \gG(A)/\r0$.
By applying $\phi_*$, the property (ii) yields $[g]=1 \in  \gG(B)/\r0$.
\end{proof}

\begin{scorollary}\label{cor:polynomial_local}
Let $k$ be a field and let $G$ be an affine $k$--algebraic group. Let $A$
be the local ring at a prime ideal of a polynomial algebra $k[t_1, \dots, t_d]$. Then the homomorphism
$$
G( A)/\r0 \to G\bigl( k(t_1, \dots, t_d) \bigr)/\r0
$$
is injective.
\end{scorollary}
\begin{proof}
If $k$ is infinite, this is the claim of Proposition~\ref{prop_moser}.
Assume that $k$ is finite. Let $G_{red}$ denote the reduced affine algebraic $k$-scheme corresponding to $G$.
Since $k$ is perfect, $G_{red}$ is a smooth algebraic $k$-subgroup of $G$~\cite[Prop.\ 1.26, Cor.\ 1.39]{Milne}.
Since $A$ is reduced, $G(A)=G_{red}(A)$ and $G(A[u])=G_{red}(A[u])$, therefore, $G(A)/\r0=G_{red}(A)/\r0$,
and hence we can assume that $G$ is smooth from the start. Let $G^\circ$ be the connected component of the identity $e\in G(k)$.
Let $\pi_0(G)$ be the finite \'etale $k$-scheme of connected components of $G$
Then $G^\circ$ is a smooth geometrically connected algebraic $k$-subgroup of $G$, the fiber of the natural map
$G\to\pi_0(G)$ at the image of $e$~\cite[Prop.\ 1.31, 1.34]{Milne}. Since $\pi_0(G)$ is $k$-finite, we have $\pi_0(G)(A[u])=\pi_0(G)(A)$,
and hence $\pi_0(G)(A)/\r0=\pi_0(G)(A)$ injects into $\pi_0(G)(K)/\r0=\pi_0(G)(K)$, where $K=k(t_1,\ldots,t_d)$.
Therefore, in order to prove the claim for $G$, it is enough to prove it for $G^\circ$.
Hence we can assume that $G$ is smooth and connected. Let $U$ be the unipotent radical of $G$ over $k$, i.e. the largest
smooth connected unipotent normal $k$-subgroup of $G$. Since $k$ is perfect,
the group $U$ is $k$-split, admits a subnormal series each of whose quotients are isomorphic to $\A^1_{k}$~\cite[14.63]{Milne}.
Therefore $U(A)/\r0=1$ and $H^1(R,U)=1$ for every $k$-algebra $R$.
Also, since $k$ is perfect, $G/U$ is a reductive algebraic $k$-group~\cite[Prop.\ 19.11]{Milne}. By Lang's theorem
 \cite[Th.\ 2]{Lg},
 $G/U$ admits a Borel $k$-subgroup hence is quasi-split;
therefore either $G/U$ is a $k$-torus, or it contains a strictly proper parabolic $k$-subgroup
and then satisfies Theorem~\ref{thm:surj} (2). In both cases the map
$(G/U)(A)/\r0\to (G/U)(K)/\r0$ is injective. Now let $g\in G(A)$ be mapped into $\r0G(K)\subseteq G(K)$. By the
previous argument, there is $h(u)\in (G/U)(A[u])$ such that $h(0)=1$ and $h(1)$ is the image of $g$ in $(G/U)(A)$.
Since $H^1(A,U)=H^1(A[u],U)=1$, there is $g(u)\in G(u)$ such that $g(0)\in U(A)$ and $g(1)g^{-1}\in U(A)$.
Since $U(A)\subseteq \r0G(A)$, we conclude that $g\in\r0G(A)$, as required.
\end{proof}

\begin{scorollary}\label{cor:hens_surj}
Let $B$ be a henselian regular local ring that contains a field, and let $K$ be the
fraction field of $B$. Let $\gG$ be a reductive $B$-group scheme.
Then the map
$$
\gG(B)/R\to \gG(K)/R
$$
is surjective.
\end{scorollary}

\begin{proof}
Assume first that $\gG$ is anisotropic. Let $m$ be the maximal ideal of $B$ and let $l=B/m$. Since $B$
is henselian, $\gG_l$ is also
anisotropic by~\cite[Exp. X, Cor. 4.6]{SGA3} and~\cite[Exp. XXVI, 7.15]{SGA3}. Thus, Theorem~\ref{thm:surj} applies to
$\gG$.

Next, assume that $\gG$ is not anisotropic, and let $\gS$ be a maximal split $B$-subtorus of $\gG$.
By~\cite[Exp. XXVI, Prop. 6.16]{SGA3} $\gG$ contains a  parabolic $B$-subgroup $\gP$ such that $\gS$ is
the maximal central split $B$-subtorus of a Levi subgroup $\gL$ of $\gP$ (it is possible that $\gP$ is not proper,
i.e. $\gG=\gP=\gL$, and $\gS$ is a central subtorus of $\gG$).
By Corollary~\ref{cor_parabolic}
there is a factorization $\lambda: \GG_{m,B} \to \gS \hookrightarrow \gG$ (in the notation of that corollary),
and isomorphisms
 $$
\xymatrix@C=30pt{
 \gG(B)/R  & \ar[l]_\sim  \gL(B)/R \ar[r]^{\sim\quad} &   \bigl( \gL/\gS\bigr)(B)/R.
}
$$
Next, $\gS_K$ is a split $K$-subtorus of $\gL_K$, such that there is a factorization
$\lambda_K: \GG_{m,K} \to \gS_K \hookrightarrow \gG_K$. Therefore,
applying the same corollary to $\gG_K$, $\gL_K$ and $\gS_K$, we conclude that there are isomorphisms
 $$
\xymatrix@C=30pt{
\gG(K)/R  & \ar[l]_\sim  \gL(K)/R \ar[r]^{\sim\quad} &   \bigl( \gL/\gS\bigr)(K)/R .
}
$$
It follows that in order to show that $\gG(B)/R\to\gG(K)/R$ is surjective, it is enough to show that
$\bigl( \gL/\gS\bigr)(B)/R\to \bigl( \gL/\gS\bigr)(K)/R$ is surjective. The group $\gL/\gS$ is an anisotropic reductive
$B$-group, therefore, the previous case applies.
\end{proof}


\section{The case of simply connected semisimple isotropic groups}



\subsection{Coincidence of equivalence relations}


We address the following question.

\begin{squestion} \label{question_iso} Assume that $B$ is regular semilocal and that
$\gG$ is semisimple simply connected
and
equipped with  a strictly proper parabolic $B$-subgroup.\\
Is the map $K_1^{\gG,\gP}(B) \, \to \,  \gG(B)/R$ an isomorphism?\\
\noindent Is the map $\gG(B)/\r0 \, \to \,  \gG(B)/R$ an isomorphism?
\end{squestion}

The answer is known to be positive in both cases if $B$ is a field. This is implied
by Margaux--Soul\'e isomorphism  \cite[Th.\ 3.10]{M1} combined with~\cite[Th.\ 7.2]{Gi2}.


\begin{stheorem} \label{thm_KV_R}
Assume that $B$ is a semilocal regular domain
containing a field $k$ and denote by $K$ its fraction field.
Let $\gG$ be a semisimple simply connected $B$-group having a strictly proper parabolic $B$-subgroup.
Then we have a commutative square of isomorphisms
$$
\xymatrix@C=30pt{
G(B)/\r0  \ar[d]_\wr \ar[r]^{\ \sim} &  G(B)/R \ar[d]_\wr \\
G(K)/\r0 \ar[r]^{\ \sim} &  G(K)/R
}
$$
\end{stheorem}

\begin{proof} Let $K$ be the fraction field of $B$.
The bottom horizontal arrow of the square is an isomorphism by the Margaux--Soul\'e theorem~\cite[Th.\ 3.10]{M1} combined with~\cite[Th.\ 7.2]{Gi2}.
On the other hand,   the left vertical map is  injective by Theorem \ref{thm:surj} (2).
Then the top horizontal arrow is also injective. Since it is surjective by definition, it is an isomorphism.
The right vertical arrow is surjective by Theorem \ref{thm:surj} (1). Hence the vertical arrows are also isomorphisms.
\end{proof}


\begin{sremark}{\rm
The above result does not extend to anisotropic groups. For example, let $k$ be an infinite field and let $G$ be a  wound linear algebraic group, i.e. does not contain any subgroups isomorphic to $\GG_a$ or $\GG_m$.
Then by~\cite[Corollary 3.8]{GF}
we have $G(k[x])=G(k)$ and, consequently, $G(k)/\r0=G(k)$.
This applies in particular to the case of an anisotropic reductive $k$--group $G$.
On the other hand,
the $R$-equivalence class group of $G$ may be even trivial, e.g. if $G$ is a
semisimple anisotropic group of rank $\le 2$. Indeed, in this case every element of $G(k)$ is $R$-equivalent to a semisimple regular
element, and all maximal tori of $G$ are of rank $\le 2$ and hence rational.
}
\end{sremark}

In the same vein, we can establish the following fact.

\begin{scorollary}
Let $k$ be a field and let
$G$ be a semisimple simply connected $k$--group
of $k$-rank $\geq 1$.
Let $A$ be the localization of $k[x_1,\dots, x_d]$ at
a prime ideal. Then we have a commutative square of isomorphisms
$$
\xymatrix@C=30pt{
G(k)/\r0  \ar[d]_\wr \ar[r]^\sim &  G(k)/R \ar[d]_\wr  \\
G(A)/\r0 \ar[r]^{\sim} &  G(A)/R.
}
$$
\end{scorollary}
\begin{proof}
By~\cite[Th.\ 5.8]{Gi2} there is an isomorphism
$G(k)/\r0 \simlgr G\bigl(k(x_1,\dots, x_d) \bigr)/\r0$.
Then the claim follows from Theorem~\ref{thm_KV_R}.
\end{proof}




\begin{stheorem} \label{thm_main}
Assume that $B$ is a semilocal regular domain
containing a field $k$ and that
$\gG$ is semisimple simply connected $B$-group of $B$--rank $\geq 2$.
Then  the map $K_1^\gG(B) \to \gG(B)/R$ is an isomorphism.
\end{stheorem}

\begin{proof} Let $K$ be the fraction field of $B$.
We consider the commutative diagram
$$
\xymatrix@C=30pt{
K_1^{\gG}(B)  \ar[d] \ar@{->>}[r] &  G(B)/R \ar[d] \\
K_1^{\gG}(K) \ar[r]^{\sim} &  G(K)/R
}
$$
where the bottom isomorphism is \cite[Th.\ 7.2]{Gi2}.
On the other hand,   the left vertical map is  injective
 \cite[Th.\ 1.2]{St20}. By diagram chase, the top horizontal
 map is an isomorphism.
\end{proof}


\subsection{The retract rational case}



\begin{slemma}\label{lem_W_trivial}
We assume that  the base ring $B$ is a semilocal domain.
Let $\gG$ be a reductive $B$--group scheme having
a strictly proper $B$--parabolic subgroup $\gP$.
 We consider the following assertions:

\smallskip

$(i)$ $K_1^{\gG,\gP}(F)=1$ for every $B$-field $F$;

\smallskip

$(ii)$ $\gG$ satisfies the lifting property (see Definition~\ref{def:lift});

\smallskip

$(iii)$  $(\gG,e)$ is a retract rational $B$--scheme.

\smallskip

\noindent Then the  following implications
$(i) \Longrightarrow (ii) \Longrightarrow (iii)$ hold.
\end{slemma}

\begin{proof} $(i) \Longrightarrow (ii)$.
Let $C$ be a semilocal $B$-ring with
residue fields $F_1,\dots, F_s$.
We have to show that the map $\gG(C) \to \prod_{i=1,\dots,s} \gG(F_i)$
is onto. We are given an element $(g_1, \dots, g_s) \in \prod_{i=1,\dots,s} \gG(F_i)$.
Our assumption implies that there exists a positive integer $d$ such that
$$
g_i = u_{i,1} \, v_{i,1}  \, u_{i,2} \, v_{i,2} \dots u_{i,d} \, v_{i,d}
$$
with $u_{i,j} \in \rad_u(\gP)( F_i)$ (resp.\, $v_{i,j} \in \rad_u(\gP^{-})( F_i)$)
for $i=1,\dots, s$ and $j=1,\dots ,d$.
Since $\rad_u(\gP)(C) \to \prod_{i=1,..,s} \, \rad_u(\gP)( F_i)$ is onto
(and similarly for $\rad_u(\gP^{-})$), we can lift each $(u_{i,j})_{i=1,\dots, s}$
in some $u_j \in \rad_u(\gP)(C)$
(resp.\, $(v_{i,j})_{i=1,\dots, s}$
in  $v_j \in \rad_u(\gP^{-})(C)$).
Thus the product $u_{1} \, v_{1}  \, u_{2} \, v_{2} \dots u_{d} \, v_{d}$
lifts the $g_i$'s.

\smallskip

\noindent $(ii) \Longrightarrow (iii)$. This follows from
Proposition \ref{prop_retract}.
\end{proof}

\begin{sproposition} \label{prop_W_trivial}
Assume that $B$ is a semilocal domain  and that
$\gG$ is semisimple simply connected $B$-group having a strictly proper parabolic $B$-subgroup.
Let $K$ be the fraction field of $B$.  Then the following assertions are equivalent:

\smallskip

$(i)$ $\gG$ satisfies the lifting property;

\smallskip

$(ii)$ $(\gG,1)$ is a retract rational $B$--scheme;

\smallskip

$(iii)$ $\gG$ is $R$--trivial on semilocal rings, that is $\gG(C)/R=1$ for each semilocal $B$-ring $C$;

\smallskip

$(iv)$ $\gG(F)/R=1$ for each $B$-field $F$.

\end{sproposition}

\begin{proof} Let $\gP$ be a strictly proper parabolic subgroup scheme of $\gG$.

\smallskip

\noindent $(i) \Longrightarrow (ii)$. We assume that $\gG$ satisfies the lifting property.
Then Proposition \ref{prop_retract}, $(ii) \Longrightarrow (i)$, shows that
$\gG$ is retract rational over $B$.

\smallskip

\noindent $(ii) \Longrightarrow (iii)$.  If all residue fields are infinite,  this is Proposition \ref{prop_vanish}.
For the general case, we proceed as follows. We need
to show that $\gG(C)/R=1$ for any semilocal $B$-ring $C$. Assume that $C$ has residue fields $\kappa_1,\ldots,\kappa_c$,
such that $\kappa_1, \dots, \kappa_b$ are finite fields
and that $\kappa_{b+1}, \dots , \kappa_c$ are infinite.
Let $(\gU, 1)$ be an open retrocompact subset of $(\gG,1)$
which is a $B$--retract of some open retrocompact of $(\mathbf{A}^N_B,0)$.
We know that $E_\gP(\kappa_i)= \gG(\kappa_i)$ for $i=1,\dots, b$ \cite[1.1.2]{T}.
We consider the open $C$--subscheme $\gV= \gU \, E_\gP(C)$ of $\gG_C$.
Since $E_P(\kappa_i)$ is dense in $G_{\kappa_i}$ for $i=b+1, \dots,c$, we have
$\gV_{\kappa_i}= \gG_{\kappa_i}$ for $b+1=1, \dots, c$.
Since the map $E_\gP(C) \to \prod_{i=1,\dots, b} E_\gP(\kappa_i)$ is onto,
we have  $\gV(C)= \gG(C)$.
Lemma \ref{lem_sorite2}.(1) shows that $\gU(C)/R=1$ so that   $\gV(C)/R=1$. Thus  $\gG(C)/R=1$.

\smallskip

\noindent $(iii) \Longrightarrow (iv)$. Obvious.

\smallskip

\noindent $(iv) \Longrightarrow (i)$.
Since $(iii)$ holds in particular for
any $B$-field $F$, we have $K_1^{\gG,\gP}(F)=1$ for every $B$-field $F$
according to Margaux--Soul\'e isomorphism \cite[Th.\ 3.10]{M1}.
Lemma \ref{lem_W_trivial},
   $(i)  \Longrightarrow (ii)$, implies that $G$ satisfies the lifting property
for any semilocal $B$--algebra $C$.
\end{proof}


This can be refined in the regular case.

\begin{stheorem} \label{thm_vanish}
Assume that $B$ is a semilocal regular domain containing a field $k$ and that
$\gG$ is a semisimple simply connected $B$-group having a strictly proper parabolic $B$-subgroup.
Let $K$ be the fraction field of $B$. Then the following assertions are equivalent:

\smallskip

$(i)$ $\gG$ satisfies the lifting property;

\smallskip

$(i')$  $\gG$ satisfies the lifting property for each $B$-ring $C$ which is a semilocal regular domain
and such that $B$ embeds in $C$;

\smallskip

$(ii)$ $(\gG,1)$ is a retract rational $B$--scheme;

\smallskip

$(iii)$ $\gG$ is $R$--trivial, that is, $\gG(C)/R=1$ for each semilocal $B$-ring $C$;

\smallskip

$(iii')$ $\gG(C)/\r0=1$ for each $B$-ring $C$ which  is a semilocal regular domain;

\smallskip

$(iv)$ $\gG(F)/R=1$ for each $B$-field $F$;

\smallskip

$(v)$ $\gG_K$ is a retract rational $K$-variety.

\smallskip

If, moreover, $\gG$ is of $B$-rank $\ge 2$, then the above statements are also equivalent to the following:

$(iii'')$ $K_1^{\gG}(C)=1$ for each $B$-ring $C$ which is a semilocal regular domain.
\end{stheorem}

Since the preceding statement is rather long,
we extract the following.
\begin{scorollary} \label{cor_vanish}
Assume that $B$ is a semilocal regular domain containing a field $k$ and of fraction field $K$. Let $\gG$ be a semisimple simply connected $B$-group having a strictly proper parabolic $B$-subgroup.
 Then the following assertions are equivalent:

 \smallskip

$(ii)$ $(\gG,1)$ is a retract rational $B$--scheme;

 \smallskip

$(v)$ $\gG_K$ is a retract rational $K$-variety.

\end{scorollary}

\begin{proof}[Proof of Theorem~\ref{thm_vanish}]

Let $\gP$ be a strictly
proper parabolic $B$--subgroup scheme of $\gG$. We
detail only the additional facts from Proposition \ref{prop_W_trivial}
which provides already the equivalences  $(i) \Longleftrightarrow (ii)
\Longleftrightarrow (iii) \Longleftrightarrow (iv)$.

\smallskip

\noindent $(i) \Longrightarrow (i')$. Obvious.

\smallskip

\noindent $(i') \Longrightarrow (ii)$. In the proof of
Proposition \ref{prop_retract}, $(ii) \Longrightarrow (i)$,
we apply the lifting to a semilocalization of $B[t_1, \dots,t_n ]$
which is a regular semilocal domain which contains $B$. So the proof of Proposition \ref{prop_W_trivial},  $(i) \Longleftrightarrow (iii)$,
works so that  $(\gG, 1)$ is retract $B$--rational.

\smallskip

\noindent $(iii) \Longrightarrow (iii')$.
By~\cite[Th.\ 3.10]{M1} combined with~\cite[Th.\ 7.2]{Gi2}
we have $\gG(F)/\r0=\gG(F)/R$ for
each $B$--field $F$. Then the claim follows by Theorem~\ref{thm_KV_R}.

\smallskip

\noindent $(iii') \Longrightarrow (iv)$. Obvious.

\smallskip

\noindent $(iv) \Longrightarrow (v)$.  The assumption implies that the semisimple simply connected
$K$--group $G= \gG_K$ satisfies $G(E)/R=1$ for all
$K$--fields $E$. According to \cite[Cor.\ 5.10]{Gi2}, $G$ is a retract $K$--rational variety.

\smallskip


\noindent $(v)\Longrightarrow (i')$.
Let $C$ be a semilocal regular domain which contains $B$.
It is clear from the proof of  the implication  $(i) \Longrightarrow (ii)$ of Lemma~\ref{lem_W_trivial}
that it is enough to show that $K_1^\gG(F)=1$ for every residue field $F$ of $C$.
Let $\hat C$ be the completion of the localization of $C$ at the prime ideal corresponding to $F$.
Then $\hat C$ is a regular local ring,
and the fraction field $\hat K$ of $\hat C$ is an extension
of $K$. Since $\gG_K$ is retract rational, we have $\gG(\hat K)/R=1$. Then $\gG(\hat C)/R=1$ by Theorem~\ref{thm_KV_R}.
Since $\gG$ is affine and smooth, the map $\gG(\hat C)\to\gG(F)$ is surjective~\cite[Th.\ I.8]{Gruson}.
Hence $\gG(\hat C)/R\to \gG(F)/R$ is surjective and $\gG(F)/R=1$.
According to \cite[Th.\ 7.2]{Gi2}, we have $K_1^\gG(F)=\gG(F)/R=1$.

\smallskip

We assume now that $\gG$ is of $B$-rank $\ge 2$.

\noindent $(iii)\Longrightarrow (iii'')$
Follows from Theorem~\ref{thm_main}.

\smallskip

\noindent $(iii'')\Longrightarrow (iii')$. Obvious.
\end{proof}

\begin{scorollary}\label{cor:rat-rr}
Let $k$ be a field and let $\gX$ be an integral $k$-smooth scheme. Let $\gG$ be a
semisimple simply connected $\gX$-group scheme having a strictly proper parabolic $\gX$-subgroup. If
$\gG_{k(\gX)}$ is $k(\gX)$-retract rational, then $\gG_{k(x)}$ is $k(x)$-retract
rational for every $x\in\gX$.
\end{scorollary}
\begin{proof}
By Theorem~\ref{thm_vanish} for every local ring $\cO_{\gX,x}$ of $\gX$, the group scheme
$\gG_{\cO_{\gX,x}}$ is $\cO_{\gX,x}$-retract rational. Hence $\gG_{k(x)}$ is retract rational as well.
\end{proof}

Since a positive answer to the  Kneser-Tits problem over fields
  is known in a bunch of cases,
we get the following concrete result.

\begin{scorollary} \label{cor_main}
Assume that $B$ is a connected semilocal ring containing a field $k$.
We assume that $\gG$ is semisimple simply connected isotropic $B$-group and that
$\gG_K$ is absolutely almost $K$--simple.
Then $\gG(B)/\r0=1$ in the following cases:

\begin{enumerate}
\item $\gG$ is quasi-split;

\item  the components of the anisotropic kernel of $\gG$
are of rank $\leq 2$;

\item $\gG= \SL_{m}(A)$ where $m \geq 2$ and $A$ is an Azumaya $R$--algebra
of squarefree index;

\item $\gG$ is of type $B_n$, $C_n$;

\item $\gG=\Spin(q)$ for a regular quadratic form  $q$
which is even dimensional (and isotropic);

 \item  $\gG_K=\Spin(A,h)$ where  $A$ is an Azumaya  $R$--algebra
of degree $2$ or  $4$ equipped with an orthogonal involution of first kind
and  $h$ is an isotropic regular hermitian form.

\item $\gG$ if of type  $^{3,6}D_4$ or  $^1E_6$;

\item $\gG$ is of type $^2E_6$ with one of the following Tits indices
$$
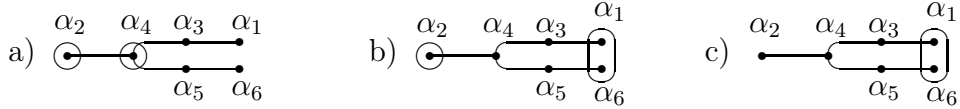
\begin{array}{lll}
\mbox{\rm a)}
\quad
\begin{picture}(80,30)
\put(00,02){\line(1,0){25}}
\put(65,02){\oval(80,10)[l]}
\put(0,02){\circle*{3}}
\put(25,02){\circle*{3}}
\put(65,-3){\circle*{3}}
\put(65,7){\circle*{3}}
\put(45,7){\circle*{3}}
\put(45,-3){\circle*{3}}
\put(0,02){\circle{10}}
\put(25,02){\circle{10}}
\put(-5,12){$\alpha_2$}
\put(20,12){$\alpha_4$}
\put(40,12){$\alpha_3$}
\put(62,12){$\alpha_1$}
\put(62,-13){$\alpha_6$}
\put(40,-13){$\alpha_5$}
\end{picture}
&
\mbox{\qquad \rm b)}
\quad
\begin{picture}(80,30)
\put(00,02){\line(1,0){25}}
\put(65,02){\oval(80,10)[l]}
\put(0,02){\circle*{3}}
\put(25,02){\circle*{3}}
\put(65,-3){\circle*{3}}
\put(65,7){\circle*{3}}
\put(45,7){\circle*{3}}
\put(45,-3){\circle*{3}}
\put(65,02){\oval(10,20)}
\put(0,02){\circle{10}}
\put(-5,11){$\alpha_2$}
\put(20,11){$\alpha_4$}
\put(40,11){$\alpha_3$}
\put(62,17){$\alpha_1$}
\put(62,-15){$\alpha_6$}
\put(40,-13){$\alpha_5$}
\end{picture}
\mbox{\qquad \rm c)}
\quad
\begin{picture}(80,30)
\put(00,02){\line(1,0){25}}
\put(65,02){\oval(80,10)[l]}
\put(0,02){\circle*{3}}
\put(25,02){\circle*{3}}
\put(65,-3){\circle*{3}}
\put(65,7){\circle*{3}}
\put(45,7){\circle*{3}}
\put(45,-3){\circle*{3}}
\put(65,02){\oval(10,20)}
\put(-5,11){$\alpha_2$}
\put(20,11){$\alpha_4$}
\put(40,11){$\alpha_3$}
\put(62,17){$\alpha_1$}
\put(62,-15){$\alpha_6$}
\put(40,-13){$\alpha_5$}
\end{picture}
\end{array}
$$

\vskip3mm

\noindent where for the last case we assume that $6 \in k^\times$.

\bigskip

\item $\gG$ is of type $E_7$ with one of the following Tits indices
$$
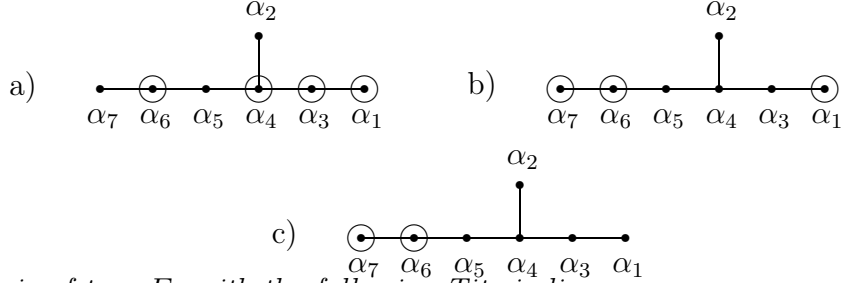
\begin{array}{ll}
\mbox{\rm a)}
\qquad
\begin{picture}(105,45)
\put(00,02){\line(1,0){100}}
\put(60,02){\line(0,1){20}}
\put(00,02){\circle*{3}}
\put(20,02){\circle*{3}}
\put(40,02){\circle*{3}}
\put(60,02){\circle*{3}}
\put(80,02){\circle*{3}}
\put(100,02){\circle*{3}}
\put(60,22){\circle*{3}}
\put(20,02){\circle{10}}
\put(60,02){\circle{10}}
\put(80,02){\circle{10}}
\put(100,02){\circle{10}}
\put(-5,-11){$\alpha_7$}
\put(15,-11){$\alpha_6$}
\put(35,-11){$\alpha_5$}
\put(55,-11){$\alpha_4$}
\put(75,-11){$\alpha_3$}
\put(95,-11){$\alpha_1$}
\put(55,30){$\alpha_2$}
\end{picture}
&
\mbox{\qquad \rm  b)}
\qquad
\begin{picture}(125,45)
\put(00,02){\line(1,0){100}}
\put(60,02){\line(0,1){20}}
\put(00,02){\circle*{3}}
\put(20,02){\circle*{3}}
\put(40,02){\circle*{3}}
\put(60,02){\circle*{3}}
\put(80,02){\circle*{3}}
\put(100,02){\circle*{3}}
\put(60,22){\circle*{3}}
\put(00,02){\circle{10}}
\put(20,02){\circle{10}}
\put(100,02){\circle{10}}
\put(-5,-11){$\alpha_7$}
\put(15,-11){$\alpha_6$}
\put(35,-11){$\alpha_5$}
\put(55,-11){$\alpha_4$}
\put(75,-11){$\alpha_3$}
\put(95,-11){$\alpha_1$}
\put(55,30){$\alpha_2$}
\end{picture}
\end{array}
$$

$$
\begin{array}{l}
\mbox{\qquad \rm  c)}
\qquad
\begin{picture}(125,45)
\put(00,02){\line(1,0){100}}
\put(60,02){\line(0,1){20}}
\put(00,02){\circle*{3}}
\put(20,02){\circle*{3}}
\put(40,02){\circle*{3}}
\put(60,02){\circle*{3}}
\put(80,02){\circle*{3}}
\put(100,02){\circle*{3}}
\put(60,22){\circle*{3}}
\put(20,02){\circle{10}}
\put(-5,-11){$\alpha_7$}
\put(15,-11){$\alpha_6$}
\put(35,-11){$\alpha_5$}
\put(55,-11){$\alpha_4$}
\put(75,-11){$\alpha_3$}
\put(95,-11){$\alpha_1$}
\put(55,30){$\alpha_2$}
\end{picture}
\end{array}
$$

\medskip

\item $\gG$ is of type $E_8$ with the following Tits indices

\vskip-10mm

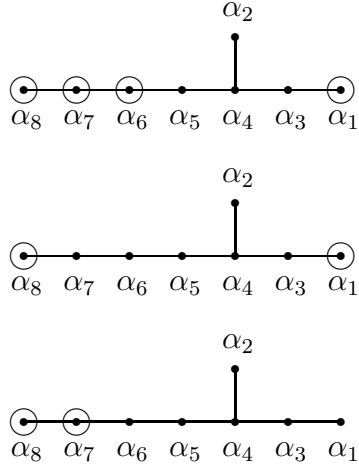
$$
\begin{array}{l}
\begin{picture}(125,55)
\put(00,02){\line(1,0){120}}
\put(80,02){\line(0,1){20}}
\put(00,02){\circle*{3}}
\put(20,02){\circle*{3}}
\put(40,02){\circle*{3}}
\put(60,02){\circle*{3}}
\put(80,02){\circle*{3}}
\put(100,02){\circle*{3}}
\put(120,02){\circle*{3}}
\put(80,22){\circle*{3}}
\put(00,02){\circle{10}}
\put(20,02){\circle{10}}
\put(40,02){\circle{10}}
\put(120,02){\circle{10}}
\put(-5,-11){$\alpha_8$}
\put(15,-11){$\alpha_7$}
\put(35,-11){$\alpha_6$}
\put(55,-11){$\alpha_5$}
\put(75,-11){$\alpha_4$}
\put(95,-11){$\alpha_3$}
\put(115,-11){$\alpha_1$}
\put(75,30){$\alpha_2$}
\end{picture}
\end{array}
$$
$$
\begin{array}{l}
\begin{picture}(125,55)
\put(00,02){\line(1,0){120}}
\put(80,02){\line(0,1){20}}
\put(00,02){\circle*{3}}
\put(20,02){\circle*{3}}
\put(60,02){\circle*{3}}
\put(80,02){\circle*{3}}
\put(100,02){\circle*{3}}
\put(120,02){\circle*{3}}
\put(80,22){\circle*{3}}
\put(00,02){\circle{10}}
\put(40,02){\circle*{3}}
\put(120,02){\circle{10}}
\put(-5,-11){$\alpha_8$}
\put(15,-11){$\alpha_7$}
\put(35,-11){$\alpha_6$}
\put(55,-11){$\alpha_5$}
\put(75,-11){$\alpha_4$}
\put(95,-11){$\alpha_3$}
\put(115,-11){$\alpha_1$}
\put(75,30){$\alpha_2$}
\end{picture}
\end{array}
$$
$$
\begin{array}{l}
\begin{picture}(125,55)
\put(00,02){\line(1,0){120}}
\put(80,02){\line(0,1){20}}
\put(00,02){\circle*{3}}
\put(20,02){\circle*{3}}
\put(20,02){\circle{10}}
\put(60,02){\circle*{3}}
\put(80,02){\circle*{3}}
\put(100,02){\circle*{3}}
\put(120,02){\circle*{3}}
\put(80,22){\circle*{3}}
\put(00,02){\circle{10}}
\put(40,02){\circle*{3}}
\put(-5,-11){$\alpha_8$}
\put(15,-11){$\alpha_7$}
\put(35,-11){$\alpha_6$}
\put(55,-11){$\alpha_5$}
\put(75,-11){$\alpha_4$}
\put(95,-11){$\alpha_3$}
\put(115,-11){$\alpha_1$}
\put(75,30){$\alpha_2$}
\end{picture}
\end{array}
$$
\end{enumerate}

\bigskip

\smallskip

\noindent If furthermore $\gG$ is of $B$--rank $\geq 2$, then
 $K_1^\gG(B)=1$.
\end{scorollary}

\medskip

\begin{proof}
We assume firstly that the $k$--algebra $B$ is a regular domain
so that  the statement is a case by case  application of
Theorem \ref{thm_vanish}, $(v) \Longrightarrow (iii')$ or  $(iv) \Longrightarrow (iii')$.
Almost all results of retract rationality over $K$ are quoted  in \cite[Th.\ 6.1]{Gi2} excepted the following  cases.

 \smallskip

\noindent{\it Third outer $E_6$ case, i.e. $E_{6,1}^{29}$.}
This is due to Garibaldi, see \cite[Th.\ 6.2]{Gi2}.

\smallskip

\noindent{\it Second $E_8$ case, i.e. $E_{8,2}^{66}$.} This is
a result by Parimala-Tignol-Weiss \cite[\S 3]{PTW}.

 \smallskip

\noindent{\it Third $E_7$ (resp.\, $E_8$) case, i.e.
$E_{7,1}^{78}$ (resp.\ $E_{8,2}^{78}$).} The $R$--triviality over fields is
a result by Alsaody-Chernousov-Pianzola \cite[Th.\ 8.1]{ACP} and  by Thakur
\cite[Th.\ 4.2  and Cor.\ 4.3]{Thakur} independently.

\smallskip

To deduce the general case where $B$ is not necessarily regular,
we use Hoobler's trick, see~\cite[proof of theorem 2]{Ho} or~\cite[p. 109]{Ka}.
There exists a henselian pair $(C,I)$ such that $C/I=B$
and $C= \limind C_\alpha$ where each $C_\alpha$ is  a semilocalization of
 an affine $k$--space.

We are given a minimal $B$--parabolic subgroup $\gP$ of $\gG$.
Denote by $G_0$ the split Chevalley $k$--form of $\gG$.
Then $(\gG,\gP)$ is a form of $(G_0,P_0)$. Since $\Aut(G_0,P_0)$ is a smooth affine
$k$--group~\cite[lemme 5.1.2]{Gi4}, the map
$$H^1(C,\Aut(G_0,P_0)) \to H^1(C/I,\Aut(G_0,P_0))
$$
is bijective \cite[th.\1]{Strano}.
This implies that there exists a couple $(\widetilde \gG, \widetilde \gP)$ over $C$ such that
 $(\widetilde \gG, \widetilde \gP) \times_C B = (\gG, \gP)$.
 Since $\widetilde \gG$ is smooth over $C$, the map
${\widetilde \gG}(C) \to \gG(B)$ is onto and so is  ${\widetilde \gG}(C)/\r0 \to \gG(B)/\r0$.

On the other hand, we have $H^1(C,\Aut(G_0))= \limind H^1(C_\alpha,\Aut(G_0))$
\cite[VI$_B$.10.16]{SGA3}. It follows that there exists $\alpha_0$ and a
couple $(\gG_{\alpha_0}, \gP_{\alpha_0})$ such that $(\gG_{\alpha_0}, \gP_{\alpha_0})
\times_{C_{\alpha_0}} C=
(\widetilde \gG, \widetilde \gP)$.
We have  $\widetilde \gG(C)= \limind_{\alpha \geq \alpha_0}
 \gG_{\alpha_0}(C_\alpha)$.
But $\gG_{\alpha_0}(C_\alpha)/\r0=1$ by the
regular case of the theorem. Since   $\limind_{\alpha \geq \alpha_0}
 \gG_{\alpha_0}(C_\alpha) / \r0 \to  \widetilde \gG(C)
/ \r0 $ is onto, we conclude that
$\widetilde \gG(C)/\r0=1$.

If furthermore $\gG$ is of $B$--rank $\geq 2$, then we have similarly
$K_1^{\gG_{\alpha_0}}(C_\alpha) =1$ from the regular case and
a composite of surjective maps
$\limind_{\alpha \geq \alpha_0}
 K_1^{\gG_{\alpha_0}}(C_\alpha)  \to \hskip-3mm \to  K_1^{\widetilde \gG}(C)
 \to \hskip-3mm \to K_1^\gG(B)$. Thus  $K_1^\gG(B)=1$.
\end{proof}


\section{Behaviour for henselian pairs}

\smallskip

We address the following question with respect to  a henselian pair
$(B,I)$ \cite[15.11]{St}; this concerns, for example, the case of
a nilpotent ideal.

\begin{squestion} \label{question_henselian}
Let $\gG$ be a reductive $B$--group scheme. Is
the map $\gG(B)/R \to  \gG(B/I)/R$ an isomorphism?
\end{squestion}

Note that since $\gG$ is affine and smooth over $B$, the map $\gG(B)\to\gG(B/I)$ is surjective~\cite[Th.\ I.8]{Gruson},
and hence the map of
$R$-equivalence class groups is surjective.


\subsection{The torus case}


\begin{slemma}\label{lem_tor_isotrivial} Assume that $B/I$ is a normal noetherian domain. Let $\gT$ be a
$B$-torus. Then  $\gT$ is isotrivial.
\end{slemma}

\begin{proof} Since $B/I$ is a normal noetherian domain, $\gT_{B/I}$ is isotrivial \cite[X.5.16]{SGA3}, that is, there
exists a finite \'etale cover $C_0$ of $B/I$ such  that $\gT_{C_0} \cong \GG_{m,C_0}^r$.
Since $(B,I)$ is a henselian pair, $C_0$ lifts to a finite \'etale cover $C$ of $B$
\cite[Tag 09ZL]{St} and furthermore $(C,I C)$ is a henselian pair ({\it ibid}, Tag 09XK).
According to \cite[Prop.\ 6.1.3.(a)]{Ces-Prob}, the isomorphism $\gT_{C_0} \cong \GG_{m,C_0}^r$
lifts to an isomorphism $\gT_{C} \cong \GG_{m,C}^r$
so that $\gT \times_B C$ is split. Thus $\gT$ is isotrivial.
\end{proof}

A first evidence for the question \ref{question_henselian} is the following fact.

\begin{slemma}\label{lem_torus_pair} Let $\gT$ be a
$B$-torus. Assume that $B/I$ is a regular domain. Then  the map
$\gT(B)/R \to  \gT(B/I)/R$ is an isomorphism.
\end{slemma}

\begin{proof} By definition, the
regular domain  $B/I$ is noetherian and also is normal
\cite[Tags 00OD, 0567]{St}. The $B$--torus $\gT$ is isotrivial according to Lemma \ref{lem_tor_isotrivial}.
Let $1 \to \gS \to \gQ \xrightarrow{\pi}  \gT \to 1$ be a flasque resolution.
We have a commutative diagram of exact sequences
$$
\xymatrix@C=20pt{
0  \ar[r] & \gT(B)/\pi(\gQ(B)) \ar[r] \ar[d] & H^1(B,\gS) \ar[r] \ar[d] &  H^1(B,\gQ) \ar[d] \\
0  \ar[r] & \gT(B/I)/\pi(\gQ(B/I)) \ar[r] & H^1(B/I,\gS) \ar[r] &  H^1(B/I,\gQ). \\
}
$$
According to \cite[Th.\ 1]{Strano}, the maps $H^1(B,\gS) \to H^1(B/I,\gS)$
and $H^1(B,\gQ) \to H^1(B/I,\gQ)$ are isomorphisms. By diagram chase
we conclude that the map $\gT(B)/\pi(\gQ(B)) \to \gT(B/I)/\pi(\gQ(B/I))$ is an isomorphism.

Example \ref{ex_R_triv}(3)
 states that $R\gQ(B)=\gQ(B)$,
hence the inclusion $\pi(\gQ(B)) \subseteq R \gT(B)$.
It follows that  we deal with a  surjection $\gT(B)/\pi(\gQ(B))\to \gT(B)/R$.
Summarizing we have a commutative diagram
\[\xymatrix@1{
\gT(B)/ \pi(\gQ(B)) \ar[d]^\wr \ar@{->>}[r] & \gT(B)/R \ar[d] \\
\gT(B/I)/ \pi(\gQ(B/I)) \ar[r]^{\quad \sim} & \gT(B/I)/R
}\]
where the bottom isomorphism is provided by Proposition \ref{prop_torus1}.
By diagram chase we conclude that the top map  $\gT(B)/\pi(\gQ(B))=\gT(B)/R$ is an isomorphism and so is the
the map $\gT(B)/R \to \gT(B/I)/R$.
\end{proof}



\subsection{A generalization}


Using the case of tori, we obtain the following partial result for $R$-equivalence of arbitrary reductive groups.
We do it by generalizing an argument of Raghunathan~\cite[\S 1]{R}.

\begin{slemma} \label{lem_red_tor} We assume that $B/I$ is a regular domain.
Let $\gG$ be a reductive $B$--group scheme
admitting  $B$-subtori $\gT_1, \dots, \gT_n$ such that
$\Lie(\gG)(B)$ is generated as $B$--module by the $\Lie(\gT_i)(B)$'s.
Then $\ker\bigl(\gG(B) \to \gG(B/I) \bigr) \subseteq R\gG(B)$.
\end{slemma}

\begin{proof} We consider the map of $B$--schemes

\[\xymatrix@1{
f: \gT_1 \times_B \dots \times_B \gT_n & \to & \gG \\
(t_1,\dots, t_n) & \mapsto &   t_1 \dots t_n.
}\]
For each maximal ideal $\gm$ of $B$,
the differential at $1_{B/\gm}$ is
\[\xymatrix@1{
df_{1_k}: \Lie(\gT_1)(B/\gm) \oplus \dots \oplus \Lie(\gT_n)(B/\gm) & \to
& \Lie(G)(B /\gm) \\
(X_1,\dots, X_n) & \mapsto &   {X_1}+  \dots + {X_n}
}\]
which is onto by construction. It follows that the map $f$ is smooth
at $1_{B/\gm}$ for each  maximal ideal of $\gm$.
The Jacobian criterion shows that $f$ is smooth in the neighborhood of the unit section
of $\gT_1 \times_B \dots \times_B \gT_n$.
The Hensel lemma \cite[Th.\ I.8]{Gruson} (see also
\cite[Prop.\ 6.1.1]{Ces-Prob}) shows that the induced map
$$
\ker\bigl( (\gT_1 \times_B \dots \times_B\gT_n)(B) \to (\gT_1 \times_B \dots \times_B\gT_n)(B/I) \bigr)  \to  \ker\bigl( \gG(B) \to \gG(B/I) \bigr)
$$
is surjective.
The torus case Lemma \ref{lem_torus_pair}  shows  that $\ker\bigl( \gT_i(B) \to \gT_i(B/I) \bigr)
\subseteq R\gT_i(B)$ for $i=1,...,n$. Thus
$\ker\bigl( \gG(B) \to \gG(B/I) \bigr) \subseteq
R\gG(B)$.
\end{proof}

Together with Lemma \ref{lem_unirational}.(1), we get the following fact.

\begin{scorollary}\label{cor_red_tor}
 Let $R$ be a semilocal ring with infinite residue fields and let $\gG$
 be a reductive $R$--group scheme assumed $R$-linear. Let $(B,J)$ be a henselian pair
 where $B$ is an $R$--algebra such that $B/J$ is a regular domain. Then
$\ker\bigl( \gG(B) \to \gG(B/J) \bigr) \subseteq R\gG(B)$.
\end{scorollary}


\subsection{The semisimple case}


We continue with the henselian pair $(B,I)$.
One evidence for answering  positively the
question \ref{question_henselian} is the case of the group $\SL_N(\cA)$ for an Azumaya
$B$--algebra $\cA$ of degree invertible in $B^\times$ for $N>>0$
since Hazrat has proven that the map $\SK_1(\cA) \to \SL_1(\cA/I)$
is an isomorphism, if $B$ is semilocal~\cite{Hazrat}.
Firstly we make a variation on \cite[\S 3.4]{GPS}.

\begin{slemma}\label{lem_root2}
Let $F$ be a field and let  $G$ be a reductive $F$-group.
Let $P$ be a strictly proper parabolic $F$--subgroup and let $P^{-}$ be an opposite
parabolic subgroup to $P$. We put $U= \rad_u(P)$ and $U^{-}=\rad_u(P^{-})$  and consider
the subgroup $E_P(F)$ (resp.\ $E_P(F[\epsilon])$) of
$G(F)$  (resp.\ $G(F[\epsilon])$) as defined
 in \S \ref{subsec_non_stable}.

We consider the following commutative diagram
\[
\xymatrix@C=30pt{
 0   \ar[r]& \Lie(G)  \ar[r]  &  G(F[\epsilon]) \ar[r] & G(F) \ar[r] & 1 \\
 & & E_P(F[\epsilon]) \ar[r] \incl[u] & E_P(F) \ar[r] \incl[u] & 1
}
\]
and define  $V_P= \ker( E_P(F[\epsilon]) \to  E_P(F) ) \subseteq \Lie(G)$.

\smallskip

(1) The $F$--subspace  $V_P$ is an ideal of $\Lie(G)$
which is $G(F)$-stable. We have \break $V_P= E_P(F) .\Lie(U)+ E_P(F) .\Lie(U^{-})$.

\smallskip

(2) If $G$ is semisimple simply connected, we have $V_P= \Lie(G)$.
\end{slemma}

\begin{proof}(1)
It follows from~\cite[XXVI.5.1]{SGA3} that $E_P(F[\epsilon])$ (resp.\  $E_P(F)$)  is a  normal
subgroup of $G(F[\epsilon])$ (resp.\ $G(F)$). It implies that
$V_P$ is a Lie subalgebra of $\Lie(G)$ which is furthermore $G(F)$--equivariant.
Since  $\Lie(U), \Lie(U^{-})$ are contained in $V_P$, it follows that
$E_P(F) .\Lie(U)+ E_P(F). \Lie(U^{-}) \subseteq V_P$.
Conversely, we are given an element $v \in V_P$. It is of the shape
$v= u_1 u_2  \dots u_{2n}$ with $u_{2i+1} \in U(F[\epsilon])$ and $u_{2i} \in U^{-}(F[\epsilon])$.
We have  a decomposition $v= v_1 (g_2 v_2 g_2^{-1}) \dots   (g_{2n} v_{2n} g_{2n}^{-1})$
with $v_{2i+1} \in \Lie(U)$, $v_{2i} \in  \Lie(U^{-})$ and
$g_1, \dots, g_{2n} \in E_P(F)$. We have proven that $v$ belongs  to
 $E_P(F) .\Lie(U)+ E_P(F) .\Lie(U^{-})$.

\smallskip

\noindent (2)
Without loss of generality we can assume that $G$ is almost absolutely $F$--simple.
If $F$ is infinite, we have that $E_P(F) .\Lie(U)=\Lie(G)$ according to \cite[lemma 3.3.(3)]{GPS}
so a fortiori $V_P=\Lie(G)$.
We can then assume that $F$ is finite so that $G$ is quasi-split (Lang, see the proof of Cor.\ \ref{cor:polynomial_local}).
We have then  $E_P(F) = G(F)$  according to \cite[1.1.2]{T}.
If $G$ is split, the statement is \cite[lemma 3.3.(1)]{GPS}.
It remains to deal with the quasi-split non split case, it implies that $G$ is of outer  type $A$, $D$ or $E_6$.
In particular, all geometrical roots have same length and $G$ is not of type $A_1$. If
$G$ has $F$--rank 1, $P$ is a Borel subgroup of $G$
and if $G$ has $F$--rank $\geq 2$, we can replace
$P$ by a Borel subgroup $B$ in view of \cite[remark 2 after Theorem 1]{PS}. In both cases, we can then assume
that $P=B$ is a Borel subgroup.
Let $T$ be maximal torus of $B$,
we recall  the decomposition $\Lie(G)= \Lie(T) \oplus \Lie(U) \oplus \Lie(U^{-})$.

We consider the ideal $V_P \otimes_F F_s$ of $\Lie(G) \otimes_F F_s$.
According to \cite[prop.\ 2.6.a]{Hg}, $V_P \otimes_F F_s$ is central or contains $\Lie(T) \otimes_F F_s$.
Since $V_P$ is not central, we conclude that
$\Lie(T) \otimes_F F_s\subset V_P \otimes_F F_s$. It follows that $\Lie(T)  \subseteq V_P$, since
these are linear $F$-spaces, and thus satisfy
$V_P=(V_P \otimes_F F_s)\cap\Lie(G)$ and $\Lie(T)=(\Lie(T) \otimes_F F_s)\cap\Lie(G)$.
Since $\Lie(U), \Lie(U^{-}) \subset V_P$, we conclude that $V_P=\Lie(G)$.
\end{proof}

\begin{sproposition}\label{prop:hens-pair}
Let $R$ be a semilocal ring and let $\gG$ be a semisimple  group scheme over $B$,
such that its simply connected cover morphism $f: \gG^{sc} \to \gG$ is smooth.
We assume that $\gG$ has a strictly proper parabolic $R$-subgroup $\gP$.
Let $(B,I)$ be a henselian pair where
 $B$ is an $R$--algebra.
Then the map  $K_1^{\gG,\gP}(B)\to K_1^{\gG,\gP}(B/I)$ is an isomorphism.
\end{sproposition}

\begin{proof}
 Let $\gP^{-}$ be a parabolic $R$-subgroup of $\gG$, opposite to $\gP$. Let $\gU= \rad(\gP)$,  $\gU^{-}= \rad(\gP^{-})$.

Since $\gG$ is affine smooth, the map $\gG(B) \to \gG(B/I)$ is surjective
according to  the generalization of Hensel's lemma to  henselian pairs
\cite[Th.\ I.8]{Gruson}, hence $K_1^{\gG,\gP}(B)\to K_1^{\gG,\gP}(B/I)$
is onto. To show that it is injective, it is enough to prove that
$\ker(\gG(B)\to \gG(B/I)) \le E_{\gP}(B)$, since $E_{\gP}(B)$ surjects onto $E_\gP(B/I)$.

Combining the lifting method of   \cite[lemma 3.5]{GPS} and Lemma \ref{lem_root2},
 there exist $g_1, \dots, g_{2m} \in E_{\gP}(R)$
such that the  product map
$$
h: (\gU \times \gU^{-})^m   \to \gG, \enskip (u_1,\dots, u_{2m}) \mapsto  {^{g_1}\!u_1} \dots
{^{g_{2m}}\!u_{2m}}
$$
is smooth  at each  $(1,...,1)_{\kappa_i}$.
Then $h$ is smooth in the neighborhood of the origin of $(\gU \times \gU^{-})^m$.
Hensel's lemma yields $\ker(\gG(B)\to \gG(B/I))\le E_{\gP}(B)$.
\end{proof}

\section{Specialization for $R$--equivalence}


\subsection{The case of tori}


Let $A$ be a henselian local ring with maximal ideal $m$ and  residue field $k$.
As a special case of Lemma \ref{lem_tor_isotrivial}, any  $A$-torus is isotrivial.

\begin{sproposition} \label{prop_torus2}
Let  $A$ be  a local ring  of residue field $k$.
Let $\gT$ be an $A$-torus and put $T= \gT \times_A k$. Then

\smallskip

\noindent (1) If $A$ is henselian, then the natural map $\gT(A)/R \to  \gT(k)/R$ is an isomorphism.
In particular we have  $\ker\bigl( \gT(A) \to T(k) \bigr) \subseteq R\gT(A)$.

\smallskip

\noindent (2) If $A$ is regular and
$K$ denotes the fraction field of $A$, then the natural map $\gT(A)/R \lgr \gT(K)/R$ is an isomorphism.
\end{sproposition}

\begin{proof}
\noindent (1) Let $m$ be the maximal ideal of $A$. Since $A$ is henselian, $(A,m)$ is a Henselian pair. Then, since
$A/m=k$ is a regular domain,
by Lemma~\ref{lem_torus_pair} $\gT(A)/ R \to T(k)/R$ is an isomorphism.

\smallskip

\noindent (2) We consider  a flasque resolution
$$
1 \to \gS \to \gQ \xrightarrow{\pi} \gT \to 1.
$$
According to Proposition \ref{prop_torus1},
we have isomorphisms
$$
\gT(A)/ \pi( \gQ(A)) \simlgr  \gT(A)/R \simlgr H^1(A,\gS)
$$ and
$\gT(K)/ \pi( \gQ(K)) \simlgr  \gT(K)/R \simlgr H^1(K,\gS)$.
Since $\gS$ is flasque, the restriction map
$H^1(A,\gS) \to H^1(K,\gS)$
is  surjective \cite[Th.\ 2.2]{CTS2} and is injective ({\it ibid}, Th.\ 4.1).
Thus the map  $\gT(A)/R \to  \gT(K)/R$ is an isomorphism.
\end{proof}

\begin{scorollary} \label{cor_torus2}
We assume that the henselian local ring $A$ is
regular with residue field $k$ and
fraction field $K$. For any $A$-torus $\gT$
we have two isomorphisms
$$
T(k)/R \xleftarrow{\ \sim\ }  \gT(A)/R  \xrightarrow{\ \sim\ } \gT(K)/R.
$$
\end{scorollary}


\subsection{Reduction to the anisotropic case}\label{subsec_aniso}

We come back to the setting of the introduction
where  $A$ is  a  henselian  local domain  of residue field $k$
and fraction field $K$.

Let $\gG$ be a reductive $A$--group scheme.
Let $\gP$ be a parabolic $A$--subgroup of $\gG$
and let $\gL$ be a Levi subgroup of $\gP$.
We know that $\gL= Z_\gG(\gS)$ where $\gS$ is the maximal central  $A$--split subtorus $\gS$ of
$\gL$ \cite[XXVI]{SGA3}. We put $G=\gG\times_A k$, $P= \gP \times_A k$ and define similarly
$L$ and $S$. According to Corollary~\ref{cor_parabolic}, we have the following commutative
diagram where horizontal maps are isomorphisms
\begin{equation}\label{diag_big}
 \xymatrix{
\gG(K)/R  & \ar[l]_\sim  \gL(K)/R \ar[r]^{\sim\quad} &   \bigl( \gL/\gS\bigr)(K)/R \\
 \gG(A)/R \ar[d] \ar[u] & \ar[l]_\sim  \gL(A)/R \ar[r]^{\sim\quad}  \ar[d] \ar[u]&
 \ar[d] \ar[u] \bigl( \gL/\gS\bigr)(A)/R \\
 G(k)/R  & \ar[l]_\sim  L(k)/R \ar[r]^{\sim\quad} &   \bigl( L/S\bigr)(k)/R .
}
\end{equation}
By diagram chase, we get the following facts.

\begin{slemma} \label{lem_aniso}
(1) If $\bigl(\gL/\gS\bigr)(A) /R \to \bigl(L/S\bigr)(k) /R$
is injective, then the two maps $\gG(A)/R \to G(k)/R$ and $\gL(A)/R \to L(k)/R$ are isomorphisms.

\smallskip

\noindent (2) If $\bigl(\gL/\gS\bigr)(A) /R \to \bigl(\gL/\gS\bigr)(K) /R$
is injective (resp.\ surjective, resp.\  isomorphism), then $\gG(A)/R \to \gG(K)/R$ is
injective  (resp.\ surjective, resp.\ an isomorphism)
and the map $\gL(A)/R \to \gL(K)/R$ is  injective (resp.\ surjective, resp.\ an isomorphism).
\end{slemma}
\begin{proof}
Since $\gG$ and $\gL$ are smooth $A$-schemes and $A$ is henselian, the maps $\gG(A)/R\to G(k)/R$
and $\gL(A)/R \to L(k)/R$ are surjective. The rest follows from Corollary~\ref{cor_parabolic}.
\end{proof}

It follows that the specialization  problem reduces to
the case of $\gL$ and even to  $\gL/\gS$.
In particular, if $\gP$ is minimal, then $\gL/\gS$ is anisotropic.


\subsection{The lifting map}


\begin{slemma}\label{lem:hens-local-RG}
Let $A$ be a henselian local ring with residue field $k$ and let $\gG$ be a reductive $A$-group.
Then $\ker\bigl( \gG(A) \to G(k) \bigr) \subseteq R\gG(A)$.
\end{slemma}
\begin{proof}
If $k$ is infinite, the claim follows from Corollary~\ref{cor_red_tor}. If $k$ is finite, then
  Lang's theorem \cite[Th.\ 2]{Lg}
shows that  $G$ admits a Borel $k$-subgroup.
Since the $A$-scheme of Borel subgroups of $\gG$ is smooth,
the Hensel's lemma shows that $\gG$ admits an $A$-Borel
subgroup scheme. It follows that $\gG$ is quasi-split by~\cite[XXIV.3.9.1]{SGA3}.
Then one has $\gG(A)/R=\gG(k)/R=1$ by Gauss decomposition~\cite[XXVI.5.1]{SGA3}
combined with the fact that quasi-split tori over $A$ and $k$ are $R$-trivial.
\end{proof}

The above lemma shows that the map $\gG(A) \to \gG(A)/R$ factorizes through
$G(k)$, i.e.\ defines a surjective homomorphism  $\phi:  G(k) \to \gG(A)/R$.
One way to prove that the map $\gG(A)/R \to G(k)/R$ is an isomorphism
would be to show that $\phi$ factorizes
through $G(k)/R$, that is to complete the following diagram

\begin{equation}\label{eq_problem}
\xymatrix{
G(k) \ar[r]^\phi \ar[d] & \gG(A)/R \ar[r] &  1. \\
G(k)/R  \ar@{.>}[ur].
}
\end{equation}

\noindent The dotted map is called (when it exists) \emph{the lifting map}. In what follows we
prove the existence of the lifting map in two different cases.

\begin{sproposition}\label{prop:hens-section}
Let $A$ be a henselian local ring with the residue field $k$, and let
$\gG$ be a reductive group over $A$. Assume that $A$ is equicharacteristic, i.e. $A$ contains a field.
Then $\gG(A)/R\to G(k)/R$ is an isomorphism.
\end{sproposition}
\begin{proof}
By Lemma~\ref{lem:hens-section} $A$ is a filtered direct limit
of henselian local rings $A_i$ such that the map from $A_i$ to its residue field admits a section. Since $\gG$
is finitely presented over $A$, and the functor $\gG(-)/R$
commutes with filtered direct limits by Lemma~\ref{lem_limit}, we can assume from the start that $A\to k$ admits a section.

We have $\ker(\gG(A)\to G(k))\subseteq R\gG(A)$ by Lemma~\ref{lem:hens-local-RG}. Since $A\to k$ admits a section,
the map $R\gG(A)\to RG(k)$ is surjective. These two statements together imply that $\gG(A)/R\to G(k)/R$ is injective.
The surjectivity is obvious.
\end{proof}



\begin{stheorem}\label{thm:hens-isotr}
Let $A$ be a henselian local ring with residue field $k$,
let $\gG$ be a semisimple  group scheme over $A$,
such that its simply connected cover morphism $f: \gG^{sc} \to \gG$ is smooth.
We assume that $\gG$ has a strictly proper parabolic subgroup $\gP$.

\smallskip

\noindent (1) The map  $K_1^{\gG,\gP}(A)\to  K_1^{G,P}(k)$ is an isomorphism.

\smallskip

\noindent (2)  If $\gG= \gG^{sc}$,
we have a square of isomorphisms
$$
\xymatrix{
K_1^{\gG,\gP}(A)\ar[d]^\wr   \ar[r]^\sim  & K_1^{G,P}(k) \ar[d]^\wr   \\
\gG(A)/R   \ar[r]^\sim  & G(k)/R .
}
$$

\smallskip

\noindent (3) Assume furthermore that  $A$ is a domain with fraction field $K$.
There is a natural lifting map $K_1^{G,P}(k) \to K_1^{\gG,\gP}(A)\to K_1^{\gG,\gP}(K)$.
\end{stheorem}

\begin{proof}
(1) This is a special case of Proposition \ref{prop:hens-pair}.

\smallskip

\noindent (2) If $\gG= \gG^{sc}$,  we have the following commutative diagram
$$
\xymatrix{
K_1^{\gG,\gP}(A)\ar[d]  \ar[r]^\sim  & K_1^{G,P}(k) \ar[d]^\wr   \\
\gG(A)/R   \ar[r]  & G(k)/R
}
$$
where the right vertical  isomorphism is \cite[Th.\ 7.2]{Gi2}.
Since the left vertical map is onto,  a diagram chase
shows that all maps are isomorphisms.

\smallskip

\noindent (3) It is a straightforward consequence.
\end{proof}

\smallskip


\subsection{The case of DVRs}\label{subsec_DVR}


Assume that $A$ is a henselian DVR and $\gG$ is a reductive group over $A$. We remind the reader of
the existence of a specialization map
$$
\varphi: \gG(K)/R \to G(k)/R
$$
which is characterized by the property $\varphi([g])= [\overline{g}]$
for all $g \in G(A)$ \cite[Th.\ 0.2]{Gi1}.  In other words we have a
commutative diagram

\begin{equation}\label{eq_DVR}
\xymatrix{
\gG(A)/R \ar[r] \ar[d] & \gG(K)/R \ar[dl]^{\varphi} \\
G(k)/R .
}
\end{equation}

\noindent This is based on the existence of a specialization map $\gX(A)/R \to \gX(k)/R$
for a projective $A$--scheme $\gX$ due to Koll\'ar \cite{Ko} and Madore \cite{Ma},
see also \cite[Th.\ 6.1]{CT}.

\begin{sremark}\label{rem_char2}{\rm
The quoted reference~\cite[Th.\ 0.2]{Gi1} requires the assumption that $k$
is not of characteristic $2$. This assumption
occurs only in the de Concini--Procesi construction of the wonderful compactification of an adjoint
semisimple $A$-group scheme. It is folklore
that we can get rid of this assumption  by a refinement of
\cite[Th.\ 3.13]{CS}. By descent, the relevant case is that of adjoint Chevalley
 groups over $\ZZ$ which is used for example
 in \cite{STBT}. Note also that in the field case, there is a
 construction of the wonderful compactification in
 \cite[\S 6.1]{BK}.
 }
\end{sremark}

\begin{sremark}\label{rem_CT}{\rm
The existence of the specialization map in the reductive case over
a DVR has been established by another method by
Colliot-Th\'el\`ene, Harbater, Hartmann, Krashen, Parimala, and Suresh
which involves simpler compactifications \cite[Th.\ A.10]{CTHHKPS}.
 It follows from Lemma~\ref{lem_dvr2} below that the
two specialization maps coincide. See also Remark~\ref{rem_CT2}.
}
\end{sremark}

\begin{slemma}\label{lem_dvr2}
Let $A$ be a henselian DVR. For any reductive group $\gG$ over $A$ the map $\gG(A)/R\to\gG(K)/R$ is
surjective.
\end{slemma}
\begin{proof}
 \noindent{\it First  case: $G=\gG_k$ is irreducible (that
is $G$ is the only parabolic $k$--subgroup of $G$).}
Let $S$ be the maximal
central split subtorus of $\gG_k$. It lifts to
a central split subtorus $\gS$ of $\gG$ \cite[XI]{SGA3}.
Since $G/S$ is anisotropic, we have $(\gG/\gS)(A)=(\gG/\gS)(K)$~\cite{BT2}.
Hilbert 90 theorem yields $\gG(A)/\gS(A)=\gG(K)/\gS(K)$ hence
a decomposition $\gG(K) = \gS(K)  \, \gG(A)$.
Since $R\gS(K)= \gS(K)$, we conclude that  $\gG(K) = \gG(A) \,  R\gG(K)$.

\smallskip

\noindent{\it General  case.}
Let $\gP$ be a minimal parabolic
$A$--subgroup of $\gG$. Let $\gP^{-}$ be an opposite
parabolic $A$--subgroup scheme to $\gP$. Then the Levi subgroup
$\gL=\gP \cap \gP^{-}$ is such that $L =\gL_k$ is irreducible.
Let $\gS$ be the maximal central split subtorus of $\gL$.
The first case shows that $(\gL/\gS)(A)/R\to (\gL/\gS)(K)/R$ is surjective.
By Lemma~\ref{lem_aniso} this implies
the surjectivity of $\gG(A)/R\to\gG(K)/R$.
\end{proof}

\begin{sproposition}\label{prop:hens-dvr-isotr}
Let $A$ be a henselian DVR. Let $k$ be the residue field of $A$ and
let $K$ be the fraction field of $A$. Let    $\gG$ be a semisimple simply connected
$A$--group scheme having a strictly proper parabolic $A$-subgroup $\gP$. Then we have the following commutative diagram
of isomorphisms
\begin{equation}\label{diag_hDVR}
 \xymatrix{
K_1^{\gG,\gP}(k)\ar[d]^\wr   & \ar[l]_\sim K_1^{\gG,\gP}(A)  \ar[r]^{\sim}\ar[d]^{\wr} & K_1^{\gG,\gP}(K)\ar[d]^\wr \\
  \gG(k)/R & \ar[l]_\sim  \gG(A)/R \ar[r]^{\sim}&    \gG(K)/R \\
}
\end{equation}



\end{sproposition}

\begin{proof}
By~Theorem~\ref{thm:hens-isotr} we have that $\gG(A)/R\to \gG(k)/R$ is an isomorphism. Then it follows
from the existence of specialization map~\eqref{eq_DVR} that $\gG(A)/R\to \gG(K)/R$ is injective.
By Lemma~\ref{lem_dvr2} the map $\gG(A)/R\to \gG(K)/R$ is surjective.

Consider the commutative diagram
$$
\xymatrix{
K_1^{\gG,\gP}(A)  \ar[r]  \ar[d] &  \gG(A)/R  \ar[d]  \ar[r]& 1\\
K_1^{\gG,\gP}(k)  \ar[r]  & \gG(k)/R .\\
}
$$
The bottom horizontal map is an isomorphism as we have used several times \cite{Gi2}
and the left vertical map is an isomorphism in view of Proposition \ref{prop:hens-pair}.
It follows that $K_1^{\gG,\gP}(A)   \to \gG(A)/R$ is injective and then
an isomorphism. The remaining isomorphisms follow immediately.
\end{proof}

\begin{sremark}{\rm
The surjectivity of the map $K_1^{\gG,\gP}(A)\to K_1^{\gG,\gP}(K)$ was previously proved in~\cite[lemme 4.5.1]{Gi2}.
Note that it is does not hold for $A=k[[t]]$ and $\gG=\GL_n$ or $\PGL_n$,
so seems specific to the semisimple simply connected case.
 }
\end{sremark}


\subsection{Specialization in the equicharacteristic case}\label{subsec_const}



Assume that $A$ is a complete regular local ring containing a prime field $k_0$ and let
$K$ be its fraction field.
According to \cite[vol.\ 20, Thm.\ 19.6.4 page 102]{EGA4},
$A$ is $k_0$-isomorphic  (non-canonically)
to a formal series ring $k[[t_1,\dots ,t_d]]$, where $k$ is the residue field of $A$.






Let $\gG$ be a reductive $A$-group scheme.
There exists a unique
reductive $k$--group $G$ such that $\gG \times_A  k[[t_1,\dots ,t_d]]
\cong G \times_k  k[[t_1,\dots ,t_d]]$ (see the proof of Corollary~\ref{cor:cons} below).
Since the fraction field $K= k((t_1,\dots, t_d))$ of $A$ is a (proper) subfield of the iterated Laurent power
series field $k((t_1)) \dots ((t_d))$, and
$$
G(k)/R\to G\bigl( k((t_1)) \dots ((t_d))\bigr)/R
$$ is an isomorphism~\cite[Cor.\ 0.3]{Gi1}, we can define a
specialization map $G(K)/R\to G(k)/R$ inductively,
$$
sp: G(K)/R \to G\bigl( k((t_1,\dots, t_d)) \bigr)/R \to G\bigl( k((t_1,\dots, t_{d-1})) \bigr)/R
\to \dots \to G(k)/R.
$$
However, it is unclear whether
this map does not depend of the choice of coordinates $t_1,\ldots,t_d$.
The following theorem solves this problem.

\begin{stheorem}\label{thm:const}
Let $k$ be an arbitrary field. Then for any reductive group $G$ over $k$ and any $d\ge 1$
the natural maps
$$
G(k)/R \to G\bigl(k[[t_1,\dots, t_d]]\bigr)/R \to G\bigl(k((t_1,\dots, t_d))\bigr)/R
$$
 are isomorphisms.
\end{stheorem}
\begin{proof}
We set $A=k[[t_1,\dots, t_d]]$ and $K=k((t_1,\dots, t_d))$.
By Proposition~\ref{prop:hens-section}
we have the isomorphism
$G(k)/R \simlgr G(A)/R$. Corollary~\ref{cor:hens_surj}
shows that $G(A)/R \to G(K)/R$ is onto.
It remains  to prove that the surjective map $G(k)/R \to G(K)/R$ is an isomorphism
and we know that it holds in the one dimensional case,
 i.e. the map  $G(k)/R \to G\bigl( k((t)) \bigr)/R$
is an isomorphism \cite[Cor.\ 0.3]{Gi1}.
Using the embedding
$$K=k((t_1,\dots,  t_d)) \hookrightarrow k((t_1))\dots((t_d))$$
we get that $G(k)/R \to G(K)/R$ is injective.
\end{proof}

\begin{scorollary}\label{cor:cons}
Let $(A,m)$ be a complete regular local ring containing a
prime field $k_0$.
Let $k$ be the residue field of $A$, and let $K$ be the
fraction field of $A$. Let $\gG$ be a reductive group scheme over  $A$.
Then the maps $\gG(A) \to \gG(k)$
and $\gG(A) \subset \gG(K)$ induce two isomorphisms
$$
\gG(k)/R \xleftarrow{\ \sim\ }  \gG(A)/R  \xrightarrow{\ \sim\ } \gG(K)/R.
$$
\end{scorollary}
\begin{proof}
According to \cite[vol.\ 20, Thm.\ 19.6.4 page 102]{EGA4},
$A$ is $k_0$-isomorphic  (non-canonically)
to a formal series ring $k[[t_1,\dots ,t_d]]$, where $k$ is the residue field of $A$.
The group $\gG$ is the  twisted $A$--form of a Chevalley reductive group
$\ZZ$--scheme $\gG_0$ by a $\Aut(\gG_0)$--torsor $\gE$. Let $G=\gG\times_A k$ be the restriction of $\gG$
via the residue homomorphism $A\to k$.
Since
$$
H^1(\widehat{A}, \Aut(\gG_0)_k) \simlgr  H^1(k, \Aut(\gG_0)_k)
$$
\cite[XXIV.8.1]{SGA3}, it follows that
$\gG$ is isomorphic to ${G \times_k k[[t_1,\dots,  t_d]]}$. Then we can apply Theorem~\ref{thm:const}.
\end{proof}




\begin{stheorem}\label{thm:hens-inj}
Let $A$ be a henselian regular local ring containing a field $k_0$.
Let $k$ be the residue field of $A$ and let $K$ be the fraction field of $A$.
Let $\gG$ be a reductive group scheme over $A$.
Then the maps $\gG(A) \to \gG(k)$
and $\gG(A) \subset \gG(K)$ induce  two isomorphisms
$$
\gG(k)/R \xleftarrow{\ \sim\ }  \gG(A)/R  \xrightarrow{\ \sim\ } \gG(K)/R.
$$
In particular, we have a well-defined specialization
map $sp:\gG(K)/R \to \gG(k)/R$
and it is an isomorphism.
\end{stheorem}

\begin{proof}
By Lemma~\ref{lem:hens-section} $A$ is a filtered direct limit of henselian regular local rings
$A_i$ such that each $A_i$ contains a field
and the map from $A_i$ to its residue field admits a section.
Since the group scheme $\gG$ and its parabolic subgroups are finitely presented over $A$, and the functor $\gG(-)/R$
commutes with filtered direct limits, we can assume from the start that $A\to k$ admits a section.
Since $A$ is henselian, we have a bijection $H^1(A, \Aut(\gG_0)_k) \simlgr  H^1(k, \Aut(\gG_0)_k)$
\cite[XXIV.8.1]{SGA3}. Since $A\to k$ has a section,
it follows that
$\gG$ is isomorphic to $\gG_k \times_k A$.
Clearly, $\gG$ is isotropic if and only if $\gG_k$ is isotropic. Then
by Proposition~\ref{prop:hens-section} $\gG(A)/R\to \gG(k)/R$ is an isomorphism and
by Corollary~\ref{cor:hens_surj}
$\gG(A)/R\to\gG(K)/R$ is surjective.

Let $\hat A$ be the completion of $A$ at the maximal ideal and let $\hat K$ be its fraction field. Then $\hat A$
is a complete regular local ring containing $k_0$ and $k$ is its residue field. By Corollary~\ref{cor:cons}
the maps $\gG(\hat A)/R\to \gG(k)/R$ and $\gG(\hat A)/R\to \gG(\hat K)/R$ are isomorphisms.
Hence $\gG(A)/R\to \gG(\hat A)/R$ is an isomorphism, and consequently $\gG(A)/R\to \gG(K)/R$ is injective.
\end{proof}

\begin{scorollary}\label{cor:hens-inj}
Let $B$ be a regular local ring containing a
prime field $k_0$, let $L$ be the fraction field of $B$ and let $l$ be the
residue field of $B$.
Let $\hat B$ denote the completion of $B$ with respect to the maximal ideal, and let $\hat L$ denote the
fraction field of $\hat B$.
Let $\gG$ be a reductive group scheme over $B$. There is a well-defined specialization homomorphism
$sp: \gG(L)/R\to \gG(l)/R$, 
in the sense that it makes the following diagram commutative
\begin{equation}\label{eq:diag-sp-main}
\xymatrix{
\gG(l)/R & \ar[l]_\sim  \gG(\hat B)/R
\ar[r]^\sim &  \gG(\hat L)/R \\
 & \gG(B)/R \ar[u] \ar[r] \ar[ul] & \gG(L)/R \ar[u]
 \ar@{-->}[ull]_{{\qquad}_{\scriptstyle{sp}}}
}
\end{equation}
where the top horizontal maps are those of Theorem \ref{thm:hens-inj}.
\end{scorollary}

\begin{proof}
By Theorem~\ref{thm:hens-inj}
(or by Corollary~\ref{cor:cons})
the natural maps $\gG(\hat B)/R\to \gG(l)/R$ and $\gG(\hat B)/R\to \gG(\hat L)/R$ are isomorphisms.
The specialization map is the composition of the first isomorphism with the inverse of the second one
and the natural homomorphism $\gG(L)/R\to\gG(\hat L)/R$. The commutativity of the diagram is clear.
\end{proof}

\begin{sremarks}\label{rem_CT2}
{\rm
(a) Let $A$ be a henselian DVR containing a field, let $K$ be the fraction field of $A$ and $k$ be the residue field of $A$.
Let $\bG$ be a reductive group over $A$.
As mentioned in the beginning of section~\ref{subsec_DVR},  a specialization map $\varphi:\bG(K)/R\to \bG(k)/R$
was already constructed in~\cite{Gi1}. Since $\bG(A)/R\to\bG(K)/R$ is surjective by Lemma~\ref{lem_dvr2},
the commutativity of the diagram~\eqref{eq_DVR} implies that $\varphi$ is uniquely determined by its restriction to
the image of $\bG(A)/R$. Since this restriction is the canonical isomorphism $\bG(A)/R\to\bG(k)/R$, this map $\varphi$
coincides with the map $sp$ of Theorem~\ref{thm:hens-inj} and Corollary~\ref{cor:hens-inj}.
For the same reason, it coincides with the specialization map $sp_A:\gG(K)/R\to \gG(k)/R$
defined by Colliot-Th\'el\`ene, Harbater, Hartmann, Krashen, Parimala, and Suresh~\cite[Theorem\ A.10]{CTHHKPS}.
\smallskip \newline
\noindent (b)
If we relax the assumption and let $A$ be an aritrary DVR containing a field, then the specialization map
$sp:\gG(K)/R\to\gG(k)/R$ of Corollary~\ref{cor:hens-inj} also coincides
with the specialization map $sp_A:\gG(K)/R\to \gG(k)/R$
of~\cite[Theorem\ A.10]{CTHHKPS}, since both maps
coincide with the natural composition $\gG(K)/R\to\gG(\hat K)/R\xrightarrow{sp}\gG(k)/R$,
where $\hat K$ is the fraction field of the completion $\hat A$ of $A$.
\smallskip \newline
\noindent (c) Colliot-Th\'el\`ene, Harbater, Hartmann, Krashen, Parimala, and Suresh also construct
a specialization homomorphism for arbitrary regular local rings of dimension 2~\cite[Prop.\ A.12]{CTHHKPS},
as follows. Let $B$ be such a ring, and let $L$ and $l$ be the fraction and the residue fields of $B$. Let
$p$ be a regular height 1 prime ideal of $B$, so that
$B/p$ and $B_p$ are two DVRs. The specialization map $sp_B:\gG(L)/R\to\gG(l)/R$ is defined as the composition
$$
\gG(L)/R\xrightarrow{sp_{B_p}} \gG(K)/R\xrightarrow{sp_{B/p}} \gG(l)/R,
$$
where $K=B_p/pB_p$ is also the fraction field of $B/p$. It is proved
in~\cite[Theorem\ A.14]{CTHHKPS} that $sp_B$ is independent of the choice of $p$, and, moreover,
is functorial with respect to injective local homorphisms of 2-dimensional regular local rings.
Since the homomorphism $B\to\hat B$ from $B$ to its completion is of the latter kind, it follows that
$sp_B$ fits into the commutative diagram~\eqref{eq:diag-sp-main}
of Corollary~\ref{cor:hens-inj}, and hence coincides with
our map $sp$ defined by means of this diagram, as long as $sp_{\hat B}:\gG(\hat L)/R\to\gG(l)/R$
coincides with the natural isomorphism of the top row of~\eqref{eq:diag-sp-main}. In its turn, $sp_{\hat B}$ has to
coincide with this isomorphism
by~\cite[Prop. A.12 (b)]{CTHHKPS}.
}
\end{sremarks}


\section{Appendices}

\subsection{The big Bruhat cell is a principal open subscheme.}\label{appendix_cell}


For split groups and Borel subgroups, this statement goes back to Chevalley, see
\cite[lemma 4.5]{Bo}.

\begin{slemma}\label{lem_cell} Let $B$ be a ring and let $\gG$ be a reductive group $B$-scheme
equipped with a pair of opposite parabolic $B$--subgroups
$\gP^{\pm}$.
Then the big cell $\Omega$ of $\gG$ attached to $\gP$ and $\gP^{-}$
is a principal open subscheme of $\gG$. More precisely, there
exists $f \in B[\gG]$  such that $\Omega= \gG_f$
and $f$ can be chosen $\Aut(\gG,\gP, \gP^{-})$--invariant.
\end{slemma}

\begin{proof}
Without loss of generality, we can assume
 that $\gG$ is adjoint.
We can assume $B$ noetherian and connected so that
 $(\gG, \gP, \gP^{-})$ is a $B$--form of
 $(\gG_0, \gP_0, \gP_0^{-})_B$ where  $\gG_0$ is an adjoint Chevalley
 $\ZZ$--group scheme equipped with opposite parabolic $\ZZ$--group subschemes
 $(\gP_0, \gP_0^{-})$ related to the Chevalley pinning.

 Then  $(\gG, \gP, \gP^{-})$ is the twist of   $(\gG_0, \gP_0, \gP_0^{-})_B$
 by an $\Aut(\gG_0,\gP_0, \gP^{-}_0)$-torsor so that the statement boils
 down to the split case over $\ZZ$. We consider the Levi--subgroup
 $\gL_0 = \gP_0^+ \cap \gP_0^{-}$ so that
 $\Aut(\gG_0,\gP_0, \gP^{-}_0) = \Aut(\gG_0, \gP_0, \gL_0)$
 is the semi-direct product of $\gL_{0}$ and a finite constant $\ZZ$--group
 scheme $\Gamma$ \cite[lemme 5.1.2]{Gi4}.

 According to \cite[3.8.2.(a)]{BT2}
 there is a function $f_0 \in \ZZ[\gG_0]$
 such that
$\ZZ[\Omega_0]=\ZZ[\gG_0]_{f_0}$ and satisfying $f_0(1)=1$.
We claim that $f_0$ is $\gL_0$-invariant with respect to the
adjoint action.
We denote by $\Lambda= \Hom_{\ol{\QQ}-gr} (\gL_{\ol{\QQ}}, \GG_m)$
the lattice of characters and remind the reader of Rosenlicht decomposition \cite[Th.\ 3]{Ro}
 $$
 H^0(\gL_{\ol{\QQ}}, \GG_m) = \ol{\QQ}^\times \oplus  \Lambda
$$
which shows that $\Lambda = \bigl\{ f \in H^0(\gL_{\ol{\QQ}}, \GG_m) \, \mid \, f(1)=1 \bigr\}$.
We observe that
the induced action (by the adjoint action) of $\gL_0(\ol{\QQ})$ on $\Lambda$
is trivial. It follows that the map
\[
\phi: \gL_0(\ol{\QQ}) \to \ol{\QQ}[\gL_0]^\times \to  \Lambda, \enskip
 x \mapsto   {^x\!f}_0 \,  f_0^{-1}
\]
 is a group  homomorphism. Since
$\gL_{0, \ol{\QQ}}$ is generated by  its maximal tori, we have
 $\gL_0(\ol{\QQ}) = \langle \gL_0(\ol{\QQ})^n \rangle$ for all
 $n \geq 1$. We get that $\phi$ is zero and  this establishes the above claim.
 Taking the product of $\Gamma$-conjugates of $f_0$
 permits to assume that $f_0$ is   $\Aut(\gG_0, \gP_0, \gL_0)$-invariant.
 By descent, $f_0$ gives rise to then to $f \in B[\gG]$
 so that  $\Omega= \gG_{f}$.
\end{proof}

\subsection{Colliot-Th\'el\`ene and Ojanguren method for functors in pointed sets.}
\label{appendix_cto}
In this section we summarize the classic injectivity theorem of Colliot-Th\'el\`ene and Ojanguren~\cite[Th.\ 1.1]{CTO}.
Our goal is to make explicit the fact that a certain intermediate step in the proof of this theorem holds under weaker
assumptions than the theorem itself.

Let $k$ be an infinite field and let $R \mapsto F(R)$ be a covariant functor on the category
of $k$--algebras (commutative, unital) with values in pointed  sets. We consider the following properties:

\medskip

$({\bf P}_1)$ The functor $F$ commutes with filtered direct limits of $k$-algebras having flat transition morphisms.

\medskip

$({\bf P}_2)$ For each $k$--field $E$ and for each $n \geq 1$, the map
$$
F\bigl( E[t_1, \dots, t_n] \bigr) \to F\bigl( E(t_1, \dots, t_n) \bigr)
$$ has trivial kernel;

\medskip

$({\bf P}_3)$ (Patching property) For each finite type flat inclusion $A \hookrightarrow B$ of noetherian integral $k$--algebras  and each non-zero element $f \in A$ such that $A/fA \simlgr B/fB$, then the map
$$
\Ker\bigl( F(A) \to F(A_f)\bigr) \to \Ker\bigl( F(B) \to F(B_f)\bigr)
$$
is onto.

\medskip

One may consider the following  weaker property.

\medskip

$({\bf P}'_3)$ (Zariski patching) For each noetherian integral $k$--algebra $A$
and for each decomposition $A= Af + Ag$ with $f$ non-zero,
then the map
$$
\Ker\bigl( F(A) \to F(A_f)\bigr)  \to  \Ker\bigl( F(A_g) \to F(A_{fg} )\bigr)
$$
is onto.

\medskip

We have ${\bf P}_3 \Longrightarrow {\bf P}'_3$ by  taking  $B=A_g$ since  we have $B_f= A_{fg}$ and
$A/fA \simlgr B/fB$.

The following theorem was proved by Colliot-Th\'el\`ene and Ojanguren.

\begin{stheorem} \label{thm_cto} \cite[Th.\ 1.1]{CTO} We assume that $F$ satisfies ${\bf P}_1$, ${\bf P}_2$ and ${\bf P}_3$.
Let $A$ be a local ring of a smooth $L$--ring $C$  where $L$ is a $k$--field. Denote by $K$ the fraction field
of $A$. Then the map $
F\bigl( A \bigr) \to F\bigl( K\bigr)
$
 has trivial kernel.
\end{stheorem}

The proof of this theorem relies on the following result.

\begin{sproposition}\label{prop_cto}~\cite[prop.\ 1.5]{CTO}
We assume that $F$ satisfies ${\bf P}_1$, ${\bf P}_2$ and ${\bf P}'_3$.  Let $A$ be the local ring at a prime ideal
of a polynomial algebra $L[t_1, \dots, t_d]$ where $L$ is a $k$--field. Denote by $K$ the fraction field
of $A$. Then for each integer $n \geq 0$, the map
$$
F\bigl( A[x_1, \dots, x_n] \bigr) \to F\bigl( K(x_1, \dots, x_n)\bigr)
$$
 has trivial kernel.
\end{sproposition}
\begin{proof}
The original statement of~\cite[Prop.\ 1.5]{CTO} assumes that $F$ satisfies
${\bf P}_1$, ${\bf P}_2$ and ${\bf P}_3$, and that $A$ is a maximal localization of $L[t_1, \dots, t_d]$.
The inspection of the proof shows that instead of property ${\bf P}_3$, only the Zariski patching property
${\bf P}'_3$ was used. Furthermore, since every prime ideal of $L[t_1, \dots, t_d]$ is an intersection of maximal
ideals, and $F$ satisfies ${\bf P}_1$, the case where $A$ is a localization at a prime ideal follows from the case of maximal
localizations~\cite[p. 101, Premi\`ere r\'eduction]{CTO}.
\end{proof}

\bigskip

\subsection{Fields of representatives for henselian regular local rings}

The following fact was brought to our attention by K. \v{C}esnavi\v{c}ius.

\begin{slemma}\label{lem:hens-section}
Let $A$ be a henselian local ring containing a
prime field $k_0$. Then
$A$ is a filtered direct limit
of henselian local rings $A_i$ such that the map from $A_i$ to its residue field admits a section.
If $A$ is moreover regular, then the henselian local rings $A_i$ can be chosen regular as well.
\end{slemma}
\begin{proof}
The local ring $A$ is a filtered direct limit of local rings $C_i$ that are localizations of finitely generated $k_0$-algebras
contained in $A$. Since $A$ is henselian, we can replace each $C_i$ by its henselization
$A_i=(C_i)^h$. Let  $k_i=A_i/m_i$ be the residue field of $A_i$. Then $k_i$ is a finitely generated field extension of $k_0$.
We claim that $A_i\to k_i$ admits a section. Indeed, since $k_0$ is perfect,
it follows that $k_i$ is separably generated over $k_0$, that is,
$k_i$ is a finite separable extension of a purely transcendental field extension
$L=k_0(t_1,\ldots,t_n)$ of $k_0$ of finite
transcendence degree~\cite[II, \S 13, Theorem 31]{ZaSa-I}.
Choose arbitrary
lifts $a_1,\ldots,a_n$ of $t_1,\ldots,t_n$ to $A_i$. Then $k_0(a_1,\ldots,a_n)\cong L$ is a subfield of $A_i$ that lifts $L$.
By the primitive element theorem $k_i=L[b]=L/P(t)$ where $P$ is a separable $L$--polynomial. Since $A$
is henselian, $P(t)$ has a root
$a \in A$ which lifts $b \in L$. We define then a $L$--map
$k_i=L[b] \to A$ by mapping $b$ to $a$.
The composite map $k_i \to  A \to k_i=L[b]$ is the identity as desired.

If $A$ is a regular henselian local ring, note that the embedding $k_0\to A$ is geometrically regular,
since $k_0$ perfect~\cite[(28.M), (28.N)]{Mats}.
Then by Popescu's theorem~\cite{Po90,Swan} $A$ is a filtered direct limit of localizations $C_i$ of smooth
$k_0$-algebras. Then the henselizations $A_i$ are also regular.
\end{proof}

\bigskip

\medskip

\end{document}